\documentclass[twoside,11pt]{article}

\usepackage{blindtext}

\usepackage{jmlr2e}

\usepackage{lastpage}
\ShortHeadings{Optimal Confidence Band for Kernel Gradient
Flow Estimator}{Yuqian Cheng, Zhuo Chen, and Qian Lin}
\firstpageno{1}

\usepackage{amsmath}
\usepackage{amsfonts}
\usepackage{mathrsfs}
\usepackage{enumerate}
\usepackage{amssymb}
\usepackage{graphicx}
\usepackage{subfigure}
\usepackage{booktabs}

\newtheorem{assumption}{Assumption}

\newcommand\R{\mathbb{R}}
\newcommand\C{\mathbb{C}}
\newcommand\tr{\mathrm{tr}}
\newcommand\E{\mathrm{E}}       %mathematical expectation
\newcommand\error{\mathcal{E}}  %error
\renewcommand\O{\mathcal{O}}    %big O notation
\renewcommand\H{\mathcal{H}}    %RKHS
\renewcommand\P{\mathbb{P}}     %probability
\newcommand\X{\mathcal{X}}      %input space
\newcommand\Y{\mathcal{Y}}      %output space
\newcommand\D{\mathbb{D}}       %data set
\newcommand\K{\mathbb{K}}       %matrix notation
\newcommand\F{\mathcal{F}}      %Fourier transform

\newcommand\qed{\hfill $\blacksquare$}

\graphicspath{{pic/}}

\begin{document}

\title{Optimal Confidence Band for Kernel Gradient Flow Estimator}

\author{\name Yuqian Cheng \email yuqian.cheng.1999@gmail.com 
       \AND
       \name Zhuo Chen \email chenzhuo@tsinghua.edu.cn \\
       \addr Department of Mathematical Sciences\\
       Tsinghua University\\
       Beijing, 100084, China
       \AND
       \name Qian Lin \email qianlin@tsinghua.edu.cn \\
       \addr Department of Statistics and Data Science\\
       Tsinghua University\\
       Beijing, 100084, China}

\editor{TBD}

\thispagestyle{empty}

\maketitle

\begin{abstract}%
    In this paper, we investigate the supremum-norm generalization error and the uniform inference for a specific class of kernel regression methods, namely the kernel gradient flows. Under the widely adopted capacity-source condition framework in the kernel regression literature, we first establish convergence rates for the supremum norm generalization error of both continuous and discrete kernel gradient flows under the source condition $s>\alpha_0$, where $\alpha_0\in(0,1)$ denotes the embedding index of the kernel function. Moreover, we show that these rates match the minimax optimal rates. Building on this result, we then construct simultaneous confidence bands for both continuous and discrete kernel gradient flows. Notably, the widths of the proposed confidence bands are also optimal, in the sense that their shrinkage rates are greater than, while can be arbitrarily close to, the minimax optimal rates.
\end{abstract}


\section{Introduction} 

In recent years, a family of non-parametric regression methods, known as the kernel regression methods, has attracted considerable attention. The central idea of these methods is to estimate the target function within the reproducing kernel Hilbert space (RKHS) associated with a preselected kernel function. In order to accommodate the diverse characteristics of real-world data, various forms of kernel regression have been developed, including kernel ridge regression \citep{mathematical_foundations}, kernel gradient flow \citep{stopping_time_gd}, and kernel spectral cut-off methods \citep{optimal_rates_spectral_group}. These approaches can be further tailored through different choices of kernels, such as Gaussian kernels \citep{gaussian_process_book}, Matérn kernels \citep{matern}, neural tangent kernels \citep{ntk}, and even data-dependent kernels \citep{data_dependent_kernel}. This flexibility highlights the potential of kernel regression as a powerful and versatile framework for modeling complex input–output relationships.

In this paper, we focus on a particular class of kernel regression methods known as kernel gradient flow, which have attracted considerable attention in recent years due to their strong theoretical properties and wide applicability. This method can be viewed as a natural extension of gradient flow methods in convex optimization \citep{convex}. Moreover, in the context of artificial intelligence and deep learning, the neural tangent kernel theory \citep{ntk} provides an interpretation of kernel gradient flow as the training dynamics of sufficiently wide neural networks.

The existing literature on kernel gradient flow has primarily focused on the mean squared generalization error. In contrast, the supremum-norm generalization error and the statistical inference theory for kernel gradient flow remains relatively limited. Among the most closely related works, \citet{sgd_inference} derived an estimate of the supremum-norm generalization error for the stochastic kernel gradient descent estimator, a common variant of kernel gradient flow, and employed an online multiplier bootstrap procedure to construct asymptotically exact confidence bands for the proposed estimator. However, several important questions still remaining to be solved. First, \citet{sgd_inference} only established a loose upper bound for the supremum-norm generalization error while did not prove the minimax optimality, hence they did not prove that the proposed confidence band attains the optimal band width as well; Second, they only considers the source condition $s=1$ (that is, the target function lies in the RKHS); Moreover, they only focused on RKHS with uniformly bounded and Lipschitz eigenfunctions, which limits the range of kernel functions to which their theory applies.

We list our contributions as follows: 

\textbf{(I) Convergence rate of the supremum-norm generalization error of kernel gradient flow.}

Under the capacity-source condition framework (see Assumptions \ref{edr} and \ref{source condition}) and other additional mild conditions, we derive a convergence rate of the supremum-norm generalization error for both continuous and discrete kernel gradient flow estimators (defined in Definitions \ref{continuous kgf definition} and \ref{discrete kgd definition}) under general source conditions $s>\alpha_0$, where $\alpha_0<1$ is the embedding index of the kernel function (see Assumption \ref{embedding index assumption}). 

Moreover, for the optimal selection of the training time of the kernel gradient flow estimator, the best convergence rate in fact (nearly) matches the minimax optimal rate (Corollary \ref{best upper bound} and Theorem \ref{lower bound}), which demonstrates the minimax optimality of kernel gradient flow estimators. 

\textbf{(II) Simultaneous confidence band for kernel gradient flow.} 

Based on the results of the supremum-norm generalization error convergence rate above, we construct simultaneous confidence bands for both continuous and discrete kernel gradient flow estimators, and establish Theorem \ref{main result} which guarantees the asymptotic exactness of the proposed confidence bands.

It is worth noting that the confidence bands we construct are optimal, in the sense that their widths shrinks at nearly minimax optimal rate (see Remark \ref{optimal shrink rate}). 

The construction of the confidence bands is technically built on the theory of Gaussian approximation (see Section \ref{proof sketch section} and Appendix \ref{gaussian approximation section} for details).

\subsection{Related works} 

Prior works such as \citet{stopping_time_gd, optimality_gd} have studied the generalization ability of kernel gradient flow or its discretization with respect to mean square error, and showed that the kernel gradient flow is minimax optimal in some cases. The generalization abiliby of the kernel interpolation, which is the limit of kernel gradient flow as the training time $t$ goes to infinity, has also been studied in \citet{interpolation2, interpolation}. The kernel gradient flow is a special case of spectral algorithm \citet{spectral1, spectral2}. The techniques of spectral algorithms are helpful for the computation involving the complicated exponential form of kernel gradient flow estimator. The generalization ability of general spectral algorithms has been studied in \citet{optimality_spectral1, optimality_spectral2, optimal_rates_spectral_group}. \citet{optimal_rates_spectral} has shown the optimality of spectral algorithm on high-dimensional Hilbert spaces. \citet{analytic} has studied the analytic version of spectral algorithm. 

Researches on the supremum-norm generalization ability of kernel gradient flow has been limited. \citet{sgd_inference} has provided a non-tight convergence rate of supremum-norm generalization error for stochastic kernel gradient flow under source condition $s=1$. For other kinds of kernel regression methods, \citet{krr_sup_norm_well_specified} has focused on the supremum-norm of generalization ability of kernel ridge regression under source condition $s=1$; \citet{krr_sup_norm_high_dim, krr_sup_norm_high_dim2, krr_sup_norm_high_dim3} have studied the supremum-norm generalization ability of kernel ridge regression in high-dimensional cases; \citet{krr_inference} has proved a non-tight convergence rate of supremum-norm generalization error for kernel ridge regression. 

For several canonical function classes, there has been extensive work on establishing minimax optimal rates with respect to supremum distance along with other distances. For example, for density estimation, the supremum-norm minimax optimal rate for $1$-dimensional Sobolev spaces was obtained in \citep{sobolev_1d_minimax} and \citep{sobolev_1d_minimax2}; the supremum-norm minimax optimal rate for H\"older spaces was established in \citep{holder_minimax}; the Wasserstein distance minimax optimal rate for Besov spaces was established in \citep{besov_minimax}. On the other hand, although Sobolev spaces are typical representatives of interpolation spaces of RKHS \citep{sobolev_rkhs}, studies on the minimax optimal rate for an arbitrary RKHS and its interpolation spaces have been limited. Recently, \citet{duality} derived a minimax lower bound for general RKHS, while this work did not prove that the proposed minimax lower bound rate matches the minimax optimal rate, and was based on a strong condition on the embedding index of the kernel function (see Remark \ref{duality limit discussion} for further discussion). 

Our inference theory for kernel gradient flow is based on the tools of Gaussian approximation and multiplier bootstrap algorithm, which are developed by a series of prior works, including \citet{anti-concentration_a, anti-concentration_b, vc-type, anti-concentration_d, anti-concentration_c}. Based on the tools of Gaussian estimation, \citet{krr_inference} has constructed a confidence band for kernel ridge regression estimator, and \citet{sgd_inference} has constructed a confidence band for stochastic kernel gradient descent estimator. Similarly with \citet{krr_inference}, \citet{sgd_inference} only focused on kernels with bounded Lipschitz eigenfunctions, and did not obtain the confidence band of optimal width. On the other hand, \citet{krr_inference2} has established a bootstrap approximation for kernel ridge regression based on the Gaussian approximation of RKHS norm for empirical process, which cannot be applied to the mis-specified cases. The Gaussian estimation and bootstrap approximation in high-dimensional cases have been studied by a series of prior works, including \citet{gaussian_approximation_high_dim1, anti-concentration_b, gaussian_approximation_high_dim3, gaussian_approximation_high_dim4, gaussian_approximation_high_dim5}.

\section{Background}

\subsection{The regression problem} 

Assume that the input space $\X$ is a compact subspace of $\R^d$, and the output space is $\Y\subset\R$. Let $\rho$ be an unknown probability distribution on $\X\times\Y$, and suppose that we have $n$ i.i.d. samples $\D_n=\{(x_i,y_i):\,i=1,\dots,n\}$ drawn from $\rho$. Our goal is to find a function $\hat{f}$ such that 
\begin{equation}
    \error(\hat{f})=\frac{1}{2}\int_{\X\times\Y}(\hat{f}(x)-y)^2d\rho
\end{equation}
is sufficiently small. By elementary computation, it is equivalent to searching for a function $\hat{f}$ such that the generalization error 
\begin{equation}\label{generalization error definition}
    \frac{1}{2}\int_\X(\hat{f}(x)-f^*(x))^2d\mu(x)
\end{equation} 
is sufficiently small, where $\mu$ is the marginal distribution of $\rho$ on $\X$, and the function $f^*$ is defined by 
\begin{equation}\label{regression function definition} 
    f^*(x)=\E_\rho(y|x)=\int_\Y yd\rho(y|x)
\end{equation} 
It is easy to see that $f^*$ is the minimizer of (\ref{generalization error definition}). The function $f^*$ is called the regression function or the true function, and $\hat{f}$ can be viewed as an estimator of $f^*$. 

Basically, we assume the following assumption on the noise $\varepsilon=y-f^*(x)$ of the model: 
\begin{assumption}\label{noise moment} 
    (noise moment bound) Assume that the noise $\varepsilon$ satisfies: for a.s. $x\in\X$, we have $\E[\varepsilon|x]=0$, $\E[\varepsilon^2|x]\leq\sigma^2<\infty$, and 
    \begin{equation}\E[|\varepsilon|^m|x]\leq\frac{1}{2}m!\sigma^2 L^{m-2}\end{equation} 
    for any integer $m>2$. 
\end{assumption} 
This is a standard assumption on noise. Gaussian noise and sub-Gaussian noise are common examples that satisfy this assumption.

\subsection{Reproducing kernel Hilbert space}

In the kernel gradient flow algorithm, we search for the estimator of the regression function $f^*$ in the reproducing kernel Hilbert space (RKHS) of a positive kernel function $k:\,\X\times\X\to\R$. In this subsection, we recall some critical details about kernel functions and their RKHS. 

Suppose that the kernel function $k$ is bounded in the following sense: 
\begin{equation}\label{bound of kernel}
    \sup_{x\in\X}k(x,x)\leq\kappa^2
\end{equation} 
for some universal constant $\kappa>0$. The positiveness of the kernel $k$ means that the integral operator 
\begin{equation}\label{integral operator definition}
    T:\,L^2(\X)\to L^2(\X),\quad Tf(x)=\int_\X k(x,\xi)f(\xi)d\xi
\end{equation} 
has eigenvalues $\lambda_1\geq\lambda_2\geq\dots>0$. Let $e_i$ be the corresponding eigenfunction of $\lambda_i$ such that $e_1,e_2,\dots$ form an orthonormal basis of $L^2(\X)$. Then, by Mercer decomposition \citep{mercer}, we have 
\begin{equation}k(x,x')=\sum_{i=1}^\infty\lambda_ie_i(x)e_i(x'),\end{equation} 
and this expansion converges absolutely and uniformly. 

Denote $k_x(\cdot)=k(x,\cdot)$, and define an inner product on the linear space $\mathrm{span}\{k_x:\,x\in\X\}$ by $\langle k_x,k_{x'}\rangle_\H=k(x,x')$. The closure of $\mathrm{span}\{k_x:\,x\in\X\}$ under this inner product is called the reproducing kernel Hilbert space (RKHS) of $k$, denoted as $\H$. Accordingly, $k$ is called the reproducing kernel of $\H$. The names ``RKHS'' and ``reproducing kernel'' come from the following property, known as the reproducing property: 
\begin{equation}\label{reproducing property}
    \langle k_x,f\rangle_\H=f(x),\quad\forall x\in\X,\,f\in\H.
\end{equation} 

For any function $f\in\H$, it has a unique RKHS expansion: 
\begin{equation}\label{RKHS expansion definition} 
    f(x)=\sum_{i=1}^\infty f_i\sqrt{\lambda_i}e_i(x),\quad \{f_i\}_{i=1}^\infty\in l^2,
\end{equation}
and the inner product $\langle\cdot,\cdot\rangle_\H$ can be computed as $\langle f,g\rangle_\H=\sum_{i=1}^\infty f_ig_i$ for $f=\sum_{i=1}^\infty f_i\sqrt{\lambda_i}e_i\in\H$ and $g=\sum_{i=1}^\infty g_i\sqrt{\lambda_i}e_i\in\H$ (see \citep{mathematical_foundations} for example). Thus, $\{\sqrt{\lambda_i}e_i\}_{i=1}^\infty$ forms an orthonormal basis of $\H$. 

We introduce the following basic assumptions on the kernel function: 
\begin{assumption}\label{edr} 
    (Eigenvalue decay rate) Assume that 
    \begin{equation}ci^{-\beta}\leq\lambda_i\leq Ci^{-\beta}\end{equation} 
    for some positive constants $c$ and $C$. 
\end{assumption} 

This assumption is also called the capacity condition or the effective dimension condition \citep{optimal_rates_krr}. It is satisfied by various kinds of kernels such as Sobolev kernels \citep{sobolev_rkhs}, Mat\'ern kernels on spheres \citep{discrepancies}, neural tangent kernels \citep{ntk_spectrum}, etc. 

\begin{assumption}\label{Holder assumption} 
    (H\"older continuity) Assume that the kernel function $k$ is $2h$-H\"older for some $h\in(0,\frac{1}{2}]$. In other words, there exists a constant $L_k>0$ such that 
    \begin{equation}|k(x_1,x_2)-k(x_1',x_2')|\leq L_k|(x_1,x_2)-(x_1',x_2')|^{2h},\quad\forall x_1,x_2,x_1',x_2'\in\X.\end{equation} 
\end{assumption}

The H\"older continuity of kernel function is technically required in the proofs of our main theorems, especially in those parts where we extend pointwise estimations to uniform estimations.

\subsection{Interpolation spaces} 

In order to characterize the relative smoothness of the regression function $f^*$ with respect to $\H$, many prior works \citep{optimal_rates_krr, optimal_rates_spectral,NTK_generalization_R,optimal_rates_krr_group} usually assume that $f^*$ lies in $[\H]^s$, the interpolation space of $\H$ of order $s$, which is defined by 
\begin{equation}\label{interpolation space definition} 
    [\H]^s=\mathrm{Ran}T^{\frac{s}{2}}=\left\{\sum_{i=1}^\infty f_i\lambda_i^{\frac{s}{2}}e_i:\,\{f_i\}_{i=1}^\infty\in l^2\right\},
\end{equation}
where $T^{\frac{s}{2}}$ is the $s$-power of the integral operator $T$ defined by 
\begin{equation} 
    T^{\frac{s}{2}}:\,L^2(\X)\to L^2(\X),\quad T^{\frac{s}{2}} f=\sum_{i=1}^\infty\lambda_i^\frac{s}{2}\langle f,e_i\rangle_{L^2}e_i.
\end{equation} 
$[\H]^s$ is naturally equipped with an inner product: $\langle f,g\rangle_{[\H]^s}=\sum_{i=1}^\infty f_ig_i$ for $f=\sum_{i=1}^\infty f_i\lambda_i^{\frac{s}{2}}e_i\in[\H]^s$ and $g=\sum_{i=1}^\infty g_i\lambda_i^{\frac{s}{2}}e_i\in[\H]^s$. Note that $[\H]^0$ coincides with $L^2(\X)$, and $[\H]^1$ coincides with $\H$. For $0<s_1<s_2$, the embeddings $[\H]^{s_2}\hookrightarrow[\H]^{s_1}\hookrightarrow[\H]^0$ exist and are compact. 

\begin{assumption}\label{source condition} 
    (Source condition) Assume that for some $s>\alpha_0$, there exists a constant $R>0$ such that $f^*\in[\H]^s$ and 
    \begin{equation}\|f^*\|_{[\H]^s}\leq R.\end{equation} 
\end{assumption} 
Here, $s$ is called the source condition of $f^*$. Assumptions \ref{edr} and \ref{source condition} are jointly referred to as the capacity–source condition framework, which constitutes a standard setting in the study of kernel regression methods \citep{optimal_rates_krr,optimal_rates_spectral,optimality_spectral1,optimal_rates_krr_group,optimal_rates_spectral_group}. The source condition $s$ is commonly interpreted as the relative smoothness of the true function $f^*$ with respect to the RKHS $\H$. 

One of the important characterizations of RKHS is the embedding property of its interpolation spaces. For any $\alpha>0$, define the embedding coefficient $M_\alpha$ of order $\alpha$ by 
\begin{equation}\label{embedding coefficient def}
    M_\alpha^2=\sup_{x\in\X}\sum_{i=1}^\infty\lambda_i^\alpha e_i(x)^2\in[0,\infty].
\end{equation} 
It is clear that $M_\alpha$ is nonincreasing in $\alpha$. If $M_\alpha < \infty$, then by Cauchy’s inequality, it follows that for any $f \in [\H]^\alpha$, $f\in C^0(\X)$ and 
\begin{equation}\|f\|_\infty=\sup_{x\in\X}|f(x)|\leq M_\alpha\|f\|_{[\H]^\alpha},\end{equation} 
which implies that $[\H]^\alpha$ is naturally embedded into $C^0(\X)$. Therfore, we say that the RKHS $\mathcal{H}$ satisfies the embedding property of order $\alpha$ if $M_\alpha < \infty$. Finally, we define the \textit{embedding index} of $\H$ by 
\begin{equation}\label{embedding index def} 
    \alpha_0=\inf\left\{\alpha\in[\frac{1}{\beta},1]:\,M_\alpha<\infty\right\}.
\end{equation} 

We make the following assumption on the embedding index of the kernel function: 
\begin{assumption}\label{embedding index assumption} 
    (Embedding index) Assume that the embedding index of $\H$ is $\alpha_0=\frac{1}{\beta}$, where $\beta$ is the eigenvalue decay rate in assumption \ref{edr}, and the embedding index $\alpha_0$ is defined in (\ref{embedding index def}). 
\end{assumption} 
This assumption is satisfied by most commonly used kernels, including Sobolev kernels, inner-product kernels on spheres, and periodic translation-invariant kernels \citep{optimal_rates_spectral_group}. 

For more details about RKHS and its interpolation spaces, we refer to \citet{mercer} and \citet{mathematical_foundations} for example.

\subsection{Kernel gradient flow estimators} 

Recall that the mean square generalization error $\error(\hat{f})$ is defined by (\ref{generalization error definition}). The corresponding empirical mean square error function is 
\begin{equation}\hat{\error}(\hat{f})=\frac{1}{2}\sum_{j=1}^n(\hat{f}(x_j)-y_j)^2,\end{equation}

In the kernel gradient flow algorithm, the estimator $\hat{f}$ is set to be in the following form: $\hat{f}=\hat{f}_t$ (with parameter $t\in[0,\infty)$) lies in the RKHS $\H$ with coefficients $\{a_i(t)\}_{i=1}^\infty$; in other words, 
\begin{equation}
    \hat{f}(x)=\hat{f}_t(x)=\sum_{i=1}^\infty a_i(t)\sqrt{\lambda_i}e_i(x).
\end{equation} 
The parameter $t\geq0$ is called the training time or stopping time of the kernel gradient flow estimator. 

In this paper, we consider two kinds of kernel gradient flow estimator: the continuous kernel gradient flow estimator $\hat{f}_t=\hat{f}_t^{con}$ and the discrete kernel gradient flow estimator $\hat{f}_t=\hat{f}_t^{dis}$. 

\subsubsection{The continuous kernel gradient flow} In this case, each parameter $a_i(t)$ is initialized as $a_i(0)=0$, and is then set to evolve along the following gradient flow: 
\begin{equation}\label{gradient flow of parameters} 
    \frac{d}{dt}a_i(t)=-\frac{\partial}{\partial a_i(t)}\hat{\error}(\hat{f}_t)=-\sum_{j=1}^n\frac{\partial\hat{f}_t(x_i)}{\partial a_i(t)}\cdot(\hat{f}_t(x_j)-y_j).
\end{equation}
Equivalently, the evolution equation of the estimator $\hat{f}_t$ is given by 
\begin{equation}\label{gradient flow of estimator} 
    \begin{aligned}
        \frac{d}{dt}\hat{f}_t(x)&=\sum_{i=1}^\infty\frac{\partial\hat{f}_t(x)}{\partial a_i(t)}\cdot\frac{d a_i(t)}{dt}=-\sum_{j=1}^n\sum_{i=1}^\infty\frac{\partial\hat{f}_t(x)}{\partial a_i(t)}\cdot\frac{\partial\hat{f}_t(x_j)}{\partial a_i(t)}\cdot(\hat{f}_t(x_j)-y_j)\\
        &=-\sum_{j=1}^n k(x,x_j)(\hat{f}_t(x_j)-y_j).
    \end{aligned}
\end{equation} 

\begin{definition}\label{continuous kgf definition}
    \textnormal{\textbf{(Continuous kernel gradient flow estimator)}} Let $\hat{f}_t^{con}(x)$ be the solution to the equation (\ref{gradient flow of estimator}) with initial condition $\hat{f}_0^{con}(x)\equiv 0$. $\hat{f}_t^{con}(x)$ is called the continuous kernel gradient flow estimator. 
\end{definition}

Note that the continuous kernel gradient estimator $\hat{f}_t^{con}$ admits an explicit expression: 
\begin{equation}\label{solution to gradient flow of estimator}
    \hat{f}_t^{con}(x)=\K(x,X)\K(X,X)^{-1}\left(I_n-\exp(-\frac{t}{n}\K(X,X))\right)Y,
\end{equation} 
where $\K(x,X)=(k(x,x_1),\dots,k(x,x_n))$, $\K(X,X)=(k(x_i,x_j))_{n\times n}$ and $X=(x_1,\dots,x_n)^T$, $Y=(y_1,\dots,y_n)^T$, providing that $\K(X,X)$ is invertible. 

\subsubsection{The discrete kernel gradient flow} 

We also consider the discretized counterpart of the evolution equation (\ref{gradient flow of estimator}): 
\begin{equation}\label{discrete gradient flow of estimator}
    \hat{f}_{t_{m+1}}(x)=\hat{f}_{t_m}(x)-\eta\cdot\frac{1}{n}\sum_{j=1}^nk(x,x_j)(\hat{f}_{t_m}(x_j)-y_j),
\end{equation} 
where $t_0=0$, $t_{m}=t_{m-1}+\eta$ for $m\geq 0$, and $\eta>0$ is a preselected parameter. 

\begin{definition}\label{discrete kgd definition}
    \textnormal{\textbf{(Discrete gradient flow estimator)}} Let $\hat{f}_{t}^{dis}(x)$ ($t=t_m=m\eta$, $m=0,1,\dots$) be the solution to (\ref{discrete gradient flow of estimator}) with initial condition $\hat{f}_0^{dis}(x)\equiv 0$. $\hat{f}_{t}^{dis}(x)$ is called the discrete kernel gradient flow estimator or the kernel gradient descent estimator, and the parameter $\eta$ is called the learning rate of $\hat{f}_t^{dis}(x)$. 
\end{definition}  

~

Both the continuous kernel gradient flow estimator $\hat{f}_t^{con}$ and the discrete kernel gradient flow estimator $\hat{f}_t^{dis}$ can be expressed in the form of spectral algorithm \citep{spectral1, spectral2}. Consider the following empirical version of the integral operator $T$ defined in (\ref{integral operator definition}): 
\begin{equation}T_X:\,\H\to\H,\quad T_X f(\cdot)=\frac{1}{n}\sum_{j=1}^n f(x_j)k(x_j,\cdot),\end{equation} 
and define the sample basis function as 
\begin{equation}\hat{g}(\cdot)=\frac{1}{n}\sum_{j=1}^n y_j k(x_j,\cdot).\end{equation} 

Then, $\hat{f}_t^{con}$ and $\hat{f}_t^{dis}$ can be represented by 
\begin{equation}
    \hat{f}_t^{con}=\varphi_t^{con}(T_X)\hat{g},\quad\hat{f}_{t}^{dis}=\varphi_{t}^{dis}(T_X)\hat{g},
\end{equation} 
(we refer to \citealt{optimal_rates_spectral_group, analytic} for the details), where the functions $\varphi_t^{con}$ and $\varphi_t^{dis}$ are the filter functions of continuous and discrete kernel gradient flows, respectively, and are defined as follows: 
\begin{definition}\label{filter function def}
    We define the filter function $\varphi_t(z)$ and the remainder function $\psi_t(z)=1-z\varphi_t(z)$ of the kernel gradient flow as follows: 
    
    (I) For the continuous kernel gradient flow, we define 
    \begin{equation}\varphi_t(z)=\varphi_t^{con}(z):=\frac{1-e^{-tz}}{z},
    \end{equation} 
    \begin{equation}\psi_t(z)=\psi_t^{con}(z)=1-z\varphi_t^{con}(z):=e^{-tz};\end{equation} 
    (II) For the discrete kernel gradient flow, we define 
    \begin{equation}\label{discrete gd filter function}
        \varphi_t(z)=\varphi_t^{dis}(z):=\frac{1-(1-\eta z)^{t/\eta}}{z},
    \end{equation} 
    \begin{equation}
        \psi_t(z)=\psi_t^{dis}(z)=1-z\varphi_t^{dis}(z):=(1-\eta z)^{t/\eta}.
    \end{equation} 
\end{definition} 

For the discrete kernel gradient flow, we need an additional assumption on its learning rate $\eta$: 
\begin{assumption}\label{learning rate assumption} 
    The learning rate $\eta$ satisfies $0<\eta<\frac{1}{2\kappa^2}$, where $\kappa^2$ is the bound of the kernel function described in (\ref{bound of kernel}). 
\end{assumption} 
It is a technical assumption required in the proofs. In particular, it guarantees that the filter function $\varphi_t^{dis}$ of the discrete kernel gradient flow admits an analytic extension to a larger domain in the complex plane (see Appendix \ref{analytic section} for detailed discussions)

\section{Exact Convergence Rate of the Supremum-Norm Generalization Error}

In this section, we present our results on the convergence rate of the supremum-norm generalization error of kernel gradient flow. We further show that the proposed convergence rate is in fact minimax optimal.

\begin{theorem}\label{upper bound} 
    \textnormal{\textbf{(Upper bound of the supremum-norm generalization error)}} Suppose that Assumptions \ref{noise moment}, \ref{edr}, \ref{Holder assumption}, \ref{source condition} and \ref{embedding index assumption} are satisfied. Let $t=n^\theta$ for $\theta\in(0,\beta)$. Then for any $\varepsilon>0$ sufficiently small such that $0<\varepsilon<\min\{s-\frac{1}{\beta},\frac{1}{\theta}-\frac{1}{\beta}\}$ and for any $p>1$, the following estimations hold: 

    (I) For the continuous kernel gradient flow estimator $\hat{f}_t=\hat{f}_t^{con}$, when $n$ is sufficiently great, we have 
    \begin{equation}\label{upper bound estimation for con} 
        \|\hat{f}_t^{con}-f^*\|_\infty\leq Ct^{-\frac{s-\alpha}{2}}+C\sqrt{\frac{p\log n}{n}}t^{\frac{\alpha}{2}}
    \end{equation} 
    with probability $1-\O(n^{-p})$, where $\alpha=\frac{1}{\beta}+\varepsilon$, and the constant $C>0$ depends only on $\varepsilon$, $d$, $\beta$, $L_k$, $h$, $s$, $R$, $\sigma$ and $L$; 

    (II) For the discrete kernel gradient flow estimator $\hat{f}_t=\hat{f}_t^{dis}$, if we additionally assume that Assumption \ref{learning rate assumption} holds as well , then 
    \begin{equation}\label{upper bound estimation for dis} 
        \|\hat{f}_t^{dis}-f^*\|_\infty\leq Ct^{-\frac{s-\alpha}{2}}+C\sqrt{\frac{p\log n}{n}}t^{\frac{\alpha}{2}}
    \end{equation}
    with probability $1-\O(n^{-p})$ when $n$ is sufficiently great, where the constant $C>0$ depends only on $\varepsilon$, $d$, $\beta$, $\kappa$, $L_k$, $h$, $s$, $R$, $\sigma$, $L$ and $\eta$. 
\end{theorem}

The proof of this theorem can be found in Appendix \ref{upper bound proof section}. 

\begin{remark} 
    For brevity, here and throughout the paper, ``depending on $\varepsilon$'' refers to dependence on both $\varepsilon$ and $M_{\alpha_0+\varepsilon}$. 
\end{remark} 

~ 

Based on Theorem \ref{upper bound}, by balancing the terms on the right-hand sides of (\ref{upper bound estimation for con}) and (\ref{upper bound estimation for dis}), we immediately obtain the following result: 

\begin{corollary}\label{best upper bound} 
    \textnormal{\textbf{(Best convergence rate)}} Under the same settings of Theorem \ref{upper bound}, if we select 
    \begin{equation}t_{opt}\asymp n^{\frac{1}{s}},\end{equation} 
    then for any $\varepsilon>0$ sufficiently small, when $n$ is sufficiently great, we have 
    \begin{equation}\label{best upper bound estimation}
        \E\|\hat{f}_{t_{opt}}^{con}-f^*\|_\infty\leq C\cdot n^{-\frac{s\beta-1}{2s\beta}+\varepsilon},\quad \E\|\hat{f}_{t_{opt}}^{dis}-f^*\|_\infty\leq C\cdot n^{-\frac{s\beta-1}{2s\beta}+\varepsilon},
    \end{equation} 
    where the constant $C>0$ depends only on $\varepsilon$, $d$, $\kappa$, $\beta$, $L_k$, $h$, $s$, $R$, $\sigma$ and $L$ for continuous kernel gradient flow, and on $\eta$ additionally for discrete kernel gradient flow. 
\end{corollary} 

The proof of this corollary is deferred to Appendix \ref{best upper bound proof}.

It is natural to ask whether the convergence rate we have established in Corollary \ref{best upper bound} is optimal. Accordingly, in the case of Gaussian noises, we prove the following lower bound result, which shows that the convergence rate (\ref{best upper bound estimation}) (nearly) matches the minimax lower bound rate. 

\begin{theorem}\label{lower bound} 
    \textnormal{\textbf{(Minimax lower bound)}} Suppose that Assumption \ref{edr}, \ref{Holder assumption} and \ref{embedding index assumption} hold. We further assume that the marginal distribution $\mu$ is a Radon measure, and $\varepsilon_j=y_j-f^*(x_j)$ are independent Gaussian noise: $\varepsilon_j|x_j\sim N(0,\sigma^2 I_d)$ for some $\sigma>0$ (in this case, Assumption \ref{noise moment} is satisfied). Define 
    \begin{equation} 
        \mathcal{B}(R)=\{f\in[\H]^s:\,\|f\|_{[\H]^s}\leq R\}
    \end{equation} 
    (Note that $f^*\in\mathcal{B}(R)$ if and only if Assumption \ref{source condition} is satisfied). Then we have 
    \begin{equation} 
        \inf_{\hat{f}}\sup_{f^*\in\mathcal{B}(R)}\E\left\|\hat{f}-f^*\right\|_{\infty}\geq Cn^{-\frac{s\beta-1}{2s\beta}}
    \end{equation} 
    for some constant $C>0$ depending only on $\sigma$, $d$, $\beta$, $L_k$, $h$, $s$, $R$ and $L$, where the infimum is taken over all possible learning methods $\hat{f}$. 
\end{theorem} 
The proof of this theorem is defered to Appendix \ref{lower bound proof section}

\begin{remark}\label{duality limit discussion} 
    By setting $s=1$, Theorem \ref{lower bound} recovers the minimax lower bound established in \citet{duality} (see Theorem 7.1 of \citealt{duality} and its remarks). In \citet{duality}, the embedding index of the kernel is implicitly assumed to be no greater than $1/2$, as their proof relies on the continuity of the feature map $\phi(x,x')=\sum_{i=1}^\infty\sqrt{\lambda_i}e_i(x)e_i(x')$ (see Lemma 4.2 therein). 
\end{remark} 

~

Before concluding this section, we present an important corollary of Theorem \ref{upper bound}, namely the second-order estimation for the kernel gradient flow estimator:  

\begin{theorem}\label{Bahadur representation} 
    Suppose that Assumptions \ref{noise moment}, \ref{edr}, \ref{Holder assumption}, \ref{embedding index assumption}, \ref{source condition} and \ref{learning rate assumption} hold. By choosing $t=n^\theta$ for $\theta\in(0,\beta)$, for any $\varepsilon>0$ sufficiently small such that $0<\varepsilon<\min\{s-\frac{1}{\beta},\frac{1}{\theta}-\frac{1}{\beta}\}$, when $n$ is sufficiently great, the following estimation holds for both the continuous kernel gradient flow $\hat{f}_t=\hat{f}_t^{con}$ and the discrete kernel gradient flow estimator $\hat{f}_t=\hat{f}_t^{dis}$: 
    \begin{equation}
        \left\|\hat{f}_t-f_t-\frac{1}{n}\sum_{i=1}^n\varphi_t(T)k_{x_i}(\cdot)\varepsilon_i\right\|_\infty\leq C\sqrt{\frac{t^\alpha\log n}{n}}\cdot t^{-\frac{s-\alpha}{2}}+C\frac{t^\alpha\log n}{n},
    \end{equation} 
    with probability $1-\O(n^{-10})$, where $\alpha=1/\beta+\varepsilon$, $\varphi_t$ is the filter function defined in Definition \ref{filter function def}, and the constant $C>0$ depends only on $\varepsilon$, $\kappa$, $d$, $\beta$, $s$, $R$, $\sigma$, $L$, $h$ and $L_k$ for continuous kernel gradient flow, and on $\eta$ additionally for discrete kernel gradient flow. 
\end{theorem} 

This estimation is required and directly used in the construction of the simultaneous confidence bands for kernel gradient flow estimators in the next section. The proof of this estimation is delayed to Appendix \ref{second-order estimation proof section}. Results like Theorem \ref{Bahadur representation} are often referred to as the functional Bahadur representation \citep{Bahadur1,krr_inference2}.

\section{Inference for Kernel Gradient Flow} 

In this section, we aim to construct simultaneous confidence bands for kernel gradient flow estimators in the following form: 
\begin{definition} 
    A simultaneous confidence band of the (continuous or discrete) kernel gradient flow estimator $\hat{f}_t(x)$ is a subset of $\X\times\R$ in the following form: 
    \begin{equation} 
        \mathrm{CB}=\{(x,y):\,x\in\X,\,y\in[\hat{f}_t(x)-\hat{\lambda}(x),\hat{f}_t(x)+\hat{\upsilon}(x)]\},
    \end{equation} 
    where $\hat{\lambda}$ and $\hat{\upsilon}$ are non-negative bounded measurable functions on $\X$ depending on the training time $t$ and the data set $\D_n=\{(x_i,y_i)\}_{i=1}^n$. Given a constant $\delta\in(0,1]$, $\mathrm{CB}$ is asymptotically exact with coverage level $1-\delta$, if 
    \begin{equation} 
        \P(f^*\in\mathrm{CB})\to 1-\delta
    \end{equation} 
    in probability as $n\to\infty$. 
\end{definition} 

\subsection{Crucial quantities} 

We first introduce the key quantities involved in the construction of the confidence bands. For any $x,x'\in\X$, define  
\begin{equation}\label{kernel of GP} 
    \begin{aligned}
        &C_{t}(x,x')=\sigma^2\cdot\E_{z\sim\mu}(\varphi_t(T)k_x(z)\cdot\varphi_t(T)k_{x'}(z))\\
        =&\left\{
        \begin{aligned} 
            &\sigma^2\cdot\E_{z\sim\mu}(\varphi_t^{con}(T)k_x(z)\cdot\varphi_t^{con}(T)k_{x'}(z))\quad&&\mbox{for continuous kernel gradient flow};\\
            &\sigma^2\cdot\E_{z\sim\mu}(\varphi_t^{dis}(T)k_x(z)\cdot\varphi_t^{dis}(T)k_{x'}(z))\quad&&\mbox{for discrete kernel gradient flow},
        \end{aligned} 
        \right.
    \end{aligned}
\end{equation} 
where $T:\H\to\H$ is the integral operator defined in (\ref{integral operator definition}), and $\varphi_t$ is the filter function defined in Definition \ref{filter function def}. Define the Gaussian process $W_t(x)$, $x\in\X$ as 
\begin{equation}\label{W definition} 
    W_n(x)\sim\mathrm{GP}\left(0,\frac{C_t(x,x')}{C_t(x,x)^{\frac{1}{2}}C_t(x',x')^{\frac{1}{2}}}\right),
\end{equation} 
and let $Z_t=\|W_t(x)\|_\infty$. 

The function $C_{t}(x,x')$ is intractable in practice. When $x=x'$, we introduce the following empirical estimator of $C_{t}(x,x)$: 
\begin{equation}\label{empirical kernel of GP} 
    \begin{aligned}
        \widehat{C}_{n,t}(x,x)&=\frac{1}{n}\sum_{i=1}^n|\varphi_t(T_X)k_x(x_i)\hat{\varepsilon}_i|^2,\\
        &=\left\{
        \begin{aligned} 
            &\frac{1}{n}\sum_{i=1}^n|\varphi_t^{con}(T_X)k_x(x_i)\hat{\varepsilon}_i|^2\quad&&\mbox{for continuous kernel gradient flow};\\
            &\frac{1}{n}\sum_{i=1}^n|\varphi_t^{dis}(T_X)k_x(x_i)\hat{\varepsilon}_i|^2\quad&&\mbox{for discrete kernel gradient flow},
        \end{aligned} 
        \right.
    \end{aligned}
\end{equation} 
where $\hat{\varepsilon}_i=y_i-\hat{f}_t(x_i)$, $i=1,\dots,n$.  

For both the continuous kernel gradient flow and the discrete kernel gradient flow, the function $\widehat{C}_{n,t}(x,x)$ can be explicit computed: 

(I) For the continuous kernel gradient flow, we have 
\begin{equation}\label{empirical kernel of con GP} 
    \begin{aligned} 
        \widehat{C}_{n,t}(x,x)&=\widehat{C}_{n,t}^{con}(x,x)=\frac{1}{n}\sum_{i=1}^n|\varphi^{con}_t(T_X)k_x(x_i)\hat{\varepsilon}_i|^2\\
        &=n\left|\K(x,X)\K(X,X)^{-1}\left(I_n-\exp(-\frac{t}{n}\K(X,X))\right)\cdot\mathrm{diag}(\hat\varepsilon)\right|^2,
    \end{aligned} 
\end{equation} 
where $\hat\varepsilon=(\hat\varepsilon_1,\dots,\hat\varepsilon_n)$, $\hat\varepsilon_i=y_i-\hat{f}_t(x_i)$; 

(II) For the discrete kernel gradient flow, the function $\widehat{C}_{n,t}(x,x')=\widehat{C}_{n,t}^{dis}(x,x')$ can be computed via the following iterative procedure: For any $t_m=m\eta$, $m=0,1,2\dots$, let $\widehat{F}_{t_m}(x)=(\widehat{F}^1_{t_m}(x),\dots,\widehat{F}^n_{t_m}(x))$ be a mapping from $\X$ to $\R^n$ defined by 
\begin{equation} 
    \begin{aligned} 
        &\widehat{F}_0(x)=0;\\
        &\widehat{F}_{t_m}(x)=\widehat{F}_{t_{m-1}}(x)-\eta\cdot\frac{1}{n}\mathbb{K}(x,X)\cdot(\widehat{\mathbb{F}}_{t_{m-1}}(X)-\mathrm{diag}(\hat{\varepsilon})),\quad m=1,2,\dots
    \end{aligned} 
\end{equation} 
where 
\begin{equation} 
    \widehat{\mathbb{F}}_{t_m}(X)=(\widehat{F}_{t_m}^1(X),\dots,\widehat{F}_{t_m}^n(X)),\quad\widehat{F}_{t_m}^i(X)=(\widehat{F}_{t_m}^i(x_1),\dots,\widehat{F}_{t_m}^i(x_n))^T,
\end{equation} 
then for the discrete kernel gradient flow and $t=t_m=m\eta$, $m=1,2,\dots$, the function $\widehat{C}_{n,t}^{dis}(x,x)$ is computed as 
\begin{equation}\label{empirical kernel of dis GP}
    \widehat{C}_{n,t_m}^{dis}(x,x)=n|\widehat{F}_{t_m}(x)|^2.
\end{equation} 

Next, define 
\begin{equation}\label{W tilde definition}
    \widetilde{W}_{n,t}(x)=\frac{1}{\sqrt{C_t(x,x)}}\cdot\sqrt{n}(\hat{f}_t(x)-f_t(x)),\quad\widetilde{Z}_{n,t}=\|\widetilde{W}_{n,t}\|_\infty, 
\end{equation} 

and 
\begin{equation}\label{W hat definition} 
    \widehat{W}_{n,t}(x)=\frac{1}{\sqrt{\widehat{C}_{n,t}(x,x)}}\cdot\frac{1}{\sqrt{n}}\sum_{j=1}^n\varphi_t(T_X)k_x(x_j)\hat{\varepsilon}_j g_j,\quad \widehat{Z}_{n,t}=\|\widehat{W}_{n,t}\|_\infty
\end{equation} 
where $g_1,\dots,g_n$ are i.i.d. one-dimensional standard Gaussian random variables. $\widehat{Z}_{n,t}$ is called the multiplier bootstrap variable \citep{anti-concentration_c}. 

Similarly with (\ref{empirical kernel of con GP}) and (\ref{empirical kernel of dis GP}), for both continuous and discrete kernel gradient flows, we can compute $\widehat{W}_{n,t}(x)$ and $\widehat{Z}_{n,t}$ explicitly: 

(I) For the continuous kernel gradient flow, we have 
\begin{equation}\label{con W hat definition} 
    \begin{aligned} 
        \widehat{W}_{n,t}(x)&=\widehat{W}_{n,t}^{con}(x)=\frac{1}{\sqrt{\widehat{C}_{n,t}^{con}(x,x)}}\cdot\frac{1}{\sqrt{n}}\sum_{j=1}^n\varphi^{con}_t(T_X)k_x(x_j)\hat{\varepsilon_j} g_j\\
        &=\frac{1}{\sqrt{\widehat{C}^{con}_{n,t}(x,x)}}\cdot\frac{1}{\sqrt{n}}\K(x,X)\K(X,X)^{-1}(1-e^{-\frac{1}{n}\K(X,X)t})\cdot\mathrm{diag}(\hat\varepsilon)\cdot g,
    \end{aligned} 
\end{equation} 
where $g=(g_1,\dots,g_n)^T\sim N(0,I_n)$ is a standard $n$-dimensional Gaussian random variable; 

(II) For the discrete kernel gradient flow and $t=t_m=m\eta$, $m=1,2,\dots$, we compute $\widehat{W}_{n,t_m}(x)=\widehat{W}_{n,t_m}^{dis}(x)$ via the following iterative procedure: Let $\widehat{G}_{t_m}(x)$ be a mapping from $\X$ to $\R$ defined by
\begin{equation} 
    \begin{aligned} 
        &\widehat{G}_0(x)=0;\\
        &\widehat{G}_{t_m}(x)=\widehat{G}_{t_{m-1}}(x)-\eta\cdot\frac{1}{n}\mathbb{K}(x,X)\cdot(\widehat{G}_{t_{m-1}}(X)-\mathrm{diag}(\hat{\varepsilon})\cdot g),\quad m=1,2,\dots
    \end{aligned} 
\end{equation} 
where 
\begin{equation} 
    \quad\widehat{G}_{t_m}(X)=(\widehat{G}_{t_m}(x_1),\dots,\widehat{G}_{t_m}(x_n))^T,
\end{equation} 
then for the discrete kernel gradient flow and $t=t_m=m\eta$, $m=1,2,\dots$, the function $\widehat{W}_{n,t}(x)$ is computed as 
\begin{equation}\label{dis W hat definition}
    \widehat{W}_{n,t_m}=\widehat{W}^{dis}_{n,t_m}(x)=\frac{1}{\sqrt{\widehat{C}_{n,t_m}^{dis}(x,x)}}\cdot\frac{1}{\sqrt{n}}\widehat{G}_{t_m}(x).
\end{equation}

\subsection{Simultaneous confidence band}

In this section, we make the following additional assumption: 
\begin{assumption}\label{lower bound of C} 
    There exists a universal constant $c>0$ such that 
    \begin{equation}C_t(x,x)\geq c\sigma^2 t^{\frac{1}{\beta}}.\end{equation} 
\end{assumption} 

This assumption is introduced to control the covariance of the Gaussian process $W_t$ defined in (\ref{W definition}). It is reasonable in view of Lemma \ref{effect dimension bound} and \ref{operated basis function norm bound}, which imply that (i) $\E_xC_t(x,x)\gtrsim\sigma^2 t^{\frac{1}{\beta}}$, and (ii) $\sup_x C_t(x,x)\lesssim\sigma^2 t^{\alpha}<<\sigma^2 t^{1/\beta}$, where $\alpha=\alpha_0+\varepsilon>1/\beta$. This assumption is satisfied by a broad class of kernel functions. For instance, since $\sin^2\langle \mu,x\rangle+\cos^2\langle\mu,x\rangle=1$ for any $m,x\in\R^d$, we have $C_t(x,x)\asymp\E_xC_t(x,x)\gtrsim\sigma^2 t^{1/\beta}$ for kernel functions that satisfy Assumption \ref{edr} and has Fourier basis. In particular, the shift-invariant periodic kernels (see Section 4.3 of \citet{optimal_rates_spectral_group}) satisfy this assumption. Furthermore, we will show in Lemma \ref{lower bound verification} that the inner-product kernels satisfy this assumption as well. 

We now state our main result of this section:

\begin{theorem}\label{main result} 
    For both the continuous kernel gradient flow estimator $\hat{f}_t=\hat{f}_t^{con}$ and the discrete kernel gradient flow estimator $\hat{f}_t=\hat{f}_t^{dis}$, the following statement holds: Suppose that Assumption \ref{noise moment}, \ref{edr}, \ref{Holder assumption}, \ref{source condition}, \ref{embedding index assumption}, \ref{learning rate assumption} and \ref{lower bound of C} are all satisfied, and we additionally assume that $\E[\varepsilon^2|x]=\sigma^2$ for a.s. $x\in\X$. If we set $t\asymp n^{\theta}$ for $\theta\in(\frac{1}{s},\beta)$, then with probability at least $1-p_n$, we have 
    \begin{equation}\sup_{a\in\R}\left|\P\left(\sqrt{n}\left\|\frac{\hat{f}_t(x)-f^*(x)}{\sqrt{\widehat{C}_{n,t}(x,x)}}\right\|_\infty\leq a\right)-\P\left(\left.\widehat{Z}_n\leq a\right|\D_n\right)\right|\leq q_n,\end{equation} 
    where $p_n=c_1n^{-c_2}$, $q_n=c_3n^{-c_4}$ for some $c_1,c_2,c_3,c_4>0$ depending only on $\theta$, $d$, $\kappa$, $\beta$, $L_k$, $h$, $s$, $R$, $\sigma$, $L$ (and $\eta$ additionally for discrete kernel gradient flow). 
\end{theorem} 

The probability $1-p_n$ arises from the randomness of the samples $\{(x_i,y_i)\}$. As implied by Theorem \ref{upper bound} and its proof, the condition $\theta\in(\frac{1}{s},\beta)$ ensures that the variance term is greater than the bias term, which is a standard assumption in constructing asymptotically exact confidence bands. 

Theorem \ref{main result} implies that the quantity $\sqrt{n}\|(\hat{f}_t(x)-f^*(x))/\widehat{C}_{n,t}(x,x)^{1/2}\|_\infty$ is approximately distributed as the multiplier Gaussian bootstrap variable $\widehat{Z}_n|\D_n$, as long as the sample size $n$ is sufficiently great. This observation motivates the construction of the following confidence bands for both continuous and discrete kernel gradient flows: 

\begin{definition}\label{confidence band construction for con kgd}
    \textnormal{\textbf{(Simultaneous confidence band for continuous kernel gradient flow)}} For preselected sample size $n$, training time $t$ and coverage level $\delta\in(0,1)$, we construct the simultaneous confidence band $\mathrm{CB}^{con}(\delta)$ for the continuous kernel gradient flow by the following steps: 
    \begin{enumerate}[(1)] 
        \item Compute the continuous kernel gradient flow estimator $\hat{f}^{con}_t(x)$ by formula (\ref{solution to gradient flow of estimator}); 
        \item Compute the function $\widehat{C}_{n,t}^{con}(x,x)$ by formula (\ref{empirical kernel of con GP}); 
        \item For each bootstrap iteration, draw a vector $g=(g_1,\dots,g_n)$ from the standard $n$-dimensional normal distribution $N(0,I_n)$, and use it to compute the multiplier bootstrap variable $\widehat{Z}_{n,t}^{con}=\|\widehat{W}_{n,t}^{con}(x)\|_{\infty}$ by formula (\ref{con W hat definition}); 
        \item Across bootstrap iterations, use the samples of $\widehat{Z}_{n,t_k}^{con}$ obtained in step (3) to compute the $\delta$-quantile of $\widehat{Z}_{n,t}^{con}$ conditioning on $\D_n$, denoted as $r(\delta)$; 
        \item The confidence band is computed by 
        \begin{equation}\mathrm{CB}^{con}(\delta)=\left\{(x,y):\,x\in\X,\,y\in\left[\hat{f}^{con}_t(x)-\Delta^{con}(t,x,\delta),\hat{f}_t^{con}(x)+\Delta^{con}(t,x,\delta)\right]\right\},\end{equation} 
        where 
        \begin{equation} 
            \Delta^{con}(t,x,\delta)=r(\delta)\cdot n^{-\frac{1}{2}}\widehat{C}^{con}_{n,t}(x,x)^{\frac{1}{2}}.
        \end{equation}
    \end{enumerate} 
\end{definition} 

\begin{definition}\label{confidence band construction for dis kgd}
    \textnormal{\textbf{(Simultaneous confidence band for continuous kernel gradient flow)}} For preselected sample size $n$, sample data $\D_n=\{(x_1,y_1),\dots,(x_n,y_n)\}$, training time $t_m=m\eta$, $m\in\mathbb{Z}_+$ and coverage level $\delta\in(0,1)$, we construct the simultaneous confidence band $\mathrm{CB}^{dis}(\delta)$ for the continuous kernel gradient flow by the following steps: 
    \begin{enumerate}[(1)] 
        \item Compute the discrete kernel gradient flow estimator $\hat{f}^{dis}_{t_m}(x)$ by Definition \ref{discrete kgd definition}; 
        \item Compute the function $\widehat{C}_{n,t_m}^{dis}(x,x)$ by formula (\ref{empirical kernel of dis GP}); 
        \item For each bootstrap iteration, draw a vector $g=(g_1,\dots,g_n)$ from the standard $n$-dimensional normal distribution $N(0,I_n)$, and use it to compute the multiplier bootstrap variable $\widehat{Z}_{n,t_m}^{dis}=\|\widehat{W}_{n,t_m}^{dis}(x)\|_{\infty}$ by formula (\ref{dis W hat definition}); 
        \item Across bootstrap iterations, use the samples of $\widehat{Z}_{n,t_k}^{dis}$ obtained in step (3) to compute the $\delta$-quantile of $\widehat{Z}_{n,t_m}^{dis}$ conditioning on $\D_n$, denoted as $r(\delta)$; 
        \item The simultaneous confidence band is computed by 
        \begin{equation}\mathrm{CB}^{dis}(\delta)=\left\{(x,y):\,x\in\X,\,y\in\left[\hat{f}^{dis}_{t_m}(x)-\Delta^{dis}(t_m,x,\delta),\hat{f}_{t_k}^{dis}(x)+\Delta^{dis}(t_m,x,\Delta)\right]\right\},\end{equation} 
        where 
        \begin{equation} 
            \Delta^{dis}(t_m,x,\delta)=r(\delta)\cdot n^{-\frac{1}{2}}\widehat{C}^{dis}_{n,t_m}(x,x)^{\frac{1}{2}}
        \end{equation} 
    \end{enumerate} 
\end{definition} 

\begin{remark} 
    Theorem \ref{main result} actually guarantees the asymptotic exactness of the confidence bands $\mathrm{CB}^{con}(\delta)$ and $\mathrm{CB}^{dis}(\delta)$. 
\end{remark} 

\begin{remark}\label{optimal shrink rate}
    We will prove in Lemma \ref{operated basis function norm bound} (by setting $\gamma=0$) that $C_t(x,x)=\O(\sigma^2 t^\alpha)$ for any $x\in\Omega$ and $\alpha>\frac{1}{\beta}$. Thus, the widths of the confidence bands $\mathrm{CB}^{con}(\delta)$ and $\mathrm{CB}^{dis}(\delta)$ are both $\O\left(\sqrt{\frac{t^\alpha\log n}{n}}\right)=\O\left(n^{-\frac{1-\theta\alpha}{2}}\log^{\frac{1}{2}}n\right)$. Recall that in Theorem \ref{main result}, the stopping time is selected as $t=n^\theta$ for $\theta\in(\frac{1}{s},\beta)$. Therefore, if $\theta$ is sufficiently close to $\frac{1}{s}$, i.e. $t$ is sufficienty close to the optimal time $t_{opt}=n^{\frac{1}{s}}$, then the width of the confidence band will be sufficiently close to $\O\left(n^{-\frac{s\beta-1}{2s\beta}}\right)$, which is the minimax optimal rate given by Corollary \ref{best upper bound} and Theorem \ref{lower bound}. In other words, the widths of our confidence bands $\mathrm{CB}^{con}(\delta)$ and $\mathrm{CB}^{dis}(\delta)$ is greater than, yet can be arbitrarily close to the minimax optimal rate. 
\end{remark}

\subsection{Sketch of proof}\label{proof sketch section}

The proof of Theorem \ref{main result} is based on the technique of Gaussian approximation, and is organized into the following three steps:
\begin{equation}\sqrt{n}\left\|\frac{\hat{f}_t(x)-f^*(x)}{\sqrt{\widehat{C}_{n,t}(x,x)}}\right\|_\infty\,\mathop{\approx}\limits^1\,\widetilde{Z}_{n,t}\,\mathop{\approx}\limits^2\,Z_t\,\mathop{\approx}\limits^3\,\widehat{Z}_{n,t}|\D_n.\end{equation}

Estimations 2 and 3 are summarized in the two theorems below, while estimation 1 follows as a corollary of their proofs.

\begin{theorem}\label{gaussian approximation for Z tilde}
    \textnormal{\textbf{(Gaussian approximation for $\widetilde{Z}_{n,t}$)}} Under the same assumptions of Theorem \ref{main result}, for $t=n^\theta$, $\theta\in(\frac{1}{s},\beta)$, we have 
    \begin{equation}\sup_{a\in\R}\left|\P(\widetilde{Z}_{n,t}\leq a)-\P(Z_t\leq a)\right|\leq c_1 n^{-c_2}\end{equation} 
    for some constants $c_1,c_2>0$ depending only on $\theta$, $d$, $\kappa$, $\beta$, $s$, $\sigma$, $L$, $h$, $L_k$, $R$ (and $\eta$ additionally for discrete kernel gradient flow). 
\end{theorem} 

\proof 
See Lemma \ref{approximation 1} in Section \ref{kolmogorov distance section}. 
\qed 

\begin{theorem}\label{gaussian approximation of Z hat} 
    \textnormal{\textbf{(Gaussian approximation for $\widehat{Z}_{n,t}|\mathbb{D}_n$)}} Under the same assumptions of Theorem \ref{main result}, for $t=n^\theta$, $\theta\in(\frac{1}{s},\beta)$, with probability at least $1-p_n$, we have 
    \begin{equation}\sup_{a\in\R}\left|\P\left(\left.\widehat{Z}_{n,t}\leq a\right|\D_n\right)-\P(Z_t\leq a)\right|\leq q_n,\end{equation} 
    where $p_n=c_1n^{-c_2}$, $q_n=c_3n^{-c_4}$ for some $c_i>0$, $i=1,2,3,4$ depending only on $\theta$, $d$, $\kappa$, $\beta$, $s$, $R$, $h$, $L_k$, $\sigma$, $L$ (and $\eta$ additionally for discrete kernel gradient flow). 
\end{theorem} 

\proof 
See Lemma \ref{approximation 2} in Section \ref{kolmogorov distance section}. 
\qed 

These results ultimately lead to Theorem \ref{main result}. The detailed proof of Theorem \ref{main result}, including the remaining estimation 1, can be found in Section \ref{kolmogorov distance section}.

\section{Experiments}\label{experiment section}

In this section, we present numerical experiments to illustrate our results on supremum-norm convergence rates and simultaneous confidence bands for kernel gradient flow estimators.

\subsection{Supremum-norm convergence rate}

In the following experiment, we set the input space to be the $1$-dimensional interval $\X=[0,1]$ equipped with uniform distribution $\mu=U([0,1])$. The kernel function is set to be the Min kernel 
\begin{equation}k_{min}(x,x')=\min(x,x'),\quad x,x'\in[0,1].\end{equation} 
The RKHS of $k_{min}$ on $[0,1]$ is characterized by \citep{high_dim_stat}: 
\begin{equation}\H_{min}=\left\{f:[0,1]\to\R:\,f\mbox{ is absolutely continuous, }f(0)=0,\,\int_0^1(f'(x))^2dx<\infty\right\}.\end{equation} 
Moreover, the eigenvalues and eigenfunctions of $k_{min}$ is computed by 
\begin{equation}\label{spectrum of min kernel} 
    \lambda_j=\left(\frac{2j-1}{2}\pi\right)^{-2},\quad e_j(x)=\sqrt{2}\sin\left(\frac{2j-1}{2}\pi x\right),\quad j=1,2,\dots
\end{equation} 
Thus, it is easy to verify that the eigenvalue decay rate of $k_{min}$ is $\beta=2$.  

Consider the function $f_1(x)=\sqrt{2}\sin(2\pi x)$, whose source condition with respect to $k_{\min}$ is $s=1.5$. Taking $f_1(x)$ as the ground-truth function, we generate $n$ i.i.d. samples ${(x_i,y_i)}{i=1}^n$ from the model $y=f_1(x)+\varepsilon$, where $\varepsilon\sim N(0,\sigma^2)$ with $\sigma=0.2$. We then apply kernel gradient flow regression with kernel $k{\min}$. For continuous kernel gradient flow, we set the training time to be $t=cn^{1/s}$ with $s=1.5$; for discrete kernel gradient flow, we set the learning rate to be $\eta=0.01$ and the number of gradient descent iterations to be $\lfloor\frac{t}{\eta}\rfloor=\lfloor\frac{c}{\eta}n^{1/s}\rfloor$. The constant $c$ is chosen from $0.5$, $2.5$, $10$, $40$, $200$, and the sample size $n$ varies from $1000$ to $4000$ in step of $100$. Finally, for each $c$ and $n$, we repeat the experiments $100$ times and report the relationship between $n$ and the averaged logarithmic generalization error over all $100$ runs. In order to demonstrate the convergence rate clearly, the plot is set to be in logarithmic coordinates $\log \mbox{error} = r\log n+b$, hence the convergence rate is estimated by the slope $r$. The results are presented in Figure \ref{sin_upper_bound_plot}. 

As can be seen from the figure, the supremum-norm convergence rate is approximately $n^{-\frac{s\beta-1}{2s\beta}} = n^{-\frac{1}{3}}$, which is consistent with the result in Theorem \ref{upper bound}.

\begin{figure}[tbp]
    \centering 
    \subfigure[Continuous kernel gradient flow]{
        \includegraphics[width=7cm]{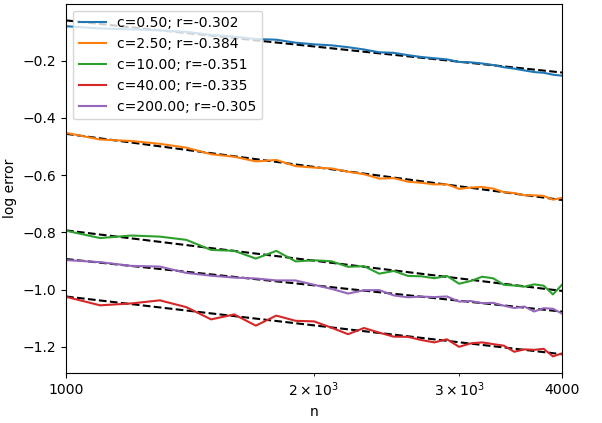}
    }
    \subfigure[Discrete kernel gradient flow]{
        \includegraphics[width=7cm]{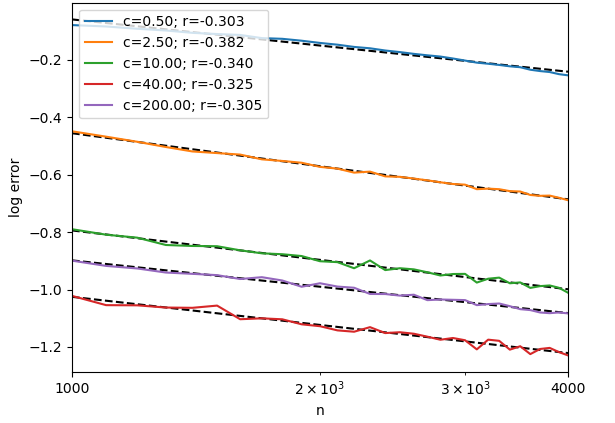}
    }
    \caption{For each selection of $c$, the slope $r$ estimates the convergence rate of the supreme norm generalization error of continuous and discrete kernel gradient flows for the true function $f^*=f_1$. }
    \label{sin_upper_bound_plot}
\end{figure} 

\subsubsection{Comparison with kernel ridge regression} 

One direct corollary of Theorem \ref{upper bound} is that kernel gradient flow regression does not suffer from the saturation effect. The saturation effect refers to the phenomenon that the convergence rate of a regression algorithm fails to attain the information-theoretic lower bound when the smoothness of the regression function (i.e., the source condition $s$) exceeds a certain level. As a typical example, the saturation effect of kernel ridge regression has been observed in practice \citep{spectral1} and has also been theoretically proved in \citep{saturation}. In the following experiment, we compare kernel gradient flow regression with classical kernel ridge regression through numerical experiments. 

We set the target function $f^*$ to be $f_2(x)=\sqrt{2}\sin(\frac{3}{2}\pi x)$. By (\ref{spectrum of min kernel}), $f_2$ is an eigenfunction of $k_{min}$, hence its source condition is $s=\infty$. We draw $n$ samples $\{(x_i,y_i)\}_{i=1}^n$ from the model $y=f^*(x)+\varepsilon$, $\varepsilon\sim N(0,\sigma^2)$, where $\sigma=0.2$. The sample size $n$ varies from $1000$ to $4000$ with step size $100$. The kernel function is still set to be $k_{min}$. For continuous kernel gradient flow, we set the training time to be $t=cn^{\frac{1}{1/\beta+\varepsilon}}$, where $c=100$ is fixed, and $\varepsilon=1,2,3,4,5,6$; for discrete kernel gradient flow, we set the learning rate to be $\eta=0.01$, and the number of gradient descent iterations to be $\lfloor\frac{t}{\eta}\rfloor=\lfloor\frac{c}{\eta}n^{1/s}\rfloor$, $c=100$, $\varepsilon=1,2,3,4,5,6$; in the conduction of kernel ridge regression, we take the ridge parameter to be $\lambda=\frac{1}{t}=\frac{1}{c}n^{-\frac{1}{1/\beta+\varepsilon}}$, $c=100$, $\varepsilon=1,2,3,4,5,6$. For each $\varepsilon$, we repeat the experiments $100$ times and plot the relationship between $n$ and the average of the log generalization error over all $100$ runs. The plot is presented in logarithmic scale $\log\,error=r\log n+b$, hence the convergence rate is estimated by the slope $r$. The results are shown in Figure \ref{comparison_kgf_krr_plot}. 

\begin{figure}[tbp]
    \centering 
    \subfigure[Continuous kernel gradient flow]{
        \includegraphics[width=7cm]{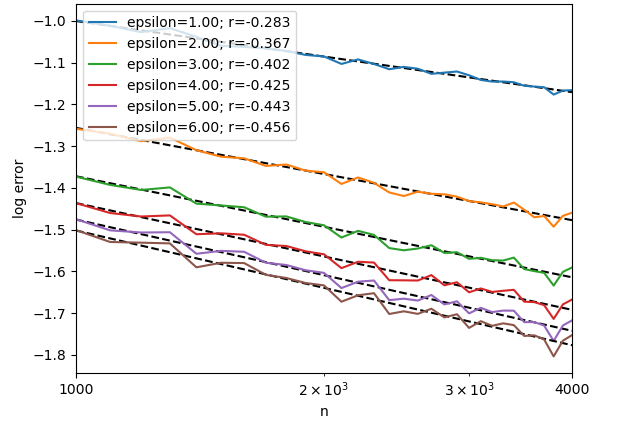}
    }
    \subfigure[Discrete kernel gradient flow]{
        \includegraphics[width=7cm]{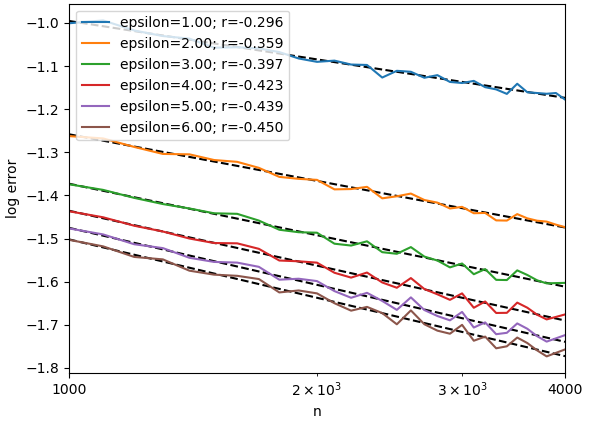}
    }
    \subfigure[KRR]{
        \includegraphics[width=7cm]{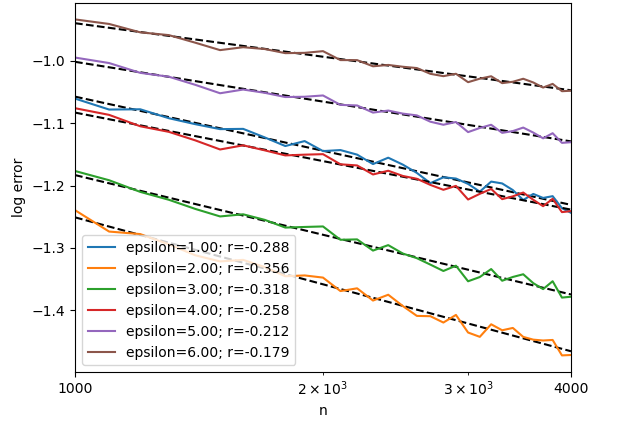}
    }
    \caption{For different selections of $t$ and $\lambda=1/t$, the supreme norm generalization errors of kernel gradient flow and kernel ridge regression for the true function $f^*=f_2$ are reported in the above two figures, respectively. }
    \label{comparison_kgf_krr_plot}
\end{figure} 

As shown in Figure \ref{comparison_kgf_krr_plot}, for the target function $f^*=f_2$ which has a high relative smoothness with respect to $\H_{min}$ (the source condition is $s=\infty$), the supremum-norm generalization error of kernel gradient flow is smaller than that of kernel ridge regression. Moreover, unlike kernel ridge regression, kernel gradient flow does not exhibit the saturation effect in this regime.

\subsubsection{Experiment on Mat\'ern kernel}\label{matern exp section}

The Mat\'ern kernel (see Section 2.1 of \citealt{matern}) is given by : 
\begin{equation}
    \mathcal{M}_{\alpha,h}(r)=\frac{1}{2^{\alpha-1}\Gamma(\alpha)}\left(\frac{\sqrt{2\alpha}r}{h}\right)^\alpha K_\alpha\left(\frac{\sqrt{2\alpha}r}{h}\right),\quad r>0,
\end{equation} 
with parameters $\alpha,h\in(0,\infty)$, where $K_\alpha$ is the modified Bessel function of the second kind of order $\alpha$. For $x,x'\in\X$, it is known that the RKHS of $k_{\alpha,h}(x,x')=\mathcal{M}_{\alpha,h}(|x-x'|)$, the Mat\'ern kernel on $X$, is equivalent with the Sobolev space $H^r(\X)$, where $r=\alpha+d/2$ (see Section 2.3 of \citealt{matern}). Note that when $\alpha=3/2$, the function $\mathcal{M}_{\alpha,h}(r)$ has a closed form: 
\begin{equation} 
    \mathcal{M}_{3/2,h}(r)=\left(1+\frac{\sqrt{3}r}{h}\right)\exp\left(-\frac{\sqrt{3}r}{h}\right).
\end{equation} 

Consider the periodic form of the Mat\'ern kernel:
\begin{equation}\label{matern example}
    k_h(x,x')=\mathcal{M}_{3/2,h}(\sqrt{2-2\cos 2\pi|x-x'|}),\quad x,x'\in[0,1].
\end{equation} 
It can be viewed as an inner-product kernel on the one-dimensional ring $\mathbb{S}^1$: 
\begin{equation} 
    k_h(x,x')=\tilde{k}_h(v,v')=\mathcal{M}_{3/2,h}(|v-v'|),\quad v=e^{2\pi ix}\in\mathbb{S}^1,\,v'=e^{2\pi ix'}\in\mathbb{S}^1.
\end{equation} 
\citet{discrepancies} proves that its RKHS is equivalent with the Sobolev space $H^2(\mathbb{S}^1)$ , hence its eigenvalue decay rate is $\beta=\frac{2r}{d}=\frac{2\alpha+d}{d}=4$ \citep{sphere_sobolev}. 

In the following experiment, we select the kernel function to be $k_h$ with $h=\sqrt{3}/{4}$, and select the ground-truth function $f^*(x)$ to be $f_3(x)=k_{h}(x,0.5)$ with $h=\sqrt{3}/{2}$, which satisfies the source condition $s=1$. We draw $n$ samples $\{(x_i,y_i)\}_{i=1}^n$ from the model $y=f_3(x)+\varepsilon$, $\varepsilon\sim N(0,\sigma^2)$, where $\sigma=0.2$. In this case, the minimax optimal rate is $n^{-3/8}$. 

We conduct the kernel gradient flow regression with kernel function ${k}$. For continuous kernel gradient flow, we select the training time to be $t=cn$, where $s=1$; for the discrete kernel gradient flow, we select the learning rate to be $\eta=0.01$ and the number of gradient descent iterations to be $\lfloor\frac{t}{\eta}\rfloor=\lfloor\frac{c}{\eta}n\rfloor$. Here, the constant $c$ is set to be $0.05$, $0.1$, $0.5$, $2$, $10$, and the sample size $n$ varies from $1500$ to $3000$ in step of $100$. Finally, for each $c$ and $n$, we repeat the experiments $100$ times and plot the relationship between $n$ and the average of the log generalization error over all $100$ runs. In order to demonstrate the convergence rate clearly, the plot is drawn in logarithmic coordinates: $\log \mbox{error} = r\log n+b$, hence the slope $r$ estimates the convergence rate. The results are presented in Figure \ref{sin_upper_bound_plot}. 

\begin{figure}[tbp]
    \centering 
    \subfigure[Continuous kernel gradient flow]{
        \includegraphics[width=7cm]{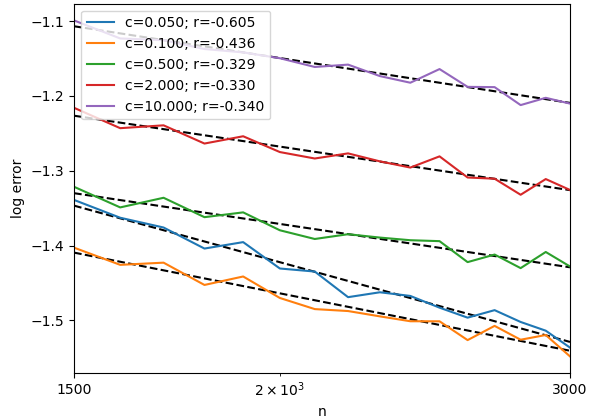}
    }
    \subfigure[Discrete kernel gradient flow]{
        \includegraphics[width=7cm]{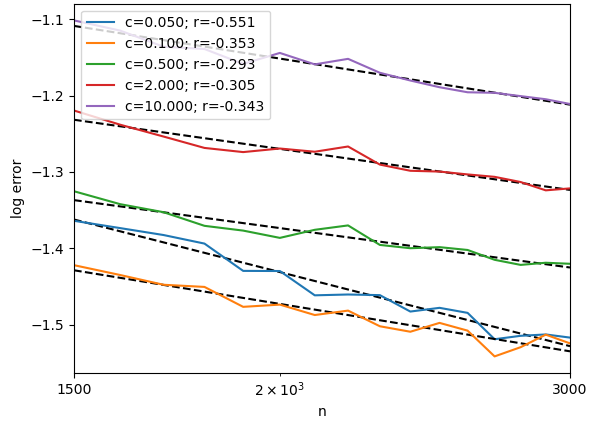}
    }
    \caption{Experimental results on Mat\'ern kernel.}
    \label{matern_plot}
\end{figure}

\subsection{Simultaneous confidence band}\label{confidence band exp section}

In this subsection, we use numerical simulation to evaluate the coverage level of our confidence bands for both continuous and discrete kernel gradient flows. We still use the settings in Section \ref{matern exp section}: $\X=[0,1]$, the kernel function is the periodic Mat\'ern kernel ${k}_h(x,x')$ with $h=\sqrt{3}/3$ and eigenvalue decay rate $\beta=4$, and the regression function is $f^*(x)=f_3(x)$ which satisfies the source condition $s=1$. 

We aim to verify the following two expected phenomena, which are predicted by Theorem \ref{main result}. First, the width of our confidence band is expected to increase as $t$ increases when $n$ is fixed, and to decrease as $n$ increases when $t$ is fixed. Second, the coverage probability of our confidence band is expected to achieve the nominal level when the training time $t$ is smaller than the optimal training time, in which regime the variance term dominates the bias term.

As is shown in Figure \ref{matern_plot}, an empirical optimal selection of training time is $t_{opt}=0.1n$, at which the bias term is comparable with the variance term. We draw $n$ samples $\{(x_i,y_i)\}_{i=1}^n$ from the model $y=f^*(x)+\varepsilon$, $\varepsilon\sim N(0,\sigma^2)$, $\sigma=0.2$. and select the sample size to be $n=500, 1000, 2000, 4000$. For continuous kernel gradient flow, the training time is selected to be $t=0.5t_{opt},\,t_{opt},\,2t_{opt},\,4t_{opt}$; for discrete kernel gradient flow, we select the learning rate to be $\eta=0.01$ and the training time to be $t=0.5t_{opt},\,t_{opt},\,2t_{opt},\,4t_{opt}$, such that the number of gradient descent iterations is $\lfloor t/\eta\rfloor$. For each sample size $n$ and training time $t$, we repeat the experiments $1000$ times; In each run, the bootstrap procedure is performed with $100$ resamples, and we compute the $\delta$-confidence band for $\delta=95\%$. Table \ref{width_plot} summarizes the mean width of the confidence bands for continuous and discrete kernel gradient flows; Table \ref{coverage_plot} summarizes the empirical coverage probabilities of the confidence bands for continuous and discrete kernel gradient flows.

As shown in Table \ref{width_plot} that the widths of confidence bands increase as $t$ increases when $n$ is fixed, and decrease as $n$ increases when $t$ is fixed. Moreover, as shown in Table \ref{coverage_plot}, the coverage probabilities that the true function $f^*$ falls into the confidence bands approximately reach $95\%$ when $t$ is greater than $t_{opt}$ and $n$ is sufficiently great, while the coverage probabilities keeps away from $95\%$ when $t$ is smaller than $t_{opt}$. 

We provide examples of the visualizations of the confidence bands in Appendix \ref{visualization section}. 

\begin{table}[tbp]
    \centering 
    \subtable[Continuous kernel gradient flow]{
    \begin{tabular}{ccccc} 
        \toprule
        sample size & $t=0.5t_{opt}$ & $t=t_{opt}$ & $t=2t_{opt}$ & $t=4t_{opt}$ \\
        \midrule
        $n=500$ & $0.1201$ & $0.1400$ & $0.1596$ & $0.1795$ \\
        $n=1000$ & $0.0999$ & $0.1136$ & $0.1289$ & $0.1440$ \\ 
        $n=2000$ & $0.0809$ & $0.0914$ & $0.1022$ & $0.1142$ \\ 
        $n=3000$ & $0.0711$ & $0.0799$ & $0.0894$ & $0.0995$ \\
        \bottomrule
    \end{tabular}
    }
    \subtable[Discrete kernel gradient flow]{
    \begin{tabular}{ccccc} 
        \toprule
        sample size & $t=0.5t_{opt}$ & $t=t_{opt}$ & $t=2t_{opt}$ & $t=4t_{opt}$ \\
        \midrule
        $n=500$ & $0.1208$ & $0.1397$ & $0.1594$ & $0.1798$ \\
        $n=1000$ & $0.0996$ & $0.1138$ & $0.1283$ & $0.1439$ \\ 
        $n=2000$ & $0.0810$ & $0.0914$ & $0.1024$ & $0.1144$ \\ 
        $n=3000$ & $0.0712$ & $0.0799$ & $0.0895$ & $0.0995$ \\
        \bottomrule
    \end{tabular}
    }
    \caption{The average widths of confidence bands for continuous kernel gradient flow for different sellections of $n$ and $t$. }\label{width_plot} 
\end{table}

\begin{table}[tbp]
    \centering 
    \subtable[Continuous kernel gradient flow]{
    \begin{tabular}{ccccc} 
        \toprule
        sample size & $t=0.5t_{opt}$ & $t=t_{opt}$ & $t=2t_{opt}$ & $t=4t_{opt}$ \\
        \midrule
        $n=500$ & $0.022$ & $0.601$ & $0.870$ & $0.893$ \\
        $n=1000$ & $0.258$ & $0.827$ & $0.925$ & $0.940$ \\ 
        $n=2000$ & $0.618$ & $0.895$ & $0.929$ & $0.935$ \\ 
        $n=3000$ & $0.741$ & $0.910$ & $0.937$ & $0.939$ \\
        \bottomrule
    \end{tabular}
    }
    \subtable[Discrete kernel gradient flow]{
    \begin{tabular}{ccccc} 
        \toprule
        sample size & $t=0.5t_{opt}$ & $t=t_{opt}$ & $t=2t_{opt}$ & $t=4t_{opt}$ \\
        \midrule
        $n=500$ & $0.021$ & $0.588$ & $0.874$ & $0.915$ \\
        $n=1000$ & $0.231$ & $0.888$ & $0.912$ & $0.920$ \\ 
        $n=2000$ & $0.597$ & $0.879$ & $0.921$ & $0.935$ \\ 
        $n=3000$ & $0.745$ & $0.919$ & $0.947$ & $0.947$ \\
        \bottomrule
    \end{tabular}
    }
    \caption{The empirical coverage probabilities of confidence bands for continuous kernel gradient flow for different selections of $n$ and $t$. }\label{coverage_plot} 
\end{table}

\section{Discussions and Conclusion} 

In this paper, we construct simultaneous confidence band for kernel gradient flow based on the estimation of its supremum-norm generalization error convergence rate. In contrast to prior work, the convergence rate established in this paper nearly attains the minimax optimal rate. Moreover, the proposed confidence band is also optimal, in the sense that its width shrinks at nearly optimal rate. 

We conclude by outlining several open problems that remain unresolved:

\begin{enumerate}[(1)]
    \item From a technical perspective, this paper relies on the embedding index condition $\alpha_0 = \frac{1}{\beta}$ (Assumption \ref{embedding index assumption}). To the best of our knowledge, no example of an RKHS violating Assumption \ref{embedding index assumption} has been identified so far. We conjecture that this assumption holds for a broad class of RKHS with eigenvalue decay rate $\beta > 1$.
    \item This paper focuses on kernel gradient flow, while other closely related methods, such as kernel ridge regression and neural networks, are beyond the scope of the present paper. Extending the inference framework developed here to more general statistical and machine learning methods remains an important direction for future research.
    \item To achieve the optimal convergence rate of the generalization error and the optimal shrinkage rate of the confidence band width, the selection of the training time in kernel gradient flow relies heavily on prior information about the target function, such as the source condition $s$. Developing data-driven methods for selecting the optimal early stopping time would significantly improve the practical utility of the proposed methods.
\end{enumerate} 

We hope that our contributions will inspire further research on regression algorithms, inference theory, and other related areas in statistics and machine learning.

% Acknowledgements and Disclosure of Funding should go at the end, before appendices and references

\acks{Zhuo Chen is supported in part by National Natural Science Foundation of China (Grant 12071241). Qian Lin is supported in part by National Natural Science Foundation of China (Grant 92370122, Grant 11971257) and the Beijing Natural Science Foundation (Grant Z190001).}

% Manual newpage inserted to improve layout of sample file - not
% needed in general before appendices/bibliography.

\newpage

\appendix

\section{Proof of Theorem \ref{upper bound}}\label{upper bound proof section} 

Before we start the proof of Theorem \ref{upper bound}, we first introduce or recall some important definitions and notations. Define the sampling covariance operator by 
\begin{equation}T_X:\,C^0(\X)\to\H,\quad (T_X f)(x)=\frac{1}{n}\K(x,X)f(X)=\frac{1}{n}\sum_{j=1}^n k(x,x_j)f(x_j).\end{equation} 
By Assumption \ref{embedding index assumption}, $[\H]^\gamma$ is embedded in $C^0(\X)$ if $\gamma>\alpha_0$, and $\H$ is compactly embedded in $[\H]^\gamma$ if $\gamma\leq 1$, hence $T_X$ can be naturally treated as a continuous operator from $[\H]^\gamma$ onto $[\H]^\gamma$ for any $\gamma\in(\alpha_0,1]$. 

The population version of $T_X$ is 
\begin{equation}T:\,C^0(\X)\to\H,\quad (Tf)(x)=\int_\X k(x,\xi)f(\xi)d\xi.\end{equation} 
Likewise, $T$ can also be treated as an operator from $[\H]^\gamma$ onto $[\H]^\gamma$ for $\gamma\in(\alpha_0,1]$, and the eigenvalues of $T$ as operator from $[\H]^\gamma$ onto $[\H]^\gamma$ are still $\lambda_i$, $i=1,2,\dots$ for any $\gamma\in(\alpha_0,1]$. 

Define the sample basis function by 
\begin{equation}\hat{g}=\frac{1}{n}\sum_{j=1}^n y_jk(x_j,\cdot).\end{equation} 
By taking expectation over the noise $\varepsilon$, we obtain 
\begin{equation}\tilde{g}=\frac{1}{n}\sum_{j=1}^n f^*(x_j)k(x_j,\cdot)=T_X f^*\in\H.\end{equation} 
The population version of $\tilde{g}$ is 
\begin{equation}g=\E_{\{x_j\}}\tilde{g}=\int_\X k(\cdot,\xi)f^*(\xi)d\xi=Tf^*\in\H.\end{equation} 

The kernel gradient flow estimator $\hat{f}_t$ can be expressed as the following form of spectral algorithm: 
\begin{equation}\label{f hat} 
    \hat{f}_t=\varphi_t(T_X)\hat{g},
\end{equation} 
where $\varphi_t(r)$ is defined in Definition \ref{filter function def}. Likewise, we define 
\begin{equation}\label{f tilde} 
    \tilde{f}_t=\varphi_t(T_X)\tilde{g},\quad f_t=\varphi_t(T)g. 
\end{equation} 

For any $x\in\X$, we make the following bias-variance decomposition: 
\begin{equation}\label{bias-variance decomposition} 
    |\hat{f}_t(x)-f^*(x)|\leq\mathrm{Bias}(x,t,X)+\mathrm{Var}(x,t,X,Y),
\end{equation} 
where 
\begin{equation}\label{bias term} 
    \mathrm{Bias}(x,t,X)=|\tilde{f}_t(x)-f^*(x)|
\end{equation} 
and 
\begin{equation}\label{variance term} 
    \mathrm{Var}(x,t,X,Y)=|\hat{f}_t(x)-\tilde{f}_t(x)|. 
\end{equation} 

We will give estimations for the two terms (\ref{bias term}) and (\ref{variance term}) in the next two subsections. Theorem \ref{upper bound} is a direct corollary of Theorem \ref{bias estimation} and Theorem \ref{variance estimation} which will be proved in section \ref{bias estimation subsection} and \ref{variance estimation subsection}, respectively.

\subsection{Bias term}\label{bias estimation subsection}

\begin{theorem}\label{bias estimation} 
    Suppose that Assumptions \ref{noise moment}, \ref{edr}, \ref{Holder assumption}, \ref{source condition}, and \ref{embedding index assumption} hold. Choose $\alpha=\alpha_0+\varepsilon$, $t=n^\theta$ for $\theta\in(0,\beta)$, then for any small $\varepsilon>0$ satisfying $0<\varepsilon<s-\frac{1}{\beta}$ and for any $p>1$, when $n$ is sufficiently great, we have 
    \begin{equation}\label{bias upper bound} 
        \sup_x\mathrm{Bias(x,t,X)}\leq C\left(t^{-\frac{s-\alpha}{2}}+\frac{p\log n}{\sqrt{n}}\right)
    \end{equation} 
    with probability $1-\O(n^{-p})$, for some constant $C>0$ (depending only on $\varepsilon$, $\beta$, $s$ and $R$ for continuous kernel gradient flow, and on $\kappa$, $\eta$ additionally for discrete kernel gradient flow). 
\end{theorem} 

\proof 
Recall that the bias term is defined in (\ref{bias term}). Then we have the following decomposition: 
\begin{equation}\label{decomposition of bias}
    \begin{aligned}
        \sup_{x}\mathrm{Bias}(x,t,X)&=\|\tilde{f}_t-f^*\|_\infty\leq\|\tilde{f}_t-f_t\|_\infty+\|f_t-f\|_\infty\\
        &\leq M_\alpha\|\tilde{f}_t-f_t\|_{[\H]^\alpha}+M_\alpha\|f_t-f^*\|_{[\H]^\alpha},
    \end{aligned}
\end{equation}
where $M_\alpha$ is the embedding coefficient of order $\alpha$ defined in (\ref{embedding coefficient def}). 

For the second term in (\ref{decomposition of bias}), we prove in Lemma \ref{2nd term of bias} that 
\begin{equation}\label{2nd term of bias estimation} 
    \|f_t-f^*\|_{[\H]^\alpha}\leq C t^{-\frac{s-\alpha}{2}}
\end{equation}
for some constant $C>0$ (depending only on $\beta$, $s$ and $R$ for continuous kernel gradient flow, and on $\kappa$, $\eta$ additionally for discrete kernel gradient flow). 

For the first term in (\ref{decomposition of bias}), we prove in Lemma \ref{1st term of bias} that when $n$ is sufficiently great, 
\begin{equation}\label{1st term of bias estimation} 
    \|\tilde{f}_t-f_t\|_{[\H]^\alpha}\leq C\left(t^{-\frac{s-\alpha}{2}}+\frac{p\log n}{\sqrt{n}}\right)
\end{equation} 
for some constant $C>0$ (depending only on $\varepsilon$, $\beta$, $s$ and $R$ for continuous kernel gradient flow, and on $\kappa$, $\eta$ additionally for discrete kernel gradient flow). 

Combining (\ref{2nd term of bias estimation}) and (\ref{1st term of bias estimation}) together, we yield the conclusion. 
\qed

\begin{lemma}\label{2nd term of bias} 
    For the second term of (\ref{decomposition of bias}), we have 
    \begin{equation}\|f_t-f^*\|_{[\H]^\alpha}\leq C t^{-\frac{s-\alpha}{2}}\end{equation}
    for some constant $C>0$ depending only on $\beta$, $s$, $R$ (and $\kappa$, $\eta$ additionally for discrete kernel gradient flow). 
\end{lemma} 

\proof 
Let $f^*(x)=\sum_{i=1}^\infty a_i\lambda_i^{\frac{s}{2}}e_i(x)$ be the expansion of $f^*$ in $[\H]^\alpha$. By assumption \ref{source condition}, we have 
\begin{equation}\sum_{i=1}^\infty a_i^2\leq R^2<\infty.\end{equation} 
Then we have 
\begin{equation}f_t(x)-f^*(x)=\varphi_t(T)g(x)-f^*(x)=(\varphi_t(T)T-1)f^*(x).\end{equation}
Thus, 
\begin{equation}f_t(x)-f^*(x)=\sum_{i=1}^\infty(\varphi_t(\lambda_i)\lambda_i-1)a_i\lambda_i^{\frac{s}{2}}e_i(x)=-\sum_{i=1}^\infty \psi_t(\lambda_i)\lambda_i^{\frac{s-\alpha}{2}}\cdot a_i\lambda_i^{\frac{\alpha}{2}}e_i,\end{equation} 
and by Lemma \ref{exponential estimation}, we obtain 
\begin{equation}\|f_t-f^*\|_{[\H]^\alpha}\leq\sup_{i}\psi_i(\lambda_i)\lambda_i^{\frac{s-\alpha}{2}}\cdot\sum_{i=1}^\infty a_i^2\leq CR^2t^{-\frac{s-\alpha}{2}}\end{equation} 
for some constant $C>0$ depending only on $\beta$, $s$ (and $\kappa$, $\eta$ additionally for discrete kernel gradient flow). 
\qed

\begin{lemma}\label{1st term of bias} 
    For the first term of (\ref{decomposition of bias}), for $p>1$, when $n$ is sufficiently great, with probability $1-\O(n^{-p})$, we have 
    \begin{equation}\|\tilde{f}_t-f_t\|_{[\H]^\alpha}\leq C\left(t^{-\frac{s-\alpha}{2}}+\frac{p\log n}{\sqrt{n}}\right)\end{equation} 
    for some constant $C>0$ depending only on $\varepsilon$, $\beta$, $s$, $R$ for continuous kernel gradient flow, and $\kappa$, $\eta$ additionally for discrete kernel gradient flow. 
\end{lemma} 

\proof 
Note that 
\begin{equation}\label{decomposition of the 1st term of bias}
    \tilde{f}_t-f_t=\varphi_t(T_X)T_Xf^*-\varphi_t(T)Tf^*=\psi_t(T)f^*-\psi_t(T_X)f^*.
\end{equation}
Since $f^*\in[\H]^s$ and $s>\alpha$, then there exists some $u^*\in[\H]^\alpha$ such that $f^*=T^{\frac{s-\alpha}{2}}u^*$, and $\|u^*\|_{[\H]^s}=\|f^*\|_{[\H]^\alpha}\leq R$. Then the first term of (\ref{decomposition of the 1st term of bias}) is estimated by 
\begin{equation}\label{bias 1.1}
    \|\psi_t(T)f^*\|_{[\H]^\alpha}=\|\psi_t(T)T^{\frac{s-\alpha}{2}}u^*\|_{[\H]^\alpha}\leq\|\psi_t(T)T^{\frac{s-\alpha}{2}}\|_{[\H]^\alpha}\cdot\|u^*\|_{[\H]^\alpha}\leq Ct^{-\frac{s-\alpha}{2}}\cdot R, 
\end{equation}
for some constant $C>0$ depending only on $\beta$, $s$ (and $\kappa$, $\eta$ additionally for discrete kernel gradient flow), where we use Lemma \ref{exponential estimation} in the second inequality. 

For the second term of (\ref{decomposition of the 1st term of bias}), we discuss the following two cases: 

(1) If $s-\alpha\leq 2$, then 
\begin{equation}\|\psi_t(T_X)f^*\|_{[\H]^\alpha}=\|\psi_t(T_X)T^{\frac{s-\alpha}{2}}u^*\|_{[\H]^\alpha}\leq\|\psi_t(T_X)T_{X\lambda}\|_{[\H]^\alpha}\|T_{X\lambda}^{-1}T_\lambda\|_{[\H]^\alpha}\|T_\lambda^{-1}T^{\frac{s-\alpha}{2}}\|_{[\H]^\alpha}\|u^*\|_{[\H]^\alpha},\end{equation}
where $T_{X\lambda}=T_X+\lambda$, $T_\lambda=T+\lambda$, $\lambda=\frac{1}{t}$. By Lemma \ref{exponential estimation} and Lemma \ref{fractal estimation}, we have 
\begin{equation}\|\psi_t(T_X)T_{X\lambda}\|_{[\H]^\alpha}\leq C\frac{1}{t},\quad \|T_\lambda^{-1}T^{\frac{s-\alpha}{2}}\|_{[\H]^\alpha}\leq C t^{-\frac{s-\alpha}{2}+1}\end{equation} 
for some constant $C>0$ (universal for continuous kernel gradient flow, and depending on $\kappa$, $\eta$ additionally for discrete kernel gradient flow). By Lemma \ref{operator quotient norm bound}, for $n$ sufficiently great, with probability at least $1-O(n^{-p})$, we have 
\begin{equation}\|T_{X\lambda}^{-1}T_\lambda\|_{[\H]^\alpha}\leq 3.\end{equation} 
Thus, in this case, we have 
\begin{equation}\|\psi_t(T_X)f^*\|_{[\H]^\alpha}\leq CR t^{-\frac{s-\alpha}{
2}}\end{equation} 
for some constant $C>0$ (universal for continuous kernel gradient flow, and depending on $\kappa$, $\eta$ for discrete kernel gradient flow). 

(2) If $s-\alpha>2$, then 
\begin{equation}\psi_t(T_X)f^*=\psi_t(T_X)T^{\frac{s-\alpha}{2}}u^*=\psi_t(T_X)T_X^{\frac{s-\alpha}{2}}u^*+\psi_t(T_X)(T^{\frac{s-\alpha}{2}}-T_X^{\frac{s-\alpha}{2}})u^*.\end{equation}
By Lemma \ref{exponential estimation}, we have 
\begin{equation}\|\psi_t(T_X)T_X^{\frac{s-\alpha}{2}}u^*\|_{[\H]^\alpha}\leq C t^{-\frac{s-\alpha}{2}}\cdot R\end{equation} 
for some constant $C>0$ depending only on $s$, $\beta$, $\varepsilon$ (and $\kappa$, $\eta$ additionally for discrete kernel gradient flow); By Lemma \ref{operator poly diff norm bound}, with probability at least $1-\O(n^{-p})$, we have 
\begin{equation}\|\psi_t(T_X)(T^{\frac{s-\alpha}{2}}-T^{\frac{s-\alpha}{2}})u^*\|_{[\H]^\alpha}\lesssim\|\psi_t(T_X)\|_{[\H]^\alpha}\|T^{\frac{s-\alpha}{2}}-T_X^{\frac{s-\alpha}{2}}\|_{[\H]^\alpha}\|u^*\|_{[\H]^\alpha}\leq 1\cdot CRM_\alpha^2 \frac{p\log n}{\sqrt{n}}\cdot R\end{equation}
for some universal constant $C>0$. 

In conclusion, with probability $1-\O(n^{-10})$, 
\begin{equation}\label{bias 1.2}
    \|\psi_t(T_X)f^*\|_{[\H]^\alpha}\leq C t^{-\frac{s-\alpha}{2}}+C\frac{p\log n}{\sqrt{n}}.
\end{equation} 
for some constant $C>0$ depending only on $\alpha=1/\beta+\varepsilon$, $R$ (and $\kappa$, $\eta$ additionally for discrete kernel gradient flow). 

Combining (\ref{bias 1.1}) and (\ref{bias 1.2}) together, we complete the proof of this lemma. 
\qed

\subsection{Variance term}\label{variance estimation subsection} 

\begin{theorem}\label{variance estimation} 
    Suppose that Assumptions \ref{noise moment}, \ref{edr}, \ref{Holder assumption}, \ref{source condition}, and \ref{embedding index assumption} hold. By choosing $\alpha=\alpha_0+\varepsilon$, $t=n^\theta$ for $\theta\in(0,\beta)$, for any sufficiently small $\varepsilon>0$ satisfying $0<\varepsilon<\frac{1}{\theta}-\frac{1}{\beta}$ and for $p>1$, when $n$ is sufficiently great, with probability $1-\O(n^{-p})$, we have 
    \begin{equation}\label{variance upper bound} 
        \sup_x\mathrm{Var}(x,t,X,Y)\leq C\sqrt{\frac{p\log n}{n}}t^{\frac{\alpha}{2}}
    \end{equation} 
    for some constant $C>0$ (depending only on $\varepsilon$, $d$, $\beta$, $L_k$, $h$, $s$, $R$, $\sigma$ and $L$ for continuous kernel gradient flow, and on $\kappa$, $\eta$ additionally for discrete kernel gradient flow). 
\end{theorem} 

\proof 
Define 
\begin{equation}\eta_i(x)=\varphi_t(T_X)k(x_i,x),\end{equation} 
then the variance term can be rewritten as 
\begin{equation}\mathrm{Var}(x,t,X,Y)=\left|\frac{1}{n}\sum_{i=1}^n\eta_i(x)\varepsilon_i\right|.\end{equation} 
We decomposite the variance term by 
\begin{equation}\label{variance decomposition}
    \mathrm{Var}(x,t,X,Y)\leq \left|\frac{1}{n}\sum_{i=1}^n\eta_i(x)\varepsilon_i^{D_n}\right|+\left|\frac{1}{n}\sum_{i=1}^n\eta_i(x)\varepsilon_{i,\mathrm{tail}}^{D_n}\right|+\left|\frac{1}{n}\sum_{i=1}^n \eta_i(x)m_i^{D_n}\right|,
\end{equation}
where 
\begin{equation}\varepsilon_i^{D_n}=\varepsilon_1I_{|\varepsilon_i|\leq D_n}-m_i^{D_n},\end{equation} 
\begin{equation}\varepsilon_{i,\mathrm{tail}}^{D_n}=\varepsilon_1I_{|\varepsilon_i|>D_n},\end{equation} 
\begin{equation}m_i^{D_n}=\E(\varepsilon_1I_{|\varepsilon_i|\leq D_n}),\end{equation}
where we choose $D_n=n^{-\eta}$ with $\eta\in(0,\frac{1-\theta\alpha}{2})$

By Lemma \ref{1st term of variance}, we have 
\begin{equation}\label{var I estimation} 
    \sup_{x}\left|\frac{1}{n}\sum_{i=1}^n\eta_i(x)\varepsilon_i^{D_n}\right|\leq C\sqrt{\frac{p\log n}{n}}t^{\frac{\alpha}{2}}, 
\end{equation} 
with probability $1-\O(n^{-p})$, where the constant $C>0$ depends only on $\varepsilon$, $d$, $\beta$, $h$, $\sigma$ and $L_k$ (and on $\kappa$, $\eta$ additionally for discrete kernel gradient flow). 

By Lemma \ref{2nd term of variance}, for $n$ sufficiently great, we have 
\begin{equation}\label{var II estimation} 
    \frac{1}{n}\sum_{i=1}^n\eta_i(x)\varepsilon_{i,\mathrm{tail}}^{D_n}=0,\quad\forall x\in\X
\end{equation} 
with probability $1-\O(n^{-p})$. 

By Lemma \ref{3rd term of variance}, for $n$ sufficiently great, we have 
\begin{equation}\label{var III estimation} 
    \sup_x\left|\frac{1}{n}\sum_{i=1}^n\eta_i(x)m_i^{D_n}\right|\leq C\sqrt{\frac{p\log n}{n}}t^{\frac{\alpha}{2}}, 
\end{equation} 
with probability $1-\O(n^{-p})$, where the constant $C>0$ depends only on $\varepsilon$, $\beta$, $\sigma$, $L$ (and $\kappa$, $\eta$ additionally for discrete kernel gradient flow).  

By plugging (\ref{var I estimation}), (\ref{var II estimation}) and (\ref{var III estimation}) together into the decomposition (\ref{variance decomposition}), we complete the proof. 
\qed 

\begin{lemma}\label{1st term of variance} 
    For $p>1$, with probability $1-\O(n^{-p})$, we have 
    \begin{equation}\sup_{x}\left|\frac{1}{n}\sum_{i=1}^n\eta_i(x)\varepsilon_i^{D_n}\right|\leq C\sqrt{\frac{p\log n}{n}}t^{\frac{\alpha}{2}},\end{equation}
    where the constant $C>0$ depends only on $\varepsilon$, $d$, $\beta$, $h$, $\sigma$, $L_k$ (and $\kappa$, $\eta$ additionally for discrete kernel gradient flow). 
\end{lemma} 

\proof 
By Lemma \ref{operated basis function norm bound}, for $n$ sufficiently great, with probability $1-\O(n^{-p})$, we have 
\begin{equation}\label{eta estimation} 
    |\eta_i(x)|=|\varphi_t(T_X)k_{x_i}(x)|\leq M_\alpha\|\varphi_t(T_X)k_{x_i}(\cdot)\|_{[\H]^\alpha}\leq CM_\alpha^2 t^{\alpha}.
\end{equation} 
for some constant $C>0$ (universal for continuous kernel gradient flow, and depending on $\kappa$, $\eta$ additionally for discrete kernel gradient flow). By Lemma \ref{operated basis function empirical norm bound} and \ref{exchange of x and x_i}, we have 
\begin{equation}\label{eta empirical estimation} 
    \sum_{i=1}^n\eta_i(x)^2=\sum_{i=1}^n(\varphi_t(T_X)k_{x_i}(x))^2=\sum_{i=1}^n(\varphi_t(T_X)k_x(x_i))^2=n\|\varphi_t(T_X)k_x(\cdot)\|_{L^2,n}^2\leq CM_\alpha^2\cdot n t^\alpha. 
\end{equation} 
for some constant $C>0$ (universal for continuous kernel gradient flow, and depending on $\kappa$, $\eta$ additionally for discrete kernel gradient flow). Note that 
\begin{equation}\E(\varepsilon_i^{D_n})^2=\E(\varepsilon_i I_{|\varepsilon_i|\leq D_n})^2-(m_i^{D_n})^2=\E\varepsilon_i^2-\E(\varepsilon_i I_{|\varepsilon_i|>D_n})^2-(m_i^{D_n})^2\leq\E\varepsilon_i^2\leq\sigma^2.\end{equation} 
Define $\xi_i(x)=\frac{1}{n}\eta_i(x)\varepsilon_i$. Then $\E\xi_i(x)=0$, and 
\begin{equation}\sum_{i=1}^n\E|\xi_i|^2\leq\sigma^2\cdot\frac{1}{n^2}\sum_{i=1}^n\eta_i(x)^2\leq CM_\alpha^2\sigma^2\cdot\frac{1}{n}t^\alpha,\end{equation}
and for $m\geq 2$, by (\ref{eta estimation}), we have 
\begin{equation}\sum_{i=1}^n\E|\xi_i|^m\leq\left(\sup_x|\xi_i(x)|\right)^{m-2}\cdot\sum_{i=1}^n\E|\xi_i|^2\leq\left(\frac{CM_\alpha^2}{n}t^\alpha D_n\right)^{m-2}\cdot CM_\alpha^2\cdot\sigma^2\cdot\frac{1}{n}t^\alpha.\end{equation}
Using Bernstein inequality in Lemma \ref{Bernstein ineq} with $\tau=(p+\frac{(1+\beta)d}{h})\log n$, $v=CM_\alpha^2\sigma^2n^{-1}t^\alpha$, and $c=CM_\alpha^2 n^{-1}t^\alpha D_n$, we obtain that 
\begin{equation}\P\left(\left|\frac{1}{n}\sum_{i=1}^n\eta_i(x)\varepsilon_i^{D_n}\right|>\sqrt{2v\tau}+c\tau\right)<2e^{-\tau}.\end{equation} 
By the choice of $\alpha$, $t$ and $D_n$ (recall that $t=n^\theta$ for $\theta\in(0,\beta)$, $\alpha=\frac{1}{\beta}+\varepsilon$ for $\varepsilon\in(0,\frac{1}{\theta}-\frac{1}{\beta})$ and $D_n=n^\eta$ for $\eta\in(0,\frac{1-\theta\alpha}{2})$), we have 
\begin{equation}\frac{1}{n}t^\alpha D_n\leq\sqrt{\frac{\log n}{n}}t^{\frac{\alpha}{2}}.\end{equation} 
Therefore, 
\begin{equation}\sqrt{2v\tau}+c\tau\leq C\sqrt{\frac{p\log n}{n}}t^{\frac{\alpha}{2}}\end{equation} 
for some constant $C>0$ depending only on $\varepsilon$, $\sigma$ (and $\kappa$, $\eta$ additionally for discrete kernel gradient flow), 
and with probability at least $1-n^{-(p+\frac{(1+\beta)d}{h})}$, we have 
\begin{equation}\left|\frac{1}{n}\sum_{i=1}^n\eta_i(x)\varepsilon_i^{D_n}\right|\leq C\sqrt{\frac{p\log n}{n}}t^{\frac{\alpha}{2}}\end{equation} 
for some constant $C>0$ depending only on $\varepsilon$, $\sigma$ (and $\kappa$, $\eta$ additionally for discrete kernel gradient flow). 

Suppose that $\mathcal{X}_0$ is an $\varepsilon_0$-net of $\mathcal{X}\subset\R^d$. It is well-known (see section 4.8 of \citep{vershynin} for example) that we can choose $\mathcal{X}_0$ such that 
\begin{equation}\label{net} 
    \varepsilon_0=C'n^{-\frac{1+\beta}{h}},\quad|\mathcal{X}_0|\leq n^{\frac{d(1+\beta)}{h}}
\end{equation}  
for some constant $C'>0$ depending only on $d$, $\beta$ and $h$. For any $x\in\mathcal{X}$, by the definition of $\mathcal{X}_0$, we can find $x_0\in\mathcal{X}_0$ such that $|x-x_0|\leq\varepsilon_0$. Then by Lemma \ref{difference estimation} and Lemma \ref{operated basis function empirical norm bound}, we have 
\begin{equation}\begin{aligned} 
    \left|\frac{1}{n}\sum_{i=1}^n\eta_i(x)\varepsilon_i^{D_n}-\frac{1}{n}\sum_{i=1}^n\eta_i(x_0)\varepsilon_i^{D_n}\right|&\leq D_n\cdot\sup_z\|\varphi_t(T_X)k_z(\cdot)\|_\H\cdot C|x-x_0|^h\\
    &\leq D_n\cdot CM_\alpha\cdot t^{\frac{1+\alpha}{2}}\cdot Cn^{-(1+\beta)}\\
    &\leq C M_\alpha L_k\sqrt{\frac{p\log n}{n}} t^{\frac{\alpha}{2}}
\end{aligned}\end{equation} 
for some constant $C>0$ depending only on $d$, $\beta$, $h$ and $L_k$. 

To sum up, with probability at least $1-n^{-(p+\frac{(1+\beta)d}{h})}\cdot|\X_0|=1-O(n^{-p})$, we have 
\begin{equation}\sup_x\left|\frac{1}{n}\sum_{i=1}^n\eta_i(x)\varepsilon_i^{D_n}\right|\leq C\sqrt{\frac{p\log n}{n}} t^{\frac{\alpha}{2}},\end{equation} 
where the constant $C>0$ depends only on $\varepsilon$, $d$, $\beta$, $h$, $\sigma$, $L_k$ (and $\kappa$, $\eta$ additionally for discrete kernel gradient flow). 
\qed

\begin{lemma}\label{2nd term of variance} 
    For $p>1$, with probability at least $1-O(n^{-p})$, for $n$ sufficiently great, we have 
    \begin{equation}\frac{1}{n}\sum_{i=1}^n\eta_i(x)\varepsilon_{i,\mathrm{tail}}^{D_n}=0,\quad\forall x\in\X.\end{equation} 
\end{lemma} 

\proof 
Note that 
\begin{equation}\P\left(\left|\frac{1}{n}\sum_{i=1}^n\eta_i(x)\varepsilon_{i,\mathrm{tail}}^{D_n}\right|>0\right)\leq\P(\exists i\,\mathrm{s.t.}\,\varepsilon_{i,\mathrm{tail}}^{D_n}\neq 0)=1-\P(|\varepsilon_i|\leq D_n,\,\forall i)\leq 1-\prod_{i=1}^n\left(1-\frac{\E|\varepsilon_i|^m}{D_n^m}\right).\end{equation}
Recall that we choose $D_n=n^{\eta}$, $\eta\in(0,\frac{1-\theta\alpha}{2})$. By choosing $m$ such that $m>\frac{2p}{\eta}$, we have 
\begin{equation}\P\left(\left|\frac{1}{n}\sum_{i=1}^n\eta_i(x)\varepsilon_{i,\mathrm{tail}}^{D_n}\right|>0\right)\leq1-\left(1-\frac{m!\sigma^2 L^{m-2}}{2n^{2p}}\right)^{n}\leq 1-e^{n^{-p}}\leq n^{-p}\end{equation} 
for $n$ sufficiently great. 
\qed

\begin{lemma}\label{3rd term of variance} 
    For $p>1$, with probability at least $1-\O(n^{-p})$, for $n$ sufficiently great, we have 
    \begin{equation}\sup_x\left|\frac{1}{n}\sum_{i=1}^n\eta_i(x)m_i^{D_n}\right|\leq C\sqrt{\frac{p\log n}{n}}t^{\frac{\alpha}{2}},\end{equation}
    where the constant $C>0$ depends only on $\varepsilon$, $\beta$, $\sigma$, $L$ (and $\kappa$, $\eta$ additionally for discrete kernel gradient flow).  
\end{lemma} 

\proof 
For any $m\geq 2$, we have 
\begin{equation}|m_i^{D_n}|=|\E(\varepsilon_iI_{|\varepsilon_i|>D_n})|\leq\frac{\E|\varepsilon_i|^{1+m}}{D_n^m}\leq\frac{1}{2}m!\sigma^2 L^{m-2}n^{-m\eta}.\end{equation}
Thus, by (\ref{eta estimation}), with probability at least $1-\O(n^{-p})$, we have 
\begin{equation}\sup_x\left|\frac{1}{n}\sum_{i=1}^n\eta_i(x)m_i^{D_n}\right|\leq CM_\alpha^2 m!\sigma^2 L^{m-2} n^{-m\eta}t^\alpha\end{equation}
for some constant $C>0$ (universal for continuous kernel gradient flow, and depending only on $\kappa$ and $\eta$ for discrete kernel gradient flow).  

By choosing $m>\frac{\alpha\beta+1}{2\eta}$, we have 
\begin{equation}n^{-m\eta}t^{\alpha}\leq\sqrt{\frac{\log n}{n}}t^{\frac{\alpha}{2}}\leq\sqrt{\frac{p\log n}{n}}t^{\frac{\alpha}{2}}\end{equation} 
which completes the proof. 
\qed 

\subsection{Proof of Corollary \ref{best upper bound}}\label{best upper bound proof} 

By setting $p=10$ and $t=t_{opt}\asymp n^{1/s}$ in Theorem \ref{upper bound}, we obtain that for any $\varepsilon>0$ sufficiently small, when $n$ is sufficiently great, with probality $1-\O(n^{-10})$, we have 
\begin{equation} 
    \|\hat{f}_{t_{opt}}^{con}-f^*\|_\infty\leq C\cdot n^{-\frac{s\beta-1}{2s\beta}+\varepsilon},\quad \|\hat{f}_{t_{opt}}^{dis}-f^*\|_\infty\leq C\cdot n^{-\frac{s\beta-1}{2s\beta}+\varepsilon}.
\end{equation} 
On the other hand, by Lemma \ref{general bound for KGF estimator}, we have 
\begin{equation} 
    \E\|\hat{f}_{t_{opt}}^{con}-f^*\|_\infty\leq Ct,\quad \E\|\hat{f}_{t_{opt}}^{dis}-f^*\|_\infty\leq Ct
\end{equation} 
for some constant $C>0$ (depending only on $\kappa$, $s$, $\sigma$ and $R$ for continuous kernel gradient flow, and on $\eta$ additionally for discrete kernel gradient flow). Therefore, 
\begin{equation} 
    \E\|\hat{f}_{t_{opt}}^{con}-f^*\|_\infty\leq (1-\O(n^{-10}))\cdot Cn^{-\frac{s\beta-1}{2s\beta}+\varepsilon}+\O(n^{-10})\cdot Ct\leq C\cdot n^{-\frac{s\beta-1}{2s\beta}+\varepsilon},
\end{equation} 
\begin{equation} 
    \E\|\hat{f}_{t_{opt}}^{dis}-f^*\|_\infty\leq (1-\O(n^{-10}))\cdot Cn^{-\frac{s\beta-1}{2s\beta}+\varepsilon}+\O(n^{-10})\cdot Ct\leq C\cdot n^{-\frac{s\beta-1}{2s\beta}+\varepsilon}.
\end{equation} 
\qed

\subsection{Auxiliary lemmata} 

\begin{lemma}\label{basis function norm bound} 
    For any $x$ and $\gamma\in[\alpha,1]$, we have 
    \begin{equation}\|k(x,\cdot)\|_{[\H]^\gamma}\leq M_\alpha,\end{equation}
    \begin{equation}\|T_\lambda^{-1}k(x,\cdot)\|_{[\H]^\gamma}\leq M_\alpha t^{\frac{\gamma+\alpha}{2}},\end{equation} 
    \begin{equation}\|T_\lambda^{-\frac{1}{2}}k(x,\cdot)\|_\H\leq M_\alpha t^{\frac{\alpha}{2}},\end{equation} 
    where $T_\lambda=T+\lambda$, $\lambda=\frac{1}{t}$.  
\end{lemma} 

\proof 
Note that 
\begin{equation}k(x,\cdot)=\sum_{i=1}^\infty\lambda_ie_i(x)e_i(\cdot)=\sum_{i=1}^\infty\lambda_i^{\frac{2-\alpha-\gamma}{2}}\cdot\lambda_i^{\frac{\alpha_i}{2}}e_i(x)\cdot\lambda_i^{\frac{\gamma}{2}}e_i(\cdot),\end{equation} 
hence 
\begin{equation}\|k(x,\cdot)\|_{[\H]^\gamma}^2=\sum_{i=1}^\infty\lambda_i^{2-\alpha-\gamma}\cdot\lambda_i^\alpha e_i(x)^2\leq \lambda_1^{2-\alpha-\gamma}M_\alpha^2\leq M_\alpha^2,\end{equation}
where the second inequality comes from 
\begin{equation}\lambda_1=\int_\X\lambda_1e_1(x)^2dx\leq\int_\X k(x,x)dx\leq \kappa^2.\end{equation} 
Also note that 
\begin{equation}T_\lambda^{-1}k(x,\cdot)=\sum_{i=1}^\infty\frac{\lambda_i^{1-\frac{\gamma+\alpha}{2}}}{\lambda_i+\lambda}\lambda_i^{\frac{\alpha}{2}}e_i(x)\cdot\lambda_i^{\frac{\gamma}{2}}e_i(\cdot),\end{equation} 
\begin{equation}T_\lambda^{-\frac{1}{2}}k(x,\cdot)=\sum_{i=1}^\infty\left(\frac{\lambda_i^{1-\alpha}}{\lambda_i+\lambda}\right)^{\frac{1}{2}}\lambda_i^{\frac{\alpha}{2}}e_i(x)\cdot\lambda_i^{\frac{1}{2}}e_i(\cdot),\end{equation} 
hence by Lemma \ref{fractal estimation}, we have 
\begin{equation}\|T_\lambda^{-1}k(x,\cdot)\|_{[\H]^\alpha}\leq \lambda^{-\frac{\gamma+\alpha}{2}}\cdot M_\alpha=t^{\frac{\gamma+\alpha}{2}}\cdot M_\alpha,\end{equation} 
\begin{equation}\|T_\lambda^{-\frac{1}{2}}k(x,\cdot)\|_\H\leq\lambda^{-\frac{\alpha}{2}}\cdot M_\alpha=t^{\frac{\alpha}{2}}\cdot M_\alpha.\end{equation} 
\qed

\begin{lemma}\label{operator norm bound} 
    For $\gamma\in[\alpha,1]$, we have 
    \begin{equation}\|T_X\|_{[\H]^\gamma}\leq M_\alpha^2.\end{equation} 
\end{lemma} 

\proof 
For any $f\in[\H]^\gamma$, by Lemma \ref{basis function norm bound}, we have 
\begin{equation}\label{norm bound of Txi}
    \|T_{x_i}f\|_{[\H]^\gamma}=\|f(x_i)k(x_i,\cdot)\|_{[\H]^\gamma}\leq\|k(x_i,\cdot)\|_{[\H]^\gamma}\cdot|f(x_i)|\leq M_\alpha^2\|f\|_{[\H]^\gamma},
\end{equation} 
hence 
\begin{equation}\|T_X f\|_{[\H]^\gamma}=\|\frac{1}{n}\sum_{i=1}^n T_{x_i}f\|_{[\H]^\alpha}\leq\frac{1}{n}\sum_{i=1}^n\|T_{x_i}f\|_{[\H]^\alpha}\leq M_\alpha^2\|f\|_{[\H]^\gamma}.\end{equation}
\qed 

\begin{lemma}\label{general bound for KGF estimator} 
    For both $\hat{f}_t=\hat{f}_t^{con}$ and $\hat{f}_t=\hat{f}_t^{dis}$, the following statement holds: Given $n$ fixed, we have 
    \begin{equation} 
        \E\|\hat{f}_t-f^*\|_\infty\leq Ct,
    \end{equation} 
    where $C>0$ is a constant depending only on $\kappa$, $s$, $\sigma$ and $R$ for continuous kernel gradient flow, and on $\eta$ additionally for discrete kernel gradient flow. 
\end{lemma} 

\proof 
Recall that by (\ref{f hat}), 
\begin{equation} 
    \hat{f}_t=\varphi_t(T_x)\hat{g}=\frac{1}{n}\sum_{j=1}^n\varphi_t(T_X)k_{x_j}y_j.
\end{equation} 
By Lemma \ref{operator norm bound}, $\|T_X\|_\H\leq M_1^2\leq\kappa^2$. Then, combining with Lemma \ref{estimation of filter function 2}, we have 
\begin{equation} 
    \|\varphi_t(T_X)\|_\H\leq\varphi_t(\kappa^2)\leq \frac{E}{\lambda+\kappa^2}\leq Et,
\end{equation} 
where $E$ is a constant (universal for continuous kernel gradient flow, or depending on $\kappa$ and $\eta$ for discrete kernel gradient flow), and $\lambda=1/t$. Moreover, since $\|k_{x_j}\|_\H^2=k(x_j,x_j)\leq\kappa^2$, we have 
\begin{equation} 
    \|\varphi_t(T_X)k_{x_j}\|_\infty\leq\kappa\|\varphi_t(T_X)k_{x_j}\|_\H\leq\kappa\|\varphi_t(T_X)\|_\H\cdot\|k_{x_j}\|_\H\leq Ct,
\end{equation} 
where $C$ is a constant depending only on $\kappa$ for continuous kernel gradient flow, and on $\eta$ additionally for discrete kernel gradient flow. 

Thus, we have 
\begin{equation} 
    \begin{aligned} 
        \E\|\hat{f}_t-f^*\|_\infty&=\E\left\|\frac{1}{n}\sum_{j=1}^n\varphi_t(T_X)k_{x_j}y_j\right\|_\infty\\
        &\leq\frac{1}{n}\sum_{j=1}^n\E\left[\|\varphi_t(T_X)k_{x_j}\|_\infty\cdot|y_j|\right]\\
        &\leq\frac{1}{n}\sum_{j=1}^n Ct\cdot\E|y_j|=\frac{1}{n}\sum_{j=1}^n Ct\cdot\E|f^*(x_j)+\varepsilon_j|\\
        &\leq\frac{1}{n}\sum_{j=1}^n Ct\cdot\left(\|f^*\|_\infty+\E|\varepsilon_j|\right)\\
        &\leq\frac{1}{n}\sum_{j=1}^n Ct\cdot\left(M_s\|f^*\|_{[\H]^s}+\sqrt{\E|\varepsilon_j|^2}\right)\\
        &\leq\frac{1}{n}\sum_{j=1}^n Ct\cdot\left(M_sR+\sigma\right)=C't,
    \end{aligned}
\end{equation} 
where the last inequality comes from Assumption \ref{source condition} and \ref{noise moment}.
\qed 

\begin{lemma}\label{operator diff norm bound}
    For $\gamma\in[\alpha,1]$ and $p>1$, with probability at least $1-\O(n^{-p})$, we have 
    \begin{equation}\|T_X-T\|_{[\H]^\gamma}\leq CM_\alpha^2 p\frac{\log n}{\sqrt{n}}\end{equation}
    for some $C>0$ is a universal constant.  
\end{lemma} 

\proof 
Recall that (\ref{norm bound of Txi}) implies 
\begin{equation}\|T_{x_i}\|_{[\H]^\gamma}\leq M_\alpha^2<\infty\end{equation}
Then, we apply Lemma \ref{Bernstein ineq vec ver} to $T_{x_i}$ by setting $\sigma=L=M_\alpha^2$, and obtain that 
\begin{equation}\|T_X-T\|_{[\H]^\gamma}=\left\|\frac{1}{n}\sum_{i=1}^n T_{x_i}-\E T_{x_1}\right\|_{[\H]^\gamma}\leq\frac{4\sqrt{2}M_\alpha^2}{\sqrt{n}}\log\frac{2}{\delta}\end{equation} 
with probability at least $1-\delta$. Finally, by setting $\delta=O(n^{-p})$, we finish the proof of this result. 
\qed

\begin{lemma}\label{operator poly diff norm bound} 
    For $r>1$ and $p>1$, with probability at least $1-O(n^{-p})$, we have 
    \begin{equation}\|T_X^r-T^r\|_{[\H]^\alpha}\leq CrM_\alpha^2 p\frac{\log n}{\sqrt{n}},\end{equation} 
    where $C>0$ is a universal constant. 
\end{lemma} 

\proof 
By Lemma 35 of \citet{optimal_rates_spectral_group}, we have 
\begin{equation}\|T_X^r-T^r\|_{[H]^\alpha}\leq rc^{r-1}\|T_X-T\|_{[\H]^\alpha},\end{equation}
where $c=\max\{\|T_X\|_{[\H]^\alpha},\|T\|_{[\H]^\alpha}\}\leq 1$. Thus, combining the above estimation with Lemma \ref{operator norm bound} and Lemma \ref{operator diff norm bound}, we obtain the desired result. 
\qed

\begin{lemma}\label{operator quotient diff norm bound} 
    For any $\gamma\in[\alpha,1]$ and $p>1$, with probability $1-\O(n^{-p})$, we have 
    \begin{equation}\|T_\lambda^{-1}(T_X-T)\|_{[\H]^\gamma}\leq C\sqrt{p}M_\alpha\sqrt{\frac{t^{\frac{\gamma+\alpha}{2}}\log n}{n}},\end{equation} 
    where $C>0$ is a constant depending only on $\beta$. 
\end{lemma} 

\proof 
Define $\xi_i=T_\lambda^{-1}T_{x_i}$. For any $f\in[\H]^\gamma$, we have 
\begin{equation}\xi_i f=T_\lambda^{-1}T_{x_i}f=T_\lambda^{-1}k(x_i,\cdot)f(x_i),\end{equation}
hence by Lemma \ref{basis function norm bound}, we have 
\begin{equation}\|\xi_i f\|_{[\H]^\gamma}\leq\|T_\lambda^{-1}k(x_i,\cdot)\|_{[\H]^\gamma}\cdot|f(x_i)|\leq M_\alpha t^{\frac{\gamma+\alpha}{2}}\cdot M_\alpha\|f\|_{[\H]^\alpha},\end{equation} 
which implies that $\|\xi_i\|_{[\H]^\gamma}\leq M_\alpha^2 t^{\frac{\gamma+\alpha}{2}}$. Using the fact that for a self-adjoint operator $L$, $\E(L-\E L)^2\preceq \E L^2$ and $L^2\preceq \|L\|L$, we obtain 
\begin{equation}\E(T_\lambda^{-1}(T_{x_i}-T))^2=\E(\xi_i-\E\xi_i)^2\preceq\E\xi_i^2\preceq\|\xi_i\|\E\xi_i\preceq M_\alpha^2 t^{\frac{\gamma+\alpha}{2}}T_\lambda^{-1}T.\end{equation} 
Define $V=M_\alpha^2 t^{\frac{\gamma+\alpha}{2}} T_\lambda^{-1}T$. Note that 
\begin{equation}\|V\|_{[\H]^\gamma}=M_\alpha^2 t^{\frac{\gamma+\alpha}{2}}\frac{\lambda_1}{\lambda_1+\lambda},\end{equation} 
\begin{equation}\tr_{[\H]^\gamma}V=M_\alpha^2 t^{\frac{\gamma+\alpha}{2}}\sum_{i=1}^\infty\frac{\lambda_i}{\lambda_i+\lambda}=M_\alpha^2 t^{\frac{\gamma+\alpha}{2}}\mathcal{N}_1(\lambda).\end{equation} 
Applying Lemma \ref{Bernstein ineq operater ver} to $A_i=\xi_i-\E\xi_1$, then for $\delta\in(0,1)$, with probability at least $1-\delta$, we have 
\begin{equation}\|T_\lambda^{-1}(T_X-T)\|_{[\H]^\gamma}\leq\frac{4M_\alpha^2 t^{\frac{\gamma+\alpha}{2}}}{3n}B+\sqrt{\frac{2M_\alpha^2 t^{\frac{\gamma+\alpha}{2}}}{n}B},\end{equation} 
where
\begin{equation}B=\log\frac{4\mathcal{N}_1(\lambda)(\lambda_1+\lambda)}{\delta\lambda_1}.\end{equation} 
By Lemma \ref{effect dimension bound}, $\mathcal{N}_1(\lambda)\asymp\lambda^{-\frac{1}{\beta}}=t^{\frac{1}{\beta}}$. Set $\delta=\O(n^{-p})$ and $t=n^\theta$ for some $\theta\in(0,\beta)$ to be selected, then with probability $1-\O(n^{-p})$, we have 
\begin{equation}\|T_\lambda^{-1}(T_X-T)\|_{[\H]^\gamma}\leq C\sqrt{p}M_\alpha\sqrt{\frac{t^{\frac{\gamma+\alpha}{2}}\log n}{n}},\end{equation}
where $C>0$ is a constant depending only on $\beta$.  
\qed 

\begin{lemma}\label{operator quotient diff norm bound ver2} 
    For any $p>1$, with probability $1-\O(n^{-p})$, we have 
    \begin{equation}\|T_\lambda^{-\frac{1}{2}}(T_X-T)T_\lambda^{-\frac{1}{2}}\|_\H\leq C\sqrt{p}M_\alpha\sqrt{\frac{t^\alpha\log n}{n}},\end{equation} 
    where $C>0$ is a constant depending only on $\beta$. 
\end{lemma} 

\proof 
Denote $\xi_i=T_\lambda^{-\frac{1}{2}}T_{x_i}T_\lambda^{-\frac{1}{2}}$, then 
\begin{equation}T_\lambda^{-\frac{1}{2}}(T_X-T)T_\lambda^{-\frac{1}{2}}=\frac{1}{n}\sum_{i=1}^n(\xi_i-\E\xi_i).\end{equation} 
Note that for any $f\in\H$, we have 
\begin{equation}\begin{aligned}
    \xi_i f&=T_\lambda^{-\frac{1}{2}}T_{x_i}T_\lambda^{-\frac{1}{2}}f=T_\lambda^{-\frac{1}{2}}(k(x_i,\cdot)\cdot T_\lambda^{-\frac{1}{2}}f(x_i))\\
    &=\langle k_{x_i},T_\lambda^{-\frac{1}{2}}f\rangle_\H\cdot T_\lambda^{-\frac{1}{2}}k_{x_i}=\langle T_\lambda^{-\frac{1}{2}}k_{x_i},f\rangle_\H\cdot T_\lambda^{-\frac{1}{2}}k_{x_i},
\end{aligned}\end{equation}
hence by Lemma \ref{basis function norm bound}, 
\begin{equation}\|\xi_i\|_\H=\|T_\lambda^{-\frac{1}{2}}k_{x_i}\|_\H^2\leq M_\alpha^2 t^\alpha.\end{equation} 
Then we have 
\begin{equation}\E(T_\lambda^{-\frac{1}{2}}(T_{x_i}-T)T_\lambda^{-\frac{1}{2}})^2=\E(\xi_i-\E\xi_i)^2\preceq\E\xi_i^2\preceq\|\xi_i\|\E\xi_i\preceq M_\alpha^2 t^\alpha T_\lambda^{-1}T.\end{equation} 
Define $V=M_\alpha^2 t^\alpha T_\lambda^{-1}T$, which satisfies 
\begin{equation}\|V\|_\H=M_\alpha^2 t^\alpha\frac{\lambda_1}{\lambda_1+\lambda},\end{equation}
\begin{equation}\tr_\H V=M_\alpha^2t^\alpha\sum_{i=1}^\infty\frac{\lambda_i}{\lambda_i+\lambda}=M_\alpha^2t^\alpha\mathcal{N}_1(\lambda).\end{equation} 
Applying Lemma \ref{Bernstein ineq operater ver}, we obtain that for any $\delta\in(0,1)$, with probability at least $1-\delta$, 
\begin{equation}\|T_\lambda^{-\frac{1}{2}}(T_X-T)T_\lambda^{-\frac{1}{2}}\|_\H\leq\frac{4M_\alpha^2t^\alpha}{3n}B+\sqrt{\frac{2M_\alpha^2t^\alpha}{n}B},\end{equation} 
where 
\begin{equation}B=\log\frac{4\mathcal{N}_1(\lambda)(\lambda_1+\lambda)}{\delta\lambda_1}.\end{equation} 
By Lemma \ref{effect dimension bound}, $\mathcal{N}_1(\lambda)\asymp\lambda^{-\frac{1}{\beta}}=t^{\frac{1}{\beta}}$. Recall that $t=n^\theta$ for $\theta\in(0,\beta)$. Finally, we set $\delta=\O(n^{-p})$ and obtain that 
\begin{equation}\|T_\lambda^{-\frac{1}{2}}(T_X-T)T_\lambda^{-\frac{1}{2}}\|_\H\leq C\sqrt{p}M_\alpha\sqrt{\frac{t^\alpha\log n}{n}},\end{equation} 
where $C>0$ is a constant depending only on $\beta$. 
\qed 

\begin{lemma}\label{operator quotient norm bound}
    For any $\gamma\in[\alpha,1]$ and $p>1$, with probability at least $1-\O(n^{-p})$, for $n/\sqrt{p}$ sufficiently great (depending on $\varepsilon$ and $\beta$), we have 
    \begin{equation}\|T_{X\lambda}^{-1}T_\lambda\|_{[\H]^\gamma}\leq 3.\end{equation}
\end{lemma}

\proof 
    Recall that $\gamma\geq\alpha=\alpha_0+\varepsilon>\alpha_0$ and $t=n^\theta$ for $\theta\in(0,\beta)$, $\beta=\frac{1}{\alpha_0}$. Then 
    \begin{equation}\sqrt{\frac{t^{\frac{\gamma+\alpha}{2}}\log n}{n}}\to 0\quad\mbox{as }n\to\infty.\end{equation} 
    Thus, by Lemma \ref{operator quotient diff norm bound}, for $n/\sqrt{p}$ sufficiently great (depending on $\varepsilon$ and $\beta$), with probability at least $1-\O(n^{-p})$, we have 
    \begin{equation}\|T_\lambda^{-1}(T_X-T)\|_{[\H]^\gamma}\leq\frac{2}{3},\end{equation}
    which implies that 
    \begin{equation}\|T_{X\lambda}^{-1}T_\lambda\|_{[\H]^\gamma}=\|I-(T_\lambda^{-1}(T_X-T))^{-1}\|_{[\H]^\gamma}\leq\sum_{j=1}^\infty\|T_\lambda^{-1}(T_X-T)\|_{[\H]^\gamma}^j\leq\sum_{j=1}^\infty(\frac{2}{3})^j=3.\end{equation}
\qed 

\begin{lemma}\label{operator quotient norm bound ver2} 
    For $n$ sufficiently great and for $p>1$, with probability $1-\O(n^{-p})$, we have 
    \begin{equation}\|T_{X\lambda}^{-\frac{1}{2}}T_\lambda^{\frac{1}{2}}\|_\H^2=\|T_\lambda^{\frac{1}{2}}T_{X\lambda}^{-1}T_\lambda^{\frac{1}{2}}\|_\H=\|T_\lambda^{\frac{1}{2}}T_{X\lambda}^{-\frac{1}{2}}\|_\H^2\leq 3.\end{equation} 
    \begin{equation}\|T_{X\lambda}^{\frac{1}{2}}T_\lambda^{-\frac{1}{2}}\|_\H^2=\|T_\lambda^{-\frac{1}{2}}T_{X\lambda}T_\lambda^{\frac{1}{2}}\|_\H=\|T_\lambda^{-\frac{1}{2}}T_{X\lambda}^{\frac{1}{2}}\|_\H^2\leq 3.\end{equation} 
\end{lemma} 

\proof 
The second estimation is a trivial corollary of Lemma \ref{operator quotient diff norm bound ver2} since 
\begin{equation}
    \begin{aligned} 
        &\|T_{X\lambda}^{\frac{1}{2}}T_\lambda^{-\frac{1}{2}}\|_\H^2=\|T_\lambda^{-\frac{1}{2}}T_{X\lambda}^{\frac{1}{2}}\|_\H^2=\|T_\lambda^{-\frac{1}{2}}T_{X\lambda}T_\lambda^{-\frac{1}{2}}\|_\H\\
        =&\|T_\lambda^{-\frac{1}{2}}(T_X-T+T_\lambda)T_\lambda^{-\frac{1}{2}}\|_\H\leq 1+\|T_\lambda^{-\frac{1}{2}}(T_X-T)T_\lambda^{-\frac{1}{2}}\|_\H.
    \end{aligned}
\end{equation} 

For the first estimation, by the choice of $\alpha=\alpha_0+\varepsilon$ and $t=n^\theta$, it is directly computed that 
\begin{equation}\sqrt{\frac{t^\alpha\log n}{n}}\to 0\quad\mbox{as }n\to\infty,\end{equation} 
hence by Lemma \ref{operator quotient diff norm bound ver2}, for $n$ sufficiently great, we have 
\begin{equation}\|T_\lambda^{-\frac{1}{2}}(T_X-T)T_\lambda^{-\frac{1}{2}}\|_\H\leq\frac{2}{3},\end{equation} 
which implies that 
\begin{equation}\begin{aligned}
    \|T_\lambda^{\frac{1}{2}}T_{X\lambda}^{-1}T_\lambda^{\frac{1}{2}}\|_\H&=\|(T_\lambda^{-\frac{1}{2}}(T_X+\lambda)T_\lambda^{-\frac{1}{2}})^{-1}\|_\H\\
    &=\|(I-T_\lambda^{-\frac{1}{2}}(T_X-T)T_\lambda^{-\frac{1}{2}})^{-1}\|_\H\\
    &\leq\sum_{j=1}^\infty\|T_\lambda^{-\frac{1}{2}}(T_X-T)T_\lambda^{-\frac{1}{2}}\|_\H^j\\
    &\leq\sum_{j=1}^\infty\left(\frac{2}{3}\right)^j=3.
\end{aligned}\end{equation}
Since $T_{X\lambda}$ and $T_\lambda$ are both self-adjoint, the adjoint operator of $T_{X\lambda}^{-\frac{1}{2}}T_\lambda^{\frac{1}{2}}$ is $T_\lambda^{\frac{1}{2}}T_{X\lambda}^{-\frac{1}{2}}$. Therefore, 
\begin{equation}\|T_{X\lambda}^{-\frac{1}{2}}T_\lambda^{\frac{1}{2}}\|_\H^2=\|T_\lambda^{\frac{1}{2}}T_{X\lambda}^{-\frac{1}{2}}\cdot T_{X\lambda}^{-\frac{1}{2}}T_\lambda^{\frac{1}{2}}\|_\H=\|T_\lambda^{\frac{1}{2}}T_{X\lambda}^{-1}T_\lambda^{\frac{1}{2}}\|_\H\leq 3.\end{equation}
\qed

\begin{lemma}\label{operated basis function norm bound} 
    For any $\gamma\in[0,1]$, we have 
    \begin{equation}\|\varphi_t(T)k_x(\cdot)\|_{[\H]^\gamma}\leq M_\alpha t^{\frac{\gamma+\alpha}{2}},\end{equation} 
    and for any $\gamma\in[\alpha,1]$ and $p>1$, with probability at least $1-\O(n^{-p})$, for $n$ sufficiently great, we have 
    \begin{equation}\|\varphi_t(T_X)k_x(\cdot)\|_{[\H]^\gamma}\leq CM_\alpha t^{\frac{\gamma+\alpha}{2}}\end{equation} 
    for some constant $C>0$ (universal for continuous kernel gradient flow, or depending on $\kappa$ and $\eta$ for discrete kernel gradient flow). 
\end{lemma} 

\proof 
For the first inequality, we first note that 
\begin{equation}\varphi_t(T)k_x(\cdot)=\sum_{i=1}^\infty\varphi_t(\lambda_i)\cdot\lambda_i^{1-\frac{\gamma+\alpha}{2}}\cdot\lambda_i^{\frac{\alpha}{2}}e_i(x)\cdot\lambda_i^{\frac{\gamma}{2}}e_i(\cdot),\end{equation} 
and by Lemma \ref{estimation of filter function}, 
\begin{equation}\left|\varphi_t(\lambda_i)\cdot\lambda_i^{1-\frac{\gamma+\alpha}{2}}\right|\leq t^{\frac{\gamma+\alpha}{2}},\end{equation}
hence 
\begin{equation}\|\varphi_t(T)k_x(\cdot)\|_{[\H]^\gamma}^2\leq t^{\gamma+\alpha}\cdot\sum_{i=1}^\infty\left(\lambda_i^{\frac{\alpha}{2}}e_i(x)\right)^2\leq M_\alpha^2t^{\gamma+\alpha}.\end{equation} 
For the second inequality, by Lemma \ref{basis function norm bound} and Lemma \ref{operator quotient norm bound}, for $n$ sufficiently great, with probability $1-\O(n^{-p})$, we have 
\begin{equation}\begin{aligned} 
    \|\varphi_t(T_X)k_x(\cdot)\|_{[\H]^\gamma}&\leq\|\varphi_t(T_X)T_{X\lambda}\|_{[\H]^\gamma}\cdot\|T_{X\lambda}^{-1}T_\lambda\|_{[\H]^\gamma}\cdot\|T_\lambda^{-1}k_x(\cdot)\|_{[\H]^\gamma}\\
    &\leq\|\varphi_t(T_X)T_{X\lambda}\|_{[\H]^\gamma}\cdot 3\cdot M_\alpha t^{\frac{\gamma+\alpha}{2}}.
\end{aligned}\end{equation}
Note that by Lemma \ref{estimation of filter function 2}, we have 
\begin{equation}\sup_{r\geq 0}\left|\varphi_t(r)(r+\lambda)\right|\leq C\end{equation} 
for some constant $C>0$ (universal for continuous kernel gradient flow, or depending on $\kappa$ and $\eta$ for discrete kernel gradient flow). Thus, 
\begin{equation}\|\varphi_t(T_X)k_x(\cdot)\|_{[\H]^\gamma}\leq 3CM_\alpha t^{\frac{\gamma+\alpha}{2}}.\end{equation} 
\qed

\begin{lemma}\label{operated basis function empirical norm bound} 
    For $p>1$, with probability at least $1-\O(n^{-p})$, for $n$ sufficiently great, we have 
    \begin{equation}\|\varphi_t(T_X)k_x(\cdot)\|_{L^2,n}\leq CM_\alpha t^\frac{\alpha}{2}\end{equation} 
    for some constant $C>0$ (universal for continuous kernel gradient flow, or depending on $\kappa$ and $\eta$ for discrete kernel gradient flow). 
\end{lemma} 

\proof 
First, we have the following observation: for any $f\in\H$, 
\begin{equation}\|f\|_{L^2,n}^2=\frac{1}{n}\sum_{i=1}^n f(x_i)^2=\frac{1}{n}\sum_{i=1}^\infty f(x_i)\langle k_{x_i},f\rangle_\H=\langle T_Xf,f\rangle_\H=\langle T_X^{\frac{1}{2}}f,T_X^{\frac{1}{2}}f\rangle_\H=\|T_X^{\frac{1}{2}}f\|_\H^2.\end{equation} 
Therefore, 
\begin{equation}\|\varphi_t(T_X)k_x(\cdot)\|_{L^2,n}=\|T_X^{\frac{1}{2}}\varphi_t(T_X)k_x(\cdot)\|_\H\leq\|T_X^{\frac{1}{2}}\varphi_t(T_X)T_{X\lambda}^{\frac{1}{2}}\|_\H\cdot\|T_{X\lambda}^{-\frac{1}{2}}T_\lambda^{\frac{1}{2}}\|_\H\cdot\|T_{\lambda}^{-\frac{1}{2}}k_x(\cdot)\|_\H.\end{equation} 
By Lemma \ref{estimation of filter function 2} and \ref{estimation of filter function}, we have 
\begin{equation}\sup_{r\geq 0}\left|r^{\frac{1}{2}}\varphi_t(r)(r+\lambda)^{\frac{1}{2}}\right|\leq C\end{equation} 
for some constant $C>0$ (universal for continuous kernel gradient flow, and depending on $\kappa$ and $\eta$ for discrete kernel gradient flow). Then we have 
\begin{equation}\|T_X^{\frac{1}{2}}\varphi_t(T_X)T_{X\lambda}^{\frac{1}{2}}\|_\H\leq C.\end{equation}
Combining Lemma \ref{basis function norm bound}, \ref{operator quotient norm bound ver2} and the estimation above, we obtain
\begin{equation}\|\varphi_t(T_X)k_x(\cdot)\|_{L^2,n}\leq CM_\alpha t^\frac{\alpha}{2}\end{equation} 
with probability at least $1-\O(n^{-p})$ when $n$ is sufficiently great. 
\qed

\section{Proof of Theorem \ref{lower bound}}\label{lower bound proof section} 

We first introduce the basic concepts and results in the duality framework introduced by \citet{duality}, which will be used in the proof of Theorem \ref{lower bound} later. 

Suppose that there are three Banach spaces: $\mathcal{F}$, $\mathcal{Q}$ and $\mathcal{M}$, such that $\mathcal{F}\subset\mathcal{Q}\cap\mathcal{M}$. The space $\mathcal{F}$ is called the model space, and the norms $\|\cdot\|_\mathcal{Q}$ and $\|\cdot\|_{\mathcal{M}}$ are used for model training and model evaluation, respectively. 

In particular, in this paper, we set 
\begin{equation} 
    \mathcal{F}=[\H]^s,\quad\mathcal{Q}=L^2(\X),\quad\mathcal{M}=C^0(\X).
\end{equation}
For a Banach space $V$, we denote $V(1)=\{w\in V:\,\|w\|_V\leq 1\}$. 

\begin{definition}
    \textnormal{\textbf{(I-complexity)}} For $\varepsilon\geq 0$, we define the I-complexity as 
    \begin{equation} 
        \mathbb{I}_{\mathcal{Q},\mathcal{M}}(\mathcal{F}(1),\varepsilon)=\sup_{f\in\mathcal{F}(1),\,\|f\|_\mathcal{Q}\leq\varepsilon}\|f\|_{\mathcal{M}}.
    \end{equation} 
\end{definition} 

The I-complexity can provide a minimax lower bound for the regression problem in $\mathcal{F}(1)$ with independent Gaussian noises: 

\begin{theorem}\label{IBC for minimax estimation}
    \textnormal{(Proposition 3.6 of \citealt{duality})} Consider the classical regression model 
    \begin{equation} 
        y_j=f^*(x_j)+\varepsilon_j,\quad f^*\in\mathcal{F}(1),\quad j=1,\dots,n
    \end{equation} 
    where $x_j$ are independent samples drawn from $(\X,\mu)$, and $\varepsilon_j$ are independent Gaussian noises, $\varepsilon_j|x_j\sim N(0,\sigma^2 I_d)$ for some $\sigma>0$ ($I_d$ is the $d$-dimensional unit matrix). Then we have 
    \begin{equation} 
        \inf_{\hat{f}}\sup_{f^*\in\mathcal{F}(1)}\left\|\hat{f}-f^*\right\|_{\mathcal{M}}\gtrsim\mathbb{I}_{L^2(\X),\mathcal{M}}(\mathcal{F}(1),\sigma/\sqrt{n}).
    \end{equation} 
\end{theorem} 

The following theorem shows that we can compute the I-complexity by solving an approximation problem in the dual spaces: 

\begin{theorem}\label{duality equation}
    \textnormal{(Theorem 3.3 of \citealt{duality})} The following equality holds: 
    \begin{equation} 
        \mathbb{I}_{\mathcal{Q},\mathcal{M}}(\mathcal{F}(1),\varepsilon)=\sup_{\|g^*\|_{\mathcal{M}^*\leq 1}}\inf_{h^*\in\mathcal{Q}^*}\left[\|g^*-h^*\|_{\mathcal{F}^*}+\varepsilon\|h^*\|_{\mathcal{Q}^*}\right].
    \end{equation}
\end{theorem} 

~ 

Now we begin to prove Theorem \ref{lower bound}. 

\begin{lemma}\label{feature map representation}
    There exists a feature map $\phi:\,\X\times\mathcal{V}\to\R$, where $\mathcal{V}$ is the weight domain of the feature map equipped with a probability distribution $\pi$, such that 
    \begin{equation} 
        [\H]^s=\left\{f(x)=\int_\mathcal{V}a(v)\phi(x,v)d\pi(v):\,a\in L^2(\mathcal{V},\pi)\right\},
    \end{equation} 
    and 
    \begin{equation} 
        \|f\|_{[\H]^s}=c_s\|a\|_{L^2(\mathcal{V},\pi)}\quad\mbox{for }f(x)=\int_\mathcal{V}a(v)\phi(x,v)d\pi(v).
    \end{equation} 
    for some constant $c_s>0$. 
\end{lemma} 
\proof 
We take $\mathcal{V}=\mathbb{Z}_+$, and define $\pi$ as $\pi(i)=c_s\lambda_i^s$, where $c_s$ is a rescaling constant to make $\pi$ indeed a probability distribution: 
\begin{equation} 
    \pi(\mathcal{V})=c_s\sum_{i=1}^\infty\lambda_i^s=1.
\end{equation} 
Note that by Assumption \ref{edr} and \ref{embedding index assumption},  
\begin{equation} 
    \sum_{i=1}^\infty \lambda_i^s\lesssim\sum_{i=1}^\infty i^{-\beta s}<\infty,
\end{equation} 
hence $c_s$ is well-defined. Note that 
\begin{equation} 
    \|a\|_{L^2(\mathcal{V},\pi)}^2=c_s\sum_{i=1}^\infty \lambda_i^s a(i)^2,\quad\forall a\in L^2(\mathcal{V},\pi).
\end{equation} 
Next, define the feature map as 
\begin{equation} 
    \phi(x,v)=e_i(x),\quad x\in\X,\,i\in\mathcal{V}. 
\end{equation}
Now we define the space $\mathcal{W}$ to be 
\begin{equation} 
    \mathcal{W}=\left\{f(x)=\int_\mathcal{V}a(v)\phi(x,v)d\pi(v):\,a\in L^2(\mathcal{V},\pi)\right\},
\end{equation} 
which is equipped with the norm 
\begin{equation} 
    \|f\|_{\mathcal{W}}=\|a\|_{L^2(\mathcal{V},\pi)}\quad\mbox{for }f(x)=\int_\mathcal{V}a(v)\phi(x,v)d\pi(v).
\end{equation} 
On the one hand, any $f\in[\H]^s$ can be represented as 
\begin{equation} 
    f(x)=\sum_{i=1}^\infty f_i\lambda_i^{s/2}e_i(x),\quad \{f_i\}_{i=1}^\infty\in l^2,
\end{equation} 
and we notice that 
\begin{equation} 
    f(x)=c_s\sum_{i=1}^\infty \lambda_i^s a(i)e_i(x)=\int_\mathcal{V}a(i)\phi(x,i)d\pi(i),
\end{equation} 
where 
\begin{equation} 
    a(i)=\frac{f_i}{c_s\lambda_i^{s/2}},
\end{equation} 
hence $f\in\mathcal{W}$, and
\begin{equation} 
    \|f\|_{\mathcal{W}}^2=\|a(i)\|_{L^2(\mathcal{V},\pi)}^2=c_s\sum_{i=1}^\infty \lambda_i^s a(i)^2=\frac{1}{c_s}\sum_{i=1}^\infty f_i^2=\frac{1}{c_s}\|f\|_{[\H]^s}^2.
\end{equation} 
On the other hand, by a similar procedure, we can prove that for any $f\in\mathcal{W}$, we have $f\in[\H]^s$ and $\|f\|_{[\H]^s}=c_s\|f\|_{\mathcal{W}}$. 

This completes the proof of this lemma. 
\qed 

\begin{lemma}\label{RKHS on pi}
    The space 
    \begin{equation} 
        \widetilde{\H}=\left\{a(i)=\int_\X h(x)\phi(x,i)d\mu(x):\, h\in L^2(\X)\right\}
    \end{equation} 
    equipped with a norm 
    \begin{equation} 
        \|a\|_{\widetilde{\H}}=\|h\|_{L^2(\X)}\quad\mbox{for }a(i)=\int_\X h(x)\phi(x,i)d\mu(x)
    \end{equation} 
    is an RKHS. Its reproducing kernel is $\tilde{k}(i,j)=\delta_{ij}$, where $\delta_{ij}$ is the Kronecker delta. Moreover, its eigenvalues (with respect to $L^2(\mathcal{V},\pi)$) are $c_s\lambda_i^s$, $i=1,2,\dots$, whose corresponding eigenfunctions are $\tilde{e}_i(j)=\delta_{ij}/\sqrt{c_s\lambda_i^s}$. 
\end{lemma} 
\proof 
Let 
\begin{equation} 
    \tilde{k}(i,j)=\int_\X \phi(x,i)\phi(x,j)d\mu(x)=\delta_{ij}.
\end{equation} 
Then its corresponding integral operator in $L^2(\mathcal{V},\pi)$ is 
\begin{equation} 
    \tilde{T}a(i)=\int_\mathcal{V}\tilde{k}(i,j)a(j)d\pi(j)=c_s\sum_{i=1}^\infty\lambda_i^s\delta_{ij}a(j)=c_s\lambda_i^sa(i).
\end{equation} 
Thus, the eigenvalues and orthonormal eigenfunctions of $\tilde{T}$ are $\tilde{\lambda}_i=c_s\lambda_i^s$ and $\tilde{e}_i(j)=\delta_{ij}/\sqrt{c_s\lambda_i^s}$,respectively. 

Furthermore, its RKHS is 
\begin{equation} 
    \left\{a(i)=\sum_{j=1}^\infty h_j\sqrt{\tilde{\lambda}_i}\tilde{e}_j(i):\,\{h_i\}\in l^2\right\}=\left\{a(i)=h_i:\,\{h_i\}\in l^2\right\},
\end{equation} 
which is exactly the space $\widetilde\H$. 
\qed

\begin{lemma}\label{finite expansion estimation in H tilde}
    Suppose that $a\in\widetilde\H$. Then for any $m\in\mathbb{Z}_+$, we have 
    \begin{equation} 
        \inf_{r_1,\dots,r_m\in\R}\left\|a-\sum_{j=1}^m r_j\tilde{e}_j\right\|_{L^2(\mathcal{V},\pi)}\leq\sqrt{\tilde\lambda_{m+1}}\|h\|_{\widetilde\H}.
    \end{equation} 
    and the minimum can be reached at some $r_1,\dots,r_m$. 
\end{lemma} 
\proof 
Note that the $L^2$ expansion of $a$ is $a=\sum_{j=1}^\infty a(j)\tilde{e}_j$, hence 
\begin{equation} 
    \left\|a-\sum_{j=1}^m r_j\tilde{e}_j\right\|_{L^2(\mathcal{V},\pi)}^2=\sum_{j=1}^m(a(j)-r_j)^2+\sum_{j={m+1}}^\infty a(j)^2.
\end{equation}
whose minimum is reached if and only if $r_j=a(j)$. Therefore, 
\begin{equation} 
    \inf_{r_1,\dots,r_m\in\R}\left\|a-\sum_{j=1}^m r_j\tilde{e}_j\right\|_{L^2(\mathcal{V},\pi)}^2=\sum_{j=m+1}^\infty a(j)^2=\tilde{\lambda}_{m+1}\sum_{j=m+1}^\infty\tilde\lambda_{j}^{-1}a(j)^2\leq\tilde{\lambda}_{m+1}\|a\|_{\widetilde\H}^2.
\end{equation} 
\qed 

\begin{lemma}\label{finite expansion estimation 2}
    For any $m\in\mathbb{Z}_+$, we have the following estimation: 
    \begin{equation} 
        \sup_{\|\gamma\|_{TV}\leq 1}\inf_{r_1,\dots,r_m\in\R}\left\|\int_\X \phi(x,\cdot)d\gamma(x)-\sum_{j=1}^m r_j\tilde{e}_j\right\|_{L^2(\mathcal{V},\pi)}\geq\sqrt{\lambda(m)}, 
    \end{equation} 
    where 
    \begin{equation} 
        \lambda(m)=\sum_{j={m+1}}^\infty\tilde{\lambda}_j=c_s\sum_{j=m+1}^\infty\lambda_j^s.
    \end{equation} 
\end{lemma} 
\proof 
For any $z\in\X$, the Dirac measure $\delta_z$ is a Radon measure with total variance $1$. Therefore,  
\begin{equation} 
    \begin{aligned} 
        &\sup_{\|\gamma\|_{TV}\leq 1}\inf_{r_1,\dots,r_m\in\R}\left\|\int_\X \phi(x,\cdot)d\gamma(x)-\sum_{j=1}^m r_j\tilde{e}_j\right\|_{L^2(\mathcal{V},\pi)}\\
        \geq&\sup_{z\in\X}\inf_{r_1,\dots,r_m\in\R}\left\|\int_\X \phi(x,\cdot)d\delta_z(x)-\sum_{j=1}^m r_j\tilde{e}_j\right\|_{L^2(\mathcal{V},\pi)}\\
        \geq&\E_{z\sim\mu}\inf_{r_1,\dots,r_m\in\R}\left\|\int_\X \phi(x,\cdot)d\delta_z(x)-\sum_{j=1}^m r_j\tilde{e}_j\right\|_{L^2(\mathcal{V},\pi)}\\
        =&\E_{z\sim\mu}\inf_{r_1,\dots,r_m\in\R}\left\|\phi(z,\cdot)-\sum_{j=1}^m r_j\tilde{e}_j\right\|_{L^2(\mathcal{V},\pi)}\\
        =&\E_{z\sim\mu}\inf_{r_1,\dots,r_m\in\R}\left[c_s\sum_{j=1}^m\lambda_j^s\left(e_j(z)-r_j/\sqrt{c_s\lambda_j^s}\right)^2+c_s\sum_{j=m+1}^\infty\lambda_j^s e_j(z)^2\right]^{\frac{1}{2}}\\
        =&\E_{z\sim\mu}\left[c_s\sum_{j=m+1}^\infty\lambda_j^s e_j(z)^2\right]^{\frac{1}{2}}\\
        =&\sqrt{c_s\sum_{j=m+1}^\infty\lambda_j^s}=\sqrt{\lambda(m)}.
    \end{aligned}
\end{equation} 
\qed

\begin{lemma}\label{duality lemma}
    Let $\mathcal{R}(\X)$ be the set of all Radon measures with finite total variances on $\X$. Recall that $\mathcal{B}(1)=\{f\in[\H]^s:\,\|f\|_{[\H]^s}\leq 1\}$ (defined in Theorem \ref{lower bound}). Then we have 
    \begin{equation} 
        \mathbb{I}_{L^2(\X),C^0(\X)}(\mathcal{B}(1),\varepsilon)=\sup_{\|\gamma\|_{TV}\leq 1}\inf_{h\in L^2(\X)}\left[I(\gamma,h)+\varepsilon\|h\|_{L^2(\X)}\right],
    \end{equation} 
    where for $\gamma\in\mathcal{R}(\X)$ and $h\in L^2(\X)$, the quantity $I(\gamma,h)$ is defined by 
    \begin{equation} 
        I(\gamma,h)=\left\|\int_\X \phi(x,\cdot)d\gamma(x)-\int_\X h(x)\phi(x,\cdot)d\mu(x)\right\|_{L^2(\mathcal{V},\pi)}.
    \end{equation} 
\end{lemma} 
\proof 
Applying Theorem \ref{duality equation} for $\mathcal{F}=[\H]^s$, $\mathcal{Q}=L^2(\X)$, $\mathcal{M}=C^0(\X)$, we obtain that 
\begin{equation}\label{duality application}
    \begin{aligned} 
        \mathbb{I}_{L^2(\X),C^0(\X)}(\mathcal{B}(1),\varepsilon)&=\sup_{\|f\|_{[\H]^s}\leq 1,\,\|f\|_{L^2(\X)}\leq\varepsilon}\|f\|_\infty\\
        &=\sup_{\|b^*\|_{C^0(\X)^*}\leq 1}\inf_{c^*\in L^2(\X)^*}\left[\|b^*-c^*\|_{([\H]^s)^*}+\varepsilon\|c^*\|_{L^2(\X)^*}\right]\\
        &=\sup_{\|\gamma\|_{TV}\leq 1}\inf_{h\in L^2(\X)}\left[\sup_{\|f\|_{[\H]^s}\leq 1}[\gamma(f)-h(f)]+\varepsilon\|h\|_{L^2(\X)}\right],
    \end{aligned} 
\end{equation} 
where in the third equality, we use the facts that $L^2(\X)^*=L^2(\X)$, $([\H]^s)^*=[\H]^s$ and $C^0(\X)^*=\mathcal{R}(\X)$ by Riesz representation theorem. 

Next, by Lemma \ref{feature map representation}, we have 
\begin{equation}\label{computation of pi norm}
    \begin{aligned} 
        \sup_{\|f\|_{[\H]^s}\leq 1}[\gamma(f)-h(f)]&=\sup_{\|f\|_{[\H]^s}\leq 1}\left[\int_\X f(x)d\gamma(x)-\int_\X f(x)h(x)d\mu(x)\right]\\
        &=\sup_{\|a\|_{L^2(\mathcal{V},\pi)}\leq 1}\int_\X\int_\mathcal{V} a(i)\phi(x,i)d\pi(i)\left[d\gamma(x)-h(x)d\mu(x)\right]\\
        &=\sup_{\|a\|_{L^2(\mathcal{V},\pi)}\leq 1}\int_\mathcal{V} a(i)\left(\int_\X \phi(x,i)d\gamma(x)-\int_\X h(x)\phi(x,i)d\mu(x)\right)d\pi(i)\\
        &=\left\|\int_\X \phi(x,i)d\gamma(x)-\int_\X h(x)\phi(x,i)d\mu(x)\right\|_{L^2(\mathcal{V},\pi)}\\
        &=I(\gamma,h),
    \end{aligned} 
\end{equation} 
where $g(i)=\int_\X \phi(x,i)d\gamma(x)\in\widetilde{\mathcal{B}}$. 
Combining (\ref{duality application}) and (\ref{computation of pi norm}) together, we complete the proof of this lemma. 
\qed 

~

\textbf{Final proof of Theorem \ref{lower bound}}. 

Without loss of generality, we assume $R=1$. Set $\varepsilon=\frac{\sigma}{\sqrt{n}}$. Then by Lemma \ref{IBC for minimax estimation} and \ref{duality lemma}, we have 
\begin{equation}\label{lower bound 1st estimation}
    \inf_{\hat{f}}\sup_{f^*\in\mathcal{B}(1)}\E\left\|\hat{f}-f^*\right\|_{\infty}\gtrsim \mathbb{I}_{L^2(\X),C^0(\X)}(\mathcal{B}(1),\varepsilon)=\sup_{\|\gamma\|_{TV}\leq 1}\inf_{h\in L^2(\X)}\left[I(\gamma,h)+\varepsilon\|h\|_{L^2(\X)}\right].
\end{equation} 
Note that 
\begin{equation} 
    I(\gamma,h)=\left\|\int_\X \phi(x,i)d\gamma(x)-a(i)\right\|_{L^2(\mathcal{V},\pi)},
\end{equation} 
where 
\begin{equation} 
    a(i)=\int_\X h(x)\phi(x,i)d\mu(x)\in\widetilde\H.
\end{equation} 
By Lemma \ref{RKHS on pi}, $\widetilde\H$ is an RKHS with eigenvalues $\tilde\lambda_i=c_s\lambda_i^{s}$ and eigenfunctions $\tilde{e}_i(j)=\delta_{ij}/\sqrt{c_s\lambda_i^s}$. For any $m\in\mathbb{Z}_+$, by Lemma \ref{finite expansion estimation in H tilde}, there exist $c_1^*,\dots,c_m^*\in\R$ such that 
\begin{equation} 
    \begin{aligned} 
        &\left\|\int_\X \phi(x,i)d\gamma(x)-a(i)\right\|_{L^2(\mathcal{V},\pi)}\\
        \geq&\left\|\int_\X \phi(x,i)d\gamma(x)-\sum_{j=1}^m c_j^*\tilde{e}_j\right\|_{L^2(\mathcal{V},\pi)}-\left\|a(i)-\sum_{j=1}^m c_j^*\tilde{e}_j\right\|_{L^2(\mathcal{V},\pi)}\\
        \geq&\left\|\int_\X \phi(x,i)d\gamma(x)-\sum_{j=1}^m c_j^*\tilde{e}_j\right\|_{L^2(\mathcal{V},\pi)}-\sqrt{\tilde{\lambda}_{m+1}}\|a\|_{\widetilde\H}.
    \end{aligned}
\end{equation}
Combining this estimation with (\ref{lower bound 1st estimation}), and setting $m$ sufficiently great so that $\varepsilon>\sqrt{\tilde{\lambda}_{m+1}}$, we obtain that 
\begin{equation}
    \begin{aligned} 
        &\inf_{\hat{f}}\sup_{f^*\in\mathcal{B}(1)}\E\left\|\hat{f}-f^*\right\|_{\infty}\\
        \gtrsim&\sup_{\|\gamma\|_{TV}\leq 1}\inf_{h\in L^2(\X)}\left[\left\|\int_\X \phi(x,i)d\gamma(x)-\sum_{j=1}^m c_j^*\tilde{e}_j\right\|_{L^2(\mathcal{V},\pi)}+(\varepsilon-\sqrt{\tilde{\lambda}_m})\|a\|_{\widetilde\H}\right]\\
        \geq&\sup_{\|\gamma\|_{TV}\leq 1}\left\|\int_\X \phi(x,i)d\gamma(x)-\sum_{j=1}^m c_j^*\tilde{e}_j\right\|_{L^2(\mathcal{V},\pi)}\\
        \geq&\sup_{\|\gamma\|_{TV}\leq 1}\inf_{c_1,\dots,c_m\in\R}\left\|\int_\X \phi(x,i)d\gamma(x)-\sum_{j=1}^m c_j\tilde{e}_j\right\|_{L^2(\mathcal{V},\pi)}\\
        \geq&\sqrt{\lambda(m)}
    \end{aligned} 
\end{equation} 
where we use Lemma \ref{finite expansion estimation 2} in the last inequality. 

Finally, we select $m$ such that $\varepsilon\asymp\sqrt{\tilde\lambda_{m+1}}$, then we have 
\begin{equation} 
    \frac{\sigma^2}{n}\asymp\tilde\lambda_{m+1}^s\asymp m^{-s\beta},
\end{equation} 
hence 
\begin{equation} 
    \begin{aligned} 
        \lambda(m)&=c_s\sum_{i=m+1}^\infty\lambda_i^s\gtrsim \sum_{i=m+1}^\infty i^{-s\beta }\gtrsim\int_{m}^\infty x^{-s\beta}dx\geq m^{-s\beta +1}\gtrsim n^{-\frac{s\beta-1}{s\beta}},
    \end{aligned} 
\end{equation} 
This completes the proof of this theorem. 
\qed

\section{Proof of Theorem \ref{main result}}\label{gaussian approximation section}

As is shown in the proof of Theorem \ref{variance estimation} in Section \ref{variance estimation subsection}, a suitable threshold of truncation always exists when we handle with unbounded noise satisfying the Bernstein-type boundedness condition (Assumption \ref{noise moment}). Therefore, for brevity, we assume in the following context that the noise $\varepsilon=y-f^*(x)$ is uniformly bounded, that is, 
\begin{equation} 
    \P(|\varepsilon|<D)=1.
\end{equation}
for some constant $D>0$.

\subsection{Second-order estimation}\label{second-order estimation proof section}

In this section, we first prove the following second-order estimation which serves as an essential technique in the proof of Theorem \ref{main result}. 

\begin{theorem}\label{second-order estimation theorem} 
    \textnormal{(Theorem \ref{Bahadur representation})} Suppose that Assumptions  \ref{noise moment}, \ref{edr}, \ref{Holder assumption}, \ref{source condition} and \ref{embedding index assumption} hold. By choosing $t=n^\theta$ for $\theta\in(0,\beta)$, for any $\varepsilon>0$ sufficiently small such that $0<\varepsilon<\min\{s-\frac{1}{\beta},\frac{1}{\theta}-\frac{1}{\beta}\}$ and for any $p>1$, when $n$ is sufficiently great, we have 
    \begin{equation}\label{second-order estimation} 
        \left\|\hat{f}_t-f_t-\frac{1}{n}\sum_{i=1}^n\varphi_t(T)k_{x_i}(\cdot)\varepsilon_i\right\|_\infty\leq C\sqrt{\frac{t^\alpha\log n}{n}}\cdot t^{-\frac{\min\{s-\alpha,2\}}{2}}+C\frac{t^\alpha\log n}{n},
    \end{equation} 
    with probability $1-\O(n^{-p})$, where the constant $C>0$ depends only on $\varepsilon$, $d$, $\kappa$, $\beta$, $s$, $R$, $\sigma$, $L$, $h$ and $L_k$ for continuous kernel gradient flow, and $\eta$ additionally for discrete kernel gradient flow. 
\end{theorem} 

\proof 
The proof is analogous to the proof of the estimation for the variance term in section \ref{variance estimation subsection}. 

By the definitions of $\hat{f}_t$, $\tilde{f}_t$ and $f_t$ given in (\ref{f hat}) and (\ref{f tilde}), we rewrite the left-hand side of (\ref{second-order estimation}) into 
\begin{equation}\label{decomposition of second-order representation}
    \begin{aligned} 
        \hat{f}_t-f_t-\frac{1}{n}\sum_{i=1}^n\varphi_t(T)k_{x_i}(\cdot)\varepsilon_i&=(\tilde{f}_t-f_t)+\left(\hat{f}_t-\tilde{f}_t-\frac{1}{n}\sum_{i=1}^n\varphi_t(T)k_{x_i}(\cdot)\varepsilon_i\right)\\
        &=(\tilde{f}_t-f_t)+\frac{1}{n}\sum_{i=1}^n\zeta_i(\cdot)\varepsilon_i,
    \end{aligned} 
\end{equation} 
where 
\begin{equation}\zeta_i(x)=\varphi_t(T_X)k_{x_i}(x)-\varphi_t(T)k_{x_i}(x).\end{equation} 

For the first term in (\ref{decomposition of second-order representation}), we will prove in Lemma \ref{1st term of bias 2} that with probability $1-\O(n^{-p})$, 
\begin{equation}\label{second-order estimation 1} 
    \|\tilde{f}_t-f_t\|_{[\H]^\alpha}\leq C\sqrt{\frac{t^\alpha\log n}{n}}\cdot t^{-\frac{\min\{s-\alpha,2\}}{2}}
\end{equation}
for some constant $C>0$ depending only on $\varepsilon$, $R$ and $\kappa$ (and $\eta$ additionally for discrete kernel gradient flow); 

For the second term in (\ref{decomposition of second-order representation}), by Lemma \ref{norm bound of zeta}, for $n$ sufficiently great, with probability $1-\O(n^{-p})$, we have 
\begin{equation}\label{zeta uniform} 
    |\zeta_i(x)|\leq\|\zeta\|_\infty\leq M_\alpha\|\zeta\|_{[\H]^\alpha}\leq C\sqrt{\frac{t^\alpha\log n}{n}}t^\alpha\log t
\end{equation}
for some constant $C>0$ depending only on $\varepsilon$,  $\beta$ and $\kappa$ (and $\eta$ additionally for discrete kernel gradient flow); By Lemma \ref{empirical norm bound of zeta}, we have 
\begin{equation}\label{zeta empirical} 
    \begin{aligned} 
        \sqrt{\frac{1}{n}\sum_{i=1}^n\zeta_i(x)^2}&=\|(\varphi_t(T_X)-\varphi_t(T))k_x(\cdot)\|_{L^2,n}\\
        &=\|T_X^{\frac{1}{2}}(\varphi_t(T_X)-\varphi_t(T))k_x(\cdot)\|_\H\leq C\sqrt{\frac{t^\alpha\log n}{n}}t^\frac{\alpha}{2}\log t
    \end{aligned}
\end{equation}
for some constant $C>0$ depending only on $\varepsilon$,  $\beta$ and $\kappa$ (and $\eta$ additionally for discrete kernel gradient flow). 

Applying the Bernstein inequality (Theorem \ref{Bernstein ineq}) with $\tau=(10+\frac{(1+\beta)d}{h})\log n$, we obtain that for any fixed $x\in\X$, with probability at least $1-2n^{-(10+\frac{(1+\beta)d}{h})}$, we have 
\begin{equation}\left|\frac{1}{n}\sum_{i=1}^n\eta_i(x)\varepsilon_i\right|\leq C\frac{t^\alpha\log n}{n}.\end{equation} 
We choose an $\varepsilon_0$-net $\X_0$ of $\X$ satisfying (\ref{net}), that is, 
\begin{equation}\varepsilon_0=C'n^{-\frac{1+\beta}{h}},\quad|\X_0|\leq n^{\frac{d(1+\beta)}{n}}.\end{equation} 
Then for any $x\in\X$, there exists $x_0\in\X_0$ such that $|x-x_0|\leq\varepsilon$ and 
\begin{equation}\begin{aligned} 
    \left|\frac{1}{n}\sum_{i=1}^n(\eta_i(x)-\eta_i(x_0))\varepsilon_i\right|&\leq C\cdot\sup_{x\in\X}\|(\varphi_t(T_X)-\varphi_t(T))k_x(\cdot)\|_\H\cdot|x-x_0|^h\\
    &\leq C\sqrt{\frac{t^{\frac{1+\alpha}{2}}\log n}{n}} t^{\frac{1+\alpha}{2}}\log t\cdot n^{-(1+\beta)}\\
    &\leq C\frac{t^\alpha\log n}{n}, 
\end{aligned}\end{equation}
where we use Lemma \ref{difference estimation} in the first inequality. Thus, with probability at least $1-|\X_0|\cdot 2n^{-(10+\frac{(1+\beta)d}{h})}=1-\O(n^{-10})$, we have 
\begin{equation}\label{second-order estimation 2}
    \left\|\hat{f}_t(x)-\tilde{f}_t(x)-\frac{1}{n}\sum_{i=1}^n\varphi_t(T)k_{x_i}(x)\varepsilon_i\right\|_\infty\leq C\frac{t^\alpha\log n}{n},
\end{equation} 
Finally, combining (\ref{decomposition of second-order representation}), (\ref{second-order estimation 1}) and (\ref{second-order estimation 2}), we complete the proof of this theorem. 
\qed

\begin{lemma}\label{1st term of bias 2} 
    When $n$ is sufficiently great, with probability $1-\O(n^{-10})$, we have 
    \begin{equation}\|\tilde{f}_t-f_t\|_{[\H]^\alpha}\leq C\sqrt{\frac{t^\alpha\log n}{n}}\cdot t^{-\frac{\min\{s-\alpha,2\}}{2}}\end{equation} 
    for some constant $C>0$ depending only on $\varepsilon$, $R$ and $\kappa$ (and $\eta$ additionally for discrete kernel gradient flow). 
\end{lemma}

\proof 
Recall that we have proved in Lemma \ref{1st term of bias} (see (\ref{decomposition of the 1st term of bias})) that 
\begin{equation}\tilde{f}_t-f_t=\psi_t(T)f^*-\psi_t(T_X)f^*,\end{equation} 
and by the integration formula (\ref{analytic integral forumla}) in Theorem \ref{analytic functional estimation}, we have 
\begin{equation}\begin{aligned} 
    &\tilde{f}_t-f_t\\
    =&\frac{1}{2\pi i}\oint_{\Gamma_{t}}(R_{T_X}(z)-R_T(z))f^*\cdot \psi_t(z)dz\\
    =&\frac{1}{2\pi i}\oint_{\Gamma_{t}}(T_X-z)^{-1}(T-T_X)(T-z)^{-1}f^*\cdot \psi_t(z)dz\\
    =&\frac{1}{2\pi i}\oint_{\Gamma_t}(T_X-z)^{-1}T_{X\lambda}\cdot T_{X\lambda}^{-1}T_\lambda\cdot T_\lambda^{-1}(T-T_X)\cdot (T-z)^{-1}T_\lambda\cdot T_\lambda^{-1}f^*\cdot \psi_t(z)dz.
\end{aligned}\end{equation} 
By Lemma \ref{fractal estimation}, it is easy to prove that 
\begin{equation}\label{operated regression function estimation} 
    \|T_\lambda^{-1}f^*(\cdot)\|_{[\H]^\alpha}\leq Ct^{-\min\{\frac{s-\alpha}{2},1\}+1}
\end{equation} 
for some constant $C>0$ depending only on $R$. Then, combining (\ref{operated regression function estimation}), Lemma \ref{operator quotient diff norm bound}, Lemma \ref{operator quotient norm bound}, Lemma \ref{norm estimation along contour} and Lemma \ref{integral estimation along contour 2}, for $n$ sufficiently great, with probability $1-\O(n^{-10})$, we have 
\begin{equation}\begin{aligned} 
    &\|\tilde{f}_t-f_t\|_{[\H]^\alpha}\\
    \leq&\frac{1}{2\pi}\|(T_X-z)^{-1}T_{X\lambda}\|_{[\H]^\alpha}\cdot\|T_{X\lambda}^{-1}T_\lambda\|_{[\H]^\alpha}\cdot\|T_\lambda^{-1}(T-T_X)\|_{[\H]^\alpha}\\
    &\cdot\|(T-z)^{-1}T_\lambda\|_{[\H]^\alpha}\cdot\|T_\lambda^{-1}f^*\|_{[\H]^\alpha}\cdot\oint_{\Gamma_{t}}|\psi_t(z)dz|\\
    \leq&\frac{1}{2\pi}\cdot C\cdot 3\cdot CM_\alpha\sqrt{\frac{t^\alpha\log n}{n}}\cdot C\cdot Ct^{-\min\{\frac{s-\alpha}{2},1\}+1}\cdot \frac{C}{t}\\
    \leq& C'\sqrt{\frac{t^\alpha\log n}{n}}\cdot t^{-\min\{\frac{s-\alpha}{2},1\}},
\end{aligned}\end{equation}
where the constant $C'>0$ depends only on $\varepsilon$, $R$, $\kappa$ (and $\eta$ additionally for discrete kernel gradient flow). 
\qed

\begin{lemma}\label{norm bound of zeta}
    For any $\gamma\in[\alpha,1]$, if $n$ is sufficiently large, with probability $1-\O(n^{-10})$, we have 
    \begin{equation}\|(\varphi_t(T)-\varphi_t(T_X))k_x(\cdot)\|_{[\H]^\gamma}\leq C\sqrt{\frac{t^{\frac{\gamma+\alpha}{2}}\log n}{n}} t^{\frac{\gamma+\alpha}{2}}\log t,\end{equation}
    where $C>0$ is a constant depending only on $\varepsilon$, $\beta$ and $\kappa$ (and $\eta$ additionally for discrete kernel gradient flow).  
\end{lemma} 

\proof 
By applying the analytic functional integration formula (\ref{analytic integral forumla}) in Theorem \ref{analytic functional estimation} on $T$ and $T_X$, we have 
\begin{equation}\begin{aligned} 
    &(\varphi_t(T)-\varphi_t(T_X))k_x(\cdot)\\
    =&\frac{1}{2\pi i}\oint_{\Gamma_t}(R_{T_X}(z)-R_{T}(z))k_x(\cdot)\cdot\varphi_t(z)dz\\
    =&\frac{1}{2\pi i}\oint_{\Gamma_t}(T_X-z)^{-1}(T-T_X)(T-z)^{-1}k_x(\cdot)\cdot\varphi_t(z)dz\\
    =&\frac{1}{2\pi i}\oint_{\Gamma_t}(T_X-z)^{-1}T_{X\lambda}\cdot T_{X\lambda}^{-1}T_\lambda\cdot T_\lambda^{-1}(T-T_X)\cdot(T-z)^{-1}T_\lambda\cdot T_\lambda^{-1}k_x(\cdot)\cdot\varphi_t(z)dz.
\end{aligned}\end{equation} 
where the contour $\Gamma_t$ is defined in Definition \ref{contour definition}, and we recall that $T_\lambda=T+\lambda$, $T_{X\lambda}=T_X+\lambda$, $\lambda=\frac{1}{t}$. Thus, by Lemma \ref{basis function norm bound}, Lemma \ref{operator quotient diff norm bound}, Lemma \ref{operator quotient norm bound}, Lemma \ref{norm estimation along contour} and Lemma \ref{integral estimation along contour}, for $n$ sufficiently great, with probability $1-\O(n^{-10})$, we have 
\begin{equation}\begin{aligned} 
    &\|(\varphi_t(T)-\varphi_t(T_X))k_x(\cdot)\|_{[\H]^\gamma}\\
    \leq&\frac{1}{2\pi}\|(T_X-z)^{-1}T_{X\lambda}\|_{[\H]^\gamma}\cdot\|T_{X\lambda}^{-1}T_\lambda\|_{[\H]^\gamma}\cdot\|T_\lambda^{-1}(T-T_X)\|_{[\H]^\gamma}\\
    &\cdot\|(T-z)^{-1}T_\lambda\|_{[\H]^\gamma}\cdot\|T_\lambda^{-1}k_x(\cdot)\|_{[\H]^\gamma}\cdot\oint_{\Gamma_t}|\varphi_t(z)dz|\\
    \leq&\frac{1}{2\pi}\cdot C\cdot 3\cdot CM_\alpha\sqrt{\frac{t^{\frac{\gamma+\alpha}{2}}\log n}{n}}\cdot C\cdot M_\alpha t^{\frac{\gamma+\alpha}{2}}\cdot C\log t. 
\end{aligned}\end{equation}
\qed

\begin{lemma}\label{empirical norm bound of zeta} 
    If $n$ is sufficiently large, with probability $1-\O(n^{-10})$, we have 
    \begin{equation}\|T_X^{\frac{1}{2}}(\varphi_t(T)-\varphi_t(T_X)k_x(\cdot))\|_\H\leq C\sqrt{\frac{t^\alpha\log n}{n}}t^{\frac{\alpha}{2}}\log t,\end{equation}
    where the constant $C>0$ depends only on $\varepsilon$, $\beta$ and $\kappa$ (and $\eta$ additionally for discrete kernel gradient flow). 
\end{lemma} 

\proof 
Similarly with Lemma \ref{norm bound of zeta}, we apply Theorem \ref{analytic functional estimation} to $T$ and $T_X$ and obtain 
\begin{equation}\begin{aligned} 
    &T_X^{\frac{1}{2}}(\varphi_t(T)-\varphi_t(T_X)k_x(\cdot))\\
    =&T_X^{\frac{1}{2}}\frac{1}{2\pi i}\oint_{\Gamma_t}(R_{T_X}(z)-R_{T}(z))k_x(\cdot)\cdot\varphi_t(z)dz\\
    =&\frac{1}{2\pi i}\oint_{\Gamma_t}T_X^{\frac{1}{2}}(T_X-z)^{-1}(T-T_X)(T-z)^{-1}k_x(\cdot)\cdot\varphi_t(z)dz\\
    =&\frac{1}{2\pi i}\oint_{\Gamma_t}T_X^{\frac{1}{2}}T_{X\lambda}^{-\frac{1}{2}}\cdot T_{X\lambda}^{\frac{1}{2}}T_\lambda^{-\frac{1}{2}}\cdot T_\lambda^{\frac{1}{2}}(T_X-z)^{-1}T_\lambda^{\frac{1}{2}}\cdot T_\lambda^{-\frac{1}{2}}(T-T_X)T_\lambda^{-\frac{1}{2}}\\
    &\quad\quad\quad\cdot T_\lambda^{\frac{1}{2}}(T-z)^{-1}T_\lambda^{\frac{1}{2}}\cdot T_\lambda^{-\frac{1}{2}}k_x(\cdot)\varphi_t(z)dz,
\end{aligned}\end{equation}
hence by Lemma \ref{basis function norm bound}, Lemma \ref{operator quotient diff norm bound ver2}, Lemma \ref{operator quotient norm bound ver2}, Lemma \ref{norm estimation along contour} and Lemma \ref{integral estimation along contour}, we have 
\begin{equation}\begin{aligned}
    &\|T_X^{\frac{1}{2}}(\varphi_t(T)-\varphi_t(T_X)k_x(\cdot))\|_\H\\
    \leq&\frac{1}{2\pi}\|T_X^{\frac{1}{2}}T_{X\lambda}^{-\frac{1}{2}}\|_\H\cdot\|T_{X\lambda}^{\frac{1}{2}}T_\lambda^{-\frac{1}{2}}\|_\H\cdot\|T_\lambda^{\frac{1}{2}}(T_X-z)^{-1}T_\lambda^{\frac{1}{2}}\|_\H\cdot\|T_\lambda^{-\frac{1}{2}}(T-T_X)T_\lambda^{-\frac{1}{2}}\|_\H\\
    &\cdot\|T_\lambda^{\frac{1}{2}}(T-z)^{-1}T_\lambda^{\frac{1}{2}}\|_\H\cdot\|T_\lambda^{-\frac{1}{2}}k_x(\cdot)\|_\H\cdot\oint_{\Gamma_t}|\varphi_t(z)dz|\\
    \leq&\frac{1}{2\pi}\cdot 1\cdot\sqrt{3}\cdot C\cdot CM_\alpha\sqrt{\frac{t^\alpha\log n}{n}}\cdot C\cdot CM_\alpha t^{\frac{\alpha}{2}}\cdot C\log t.
\end{aligned}\end{equation}
\qed 

Likewise, we also have the following version of second-order estimation: 

\begin{theorem}\label{second-order estimation theorem 2}
    Suppose that Assumptions \ref{edr}, \ref{Holder assumption}, \ref{embedding index assumption}, \ref{source condition} and \ref{noise moment} hold. Let $g_i$, $i=1,\dots,n$ be independent standard Gaussian random variables. By choosing $t=n^\theta$ for $\theta\in(0,\beta)$, for any $\varepsilon>0$ sufficiently small such that $0<\varepsilon<\min\{s-\frac{1}{\beta},\frac{1}{\theta}-\frac{1}{\beta}\}$, when $n$ is sufficiently great, we have 
    \begin{equation}\left\|\frac{1}{n}\sum_{i=1}^n\zeta_i(\cdot)\varepsilon_i g_i\right\|_\infty\leq C\sqrt{\frac{t^\alpha\log n}{n}}\cdot t^{-\frac{\min\{s-\alpha,2\}}{2}}+C\frac{t^\alpha\log n}{n},\end{equation}
    with probability $1-\O(n^{-10})$, where 
    \begin{equation} 
        \zeta_i(\cdot)=\varphi_t(T_X)k_{x_i}(x)-\varphi_t(T)k_{x_i}(\cdot),
    \end{equation} 
    and the constant $C>0$ depends only on $\varepsilon$, $d$, $\kappa$, $\beta$, $s$, $R$, $\sigma$, $L$, $h$ and $L_k$ for continuous kernel gradient flow, and on $\eta$ additionally for discrete kernel gradient flow. 
\end{theorem} 

The proof of this theorem follows exactly the same arguments as that of Theorem \ref{second-order estimation theorem}; the only difference is that, the proof of Theorem \ref{second-order estimation theorem} is based on the estimation for $\frac{1}{n}\sum_{i=1}^n\zeta_i(\cdot)\varepsilon_i$, and one needs to replace $\varepsilon_i$ with $\varepsilon_ig_i$ throughout the proof of Theorem \ref{second-order estimation theorem}. For brevity, the proof of Theorem \ref{second-order estimation theorem 2} is omitted.

\subsection{Estimations for Kolmogorov distances}\label{kolmogorov distance section} 

Before we begin the proof of Theorem \ref{main result}, we first recall some important quantities. 

Recall that the function $C_t(x,x')$ defined  in (\ref{kernel of GP}) is 
\begin{equation}C_t(x,x')=\sigma^2\cdot\E_{z\sim\mu}(\varphi_t(T)k_x(z)\cdot\varphi_t(T)k_{x'}(z))=\sigma^2\int_\X\varphi_t(T)k_x(z)\cdot\varphi_t(T)k_{x'}(z)d\mu(z),\end{equation} 
and an empirical estimation of $C_t(x,x)$ given by (\ref{empirical kernel of GP}) is: 
\begin{equation}\widehat{C}_{n,t}(x,x)=\frac{1}{n}\sum_{j=1}^n|\varphi_t(T_X)k_x(x_j)\hat{\varepsilon}_j|^2,\end{equation}
where $\hat\varepsilon_j=y_j-\hat{f}_t(x_j)$. It will be shown in Lemma \ref{C hat and C} that $\widehat{C}_{n,t}(x,x)$ is a good estimation for $C_t(x,x)$. 

Recall that $W_t(x)$ is a Gaussian process indexed by $x\in\X$ defined by 
\begin{equation}W_t(x)\sim\mathrm{GP}\left(0,\frac{C_t(x,x')}{\sqrt{C_t(x,x)C_t(x',x')}}\right),\quad Z_t=\|W_t(\cdot)\|_\infty,\end{equation} 
and $\widetilde{Z}_{n,t}$ is defined by 
\begin{equation}\widetilde{W}_{n,t}(x)=\frac{1}{\sqrt{C_t(x,x)}}\cdot\sqrt{n}(\hat{f}_t(x)-f_t(x)),\quad\widetilde{Z}_{n,t}=\|\widetilde{W}_{n,t}(\cdot)\|_\infty.\end{equation} 
Recall that the multiplier bootstrap variable $\widehat{Z}_{n,t}$ defined in (\ref{W hat definition}) is
\begin{equation}
    \widehat{W}_{n,t}(x)=\frac{1}{\sqrt{\widehat{C}_{n,t}(x,x)}}\cdot\frac{1}{\sqrt{n}}\sum_{j=1}^n\varphi_t(T_X)k_x(x_j)\hat{\varepsilon}_j g_j,\quad \widehat{Z}_{n,t}=\|\widehat{W}_{n,t}(\cdot)\|_\infty
\end{equation}
where $g=(g_1,\dots,g_n)^T\sim N(0,I_n)$ is a standard $n$-dimensional Gaussian random variable. 

The proof of Theorem \ref{main result} is divided into the following three estimations: 
\begin{equation}\sqrt{n}\left\|\frac{\hat{f}_t(x)-f^*(x)}{\sqrt{\widehat{C}(x,x)}}\right\|_\infty\,\mathop{\approx}\limits^1\,\widetilde{Z}_{n,t}\,\mathop{\approx}\limits^2\,Z_t\,\mathop{\approx}\limits^3\,\widehat{Z}_{n,t}|\D_n.\end{equation}

\begin{lemma}\label{approximation 1} 
    \textnormal{(Theorem \ref{gaussian approximation for Z tilde})} For $t=n^\theta$, $\theta\in(\frac{1}{s},\beta)$, we have 
    \begin{equation}\sup_{a\in\R}\left|\P(\widetilde{Z}_{n,t}\leq a)-\P(Z_t\leq a)\right|\leq c_1 n^{-c_2}\end{equation} 
    for some constants $c_1,c_2>0$. 
\end{lemma} 

\proof 
Define 
\begin{equation}\widetilde{W}_{n,t}^0(x)=\frac{1}{\sqrt{C_t(x,x)}}\cdot\frac{1}{\sqrt{n}}\sum_{i=1}^n\varphi_t(T)k_x(x_i)\cdot\varepsilon_i,\quad\widetilde{Z}_{n,t}^0=\|\widetilde{W}_{n,t}^0(\cdot)\|_\infty.\end{equation} 

Then the proof of this lemma is divided into the following two parts: we first establish an upper bound for the Kolmogorov distance between $\widehat{Z}_{n,t}^0$ and $Z_t$, and then we show that the second-order estimation in Theorem \ref{Bahadur representation} implies an upper bound for the distance between $\widetilde{Z}_{n,t}$ and $\widetilde{Z}_{n,t}^0$. 

By Lemma \ref{vc lemma}, the function class \begin{equation}\mathcal{F}=\left\{f_z(x,\varepsilon)=C_n(z,z)^{-\frac{1}{2}}\varphi_t(T)k_z(x)\cdot\varepsilon:\,z\in\X\right\}\end{equation} 
is a VC-type class with envelope $F(x,\varepsilon)=C_1 t^{-\frac{1}{2\beta}+\alpha}|\varepsilon|$, $A=C_2 t^{\frac{3\alpha+1}{2}-\frac{1}{\beta}}$ and $v=\frac{d}{h}$, where the constants $C_1,C_2>0$ depends only on $\varepsilon$, $d$, $\sigma$, $\beta$, $h$ and $L_k$ (and $\kappa, \eta$ additionally for discrete kernel gradient flow). (The definition of VC-type class is given in Definition \ref{vc definition}.) 

Denote $S=\X\times\R$. Recall that $\mu$ is the marginal distribution of $\rho$ on $\X$. Denote as $\nu(\varepsilon|x)$ the distribution of the noise $\varepsilon=y-f^*(x)$ conditioning on $x$. By Assumption \ref{noise moment}, for any $q\in[4,\infty)$, we have 
\begin{equation}\|F(x,\varepsilon)\|_{L^q(S)}\leq C_1 t^{-\frac{1}{2\beta}+\alpha}\cdot\left(\frac{1}{2}q!\sigma^2L^{q-2}\right)^{\frac{1}{q}}.\end{equation}
For any $f_z\in\mathcal{F}$ and $m=2,3$, by Assumption \ref{noise moment}, Assumption \ref{lower bound of C} and Lemma \ref{operated basis function norm bound}, we have 
\begin{equation}\begin{aligned} 
    &\int_S\left(\frac{\varphi_t(T)k_z(x)}{\sqrt{C_t(z,z)}}\varepsilon\right)^m d\nu(\varepsilon|x)d\mu(x)\\
    \leq&\frac{1}{2}m!\sigma^2L^{m-2}\cdot t^{-\frac{m}{2\beta}}\cdot\|\varphi_t(T)k_z(\cdot)\|_\infty^{m-2}\cdot\int_\X(\varphi_t(T)k_z(x))^2dx\\
    &\leq\frac{1}{2}m!\sigma^2L^{m-2}\cdot t^{-\frac{m}{2\beta}}\cdot(M_\alpha^2 t^{\alpha})^{m-2}\cdot M_\alpha^2 t^\alpha\\
    &=C\sigma^2\cdot(t^{-\frac{1}{2\beta}+\alpha})^{m-2}\cdot(t^{-\frac{1}{2\beta}+\frac{\alpha}{2}})^2,
\end{aligned}\end{equation}
where $C>0$ depends only on $\varepsilon$ and $L$ (and $\kappa$, $\eta$ additionally for discrete kernel gradient flow). 

Therefore, we can choose $b=C' t^{-\frac{1}{2\beta}+\alpha}$ with $C'$ great enough such that 
\begin{equation}\label{gaussian approximation condition verified} 
    \|F(x,\varepsilon)\|_{L^q(S)}\leq b,\quad\mbox{and}\quad\int_S\left(\frac{\varphi_t(T)k_z(x)}{\sqrt{C_t(z,z)}}\varepsilon\right)^m d\nu(\varepsilon|x)d\mu(x)\leq(t^{-\frac{1}{2\beta}+\frac{\alpha}{2}}\sigma)^2b^{m-2}.
\end{equation} 

Applying Corollary 2.2 of \citep{anti-concentration_b}, we obtain that for any $\gamma\in(0,1)$, $q\in[4,\infty)$, 
\begin{equation}\P\left(\left|\widetilde{Z}_{n,t}^0-Z_t\right|>\frac{bK_n}{\gamma^{\frac{1}{2}}n^{\frac{1}{2}-\frac{1}{q}}}+\frac{(b\tilde{\sigma})^\frac{1}{2}K_n^{\frac{3}{4}}}{\gamma^{\frac{1}{2}}n^{\frac{1}{4}}}+\frac{(b\tilde\sigma K_n^2)^{\frac{1}{3}}}{\gamma^{\frac{1}{3}}n^{\frac{1}{6}}}\right)\geq C\left(\gamma+\frac{\log n}{n}\right),\end{equation}
where $K_n=cv(\log n\vee\log(Ab/\tilde\sigma))$, $\tilde\sigma=t^{-\frac{1}{2\beta}+\frac{\alpha}{2}}\sigma$, and $c,C>0$ depends only on $q$, $\varepsilon$, $d$, $\beta$, $\sigma$, $L$, $h$ and $L_k$ (and $\kappa$, $\eta$ additionally for discrete kernel gradient flow).  

Taking $\gamma=(n t^{\frac{1}{\beta}-2\alpha})^{-\frac{1}{8}}$, then by computation, we obtain that 
\begin{equation}\frac{bK_n}{\gamma^{\frac{1}{2}}n^{\frac{1}{2}-\frac{1}{q}}}+\frac{(b\tilde{\sigma})^\frac{1}{2}K_n^{\frac{3}{4}}}{\gamma^{\frac{1}{2}}n^{\frac{1}{4}}}+\frac{(b\tilde\sigma K_n^2)^{\frac{1}{3}}}{\gamma^{\frac{1}{3}}n^{\frac{1}{6}}}\leq\Delta_1,\end{equation}
where
\begin{equation}\label{Delta 1}
    \Delta_1=C\frac{(\log n)^{\frac{2}{3}}}{(n t^{\frac{1}{\beta}-2\alpha})^\frac{1}{8}\cdot t^{\frac{1}{6}(\frac{1}{\beta}-\alpha)}}.
\end{equation} 
Thus, 
\begin{equation}\label{Z0 tilde and Z}
    \P\left(\left|\widetilde{Z}_{n,t}^0-Z_t\right|>\Delta_1\right)\leq r_1,
\end{equation} 
where 
\begin{equation}\label{r 1}
    r_1=C\left((n t^{\frac{1}{\beta}-2\alpha})^{-\frac{1}{8}}+\frac{\log n}{n}\right).
\end{equation} 
For any $a\in\R$, by (\ref{Z0 tilde and Z}), we have 
\begin{equation}\begin{aligned} 
    \P\left(\widetilde{Z}_{n,t}^0\leq a\right)&=\P\left(\widetilde{Z}_{n,t}^0\leq a,\,\left|\widetilde{Z}_{n,t}^0-Z_t\right|\leq\Delta_1\right)+\P\left(\widetilde{Z}_{n,t}^0\leq a,\,\left|\widetilde{Z}_{n,t}^0-Z_t\right|>\Delta_1\right)\\
    &\leq\P\left(Z_t\leq a+\Delta_1\right)+r_1\\
    &\leq\P\left(Z_t\leq a\right)+\P\left(\left|Z_t-a\right|<\Delta_1\right)+r_1. 
\end{aligned}\end{equation} 
By Lemma \ref{expectation of gaussian process} and Theorem \ref{anti-concentration inequality}, we have 
\begin{equation}\P\left(\left|Z_n-a\right|<\Delta_1\right)\leq C\Delta_1\sqrt{\log n},\end{equation} 
hence 
\begin{equation}\P\left(\widetilde{Z}_{n,t}^0\leq a\right)-\P\left(Z_t\leq a\right)\leq C\Delta_1\sqrt{\log n}+r_1.\end{equation} 
Similarly, we also obtain that 
\begin{equation}\P\left(Z_t\leq a\right)-\P\left(\widetilde{Z}_{n,t}^0\leq a\right)\leq C\Delta_1\sqrt{\log n}+r_1.\end{equation} 
Therefore, we have the following estimation for the Kolmogorov distance between $\widehat{Z}_n^0$ and $Z_n$: 
\begin{equation}\label{Kolmogorov distance between Z0 hat and Z}
    \sup_{a\in\R}\left|\P\left(\widetilde{Z}_{n,t}^0\leq a\right)-\P\left(Z_t\leq a\right)\right|\leq C\Delta_1\sqrt{\log n}+r_1.
\end{equation} 

By Assumption \ref{lower bound of C} and (\ref{second-order estimation}), we have 
\begin{equation}\P\left(\left|\widetilde{Z}_{n,t}-\widetilde{Z}_{n,t}^0\right|>\Delta_2\right)\leq\O(n^{-10}),\end{equation} 
where 
\begin{equation}\Delta_2=\sqrt{\frac{n}{t^{\frac{1}{\beta}}}}\cdot\left(C\sqrt{\frac{t^\alpha\log n}{n}}\cdot t^{-\frac{\min\{s-\alpha,2\}}{2}}+C\frac{t^\alpha\log n}{n}\right).\end{equation} 

Then, for any $a\in\R$, by (\ref{Kolmogorov distance between Z0 hat and Z}) we have 
\begin{equation}\label{estimation combination}
    \begin{aligned} 
        \P\left(\widetilde{Z}_{n,t}\leq a\right)&=\P\left(\widetilde{Z}_{n,t}\leq a,\,\left|\widetilde{Z}_{n,t}-\widetilde{Z}_{n,t}^0\right|\leq\Delta_2\right)+\P\left(\widetilde{Z}_{n,t}\leq a,\,\left|\widetilde{Z}_{n,t}-\widetilde{Z}_{n,t}^0\right|>\Delta_2\right)\\
        &\leq\P\left(\widetilde{Z}_{n,t}^0\leq a+\Delta_2\right)+\O(n^{-10})\\
        &\leq\P\left(\widetilde{Z}_{n,t}^0\leq a+\Delta_2\right)-\P\left(Z_t\leq a+\Delta_2\right)+\P\left(Z_t\leq a+\Delta_2\right)+\O(n^{-10})\\
        &\leq C\Delta_1\sqrt{\log n}+r_1+\P\left(Z_t\leq a+\Delta_2\right)+\O(n^{-10}).
    \end{aligned}
\end{equation}
By Lemma \ref{expectation of gaussian process} and Theorem \ref{anti-concentration inequality}, we have 
\begin{equation}\P\left(Z_t\leq a+\Delta_2\right)\leq\P(Z_t\leq a)+\P\left(\left|Z_t-a\right|\leq\Delta_2\right)\leq\P(Z_t\leq a)+C\Delta_2\sqrt{\log n}.\end{equation}
Thus, 
\begin{equation}\P\left(\widetilde{Z}_{n,t}\leq a\right)\leq\P(Z_n\leq a)+C(\Delta_1+\Delta_2)\sqrt{\log n}+r_1+\O(n^{-10}).\end{equation}
Similarly, we also obtain that 
\begin{equation}\P(Z_t\leq a)\leq\P\left(\widetilde{Z}_{n,t}\leq a\right)+C(\Delta_1+\Delta_2)\sqrt{\log n}+r_1+\O(n^{-10}).\end{equation}
Therefore, the Kolmogorov distance between $\widetilde{Z}_{n,t}$ and $Z_t$ is bounded by 
\begin{equation}\sup_{a\in\R}\left|\P\left(\widetilde{Z}_{n,t}\leq a\right)-\P(Z_t\leq a)\right|\leq C(\Delta_1+\Delta_2)\sqrt{\log n}+r_1+\O(n^{-10})\leq c_1n^{-c_2}\end{equation} 
for some $c_1,c_2>0$. 
\qed

\begin{lemma}\label{approximation 2} 
    \textnormal{(Theorem \ref{gaussian approximation of Z hat})} For $t=n^\theta$, $\theta\in(\frac{1}{s},\beta)$, with probability at least $1-p_n$, we have 
    \begin{equation}\sup_{a\in\R}\left|\P\left(\left.\widehat{Z}_{n,t}\leq a\right|\D_n\right)-\P(Z_t\leq a)\right|\leq q_n,\end{equation} 
    where $p_n=c_1n^{-c_2}$, $q_n=c_3n^{-c_4}$ for some $c_i>0$, $i=1,2,3,4$ depending only on $\theta$, $d$, $\beta$, $\kappa$, $s$, $R$, $h$, $L_k$, $\sigma$ and $L$ (and $\eta$ additionally for discrete kernel gradient flow). 
\end{lemma} 

\proof 
We first define a random variable $Z_t^e$ which is identically distributed with $Z_t$ conditioning on the data $\D_n$: 
\begin{equation}Z_t^e|\D_n\mathop{=}\limits^{d}Z_t.\end{equation} 
We further define 
\begin{equation}\widehat{W}_{n,t}^0(x)=\frac{1}{\sqrt{C_t(x,x)}}\cdot\frac{1}{\sqrt{n}}\sum_{j=1}^n\varphi_t(T)k_x(x_j)\varepsilon_j g_j,\quad \widehat{Z}_{n,t}^0=\|\widehat{W}_{n,t}^0\|_\infty.\end{equation} 
The proof of this lemma is based on the following two estimations: 
\begin{equation}\widehat{Z}_{n,t}|\D_n\,\mathop{\approx}\limits^{(A)}\,\widehat{Z}_{n,t}^0|\D_n\,\mathop{\approx}\limits^{(B)}\,Z_{t}^e|\D_n\,=\,Z_t.\end{equation}

\textit{Estimation (B).} By Theorem 2.2 in \cite{anti-concentration_c}, and combining with Lemma \ref{vc lemma}, for any $\gamma\in(0,1)$, we have 
\begin{equation}\P\left(\left|\|\widehat{W}_{n,t}^0-\mathcal{Z}_{n,t}\|_\infty-Z_t^e\right|>\frac{bK_n}{\gamma^{1+\frac{1}{q}}n^{\frac{1}{2}-\frac{1}{q}}}+\frac{(b\tilde{\sigma})^{\frac{1}{2}}K_n^{\frac{3}{4}}}{\gamma^{1+\frac{1}{q}}n^{\frac{1}{4}}}\right)<C(\gamma+\frac{1}{n}),\end{equation}
where $K_n=cv(\log n\vee\log(Ab/\tilde\sigma))$, $\tilde\sigma=t^{-\frac{1}{2\beta}+\frac{\alpha}{2}}\sigma$, and 
\begin{equation} 
    \mathcal{Z}_{n,t}=\frac{1}{n^{3/2}}\sum_{j,k=1}^n\varphi_t(T)k_x(X_j)\varepsilon_jg_k=\frac{1}{\sqrt{n}}\sum_{k=1}^n g_k\cdot\frac{1}{n}\sum_{j=1}^n \varphi_t(T)k_x(X_j)\varepsilon_j.
\end{equation}
By the tail approximation of Gaussian distribution, we obtain that 
\begin{equation} 
    \P\left(\left|\frac{1}{\sqrt{n}}\sum_{k=1}^n g_k\right|\geq M\right)=\P\left(|g_1|\geq{M}\right)\leq \frac{2e^{-M^2/2}}{\sqrt{2\pi} M}\leq n^{-3/2}, 
\end{equation} 
where $M=\sqrt{3\log n}$. Moreover, similarly with Theorem \ref{variance estimation}, we can prove that
\begin{equation}\label{similar with variance estimation}
    \P\left(\sup_{x\in\X}\left|\frac{1}{n}\sum_{j=1}^n \varphi_t(T)k_x(X_j)\varepsilon_j\right|>\sqrt{\frac{\log n}{n}}t^{\frac{\alpha}{2}}\right)=\O(n^{-10}).
\end{equation} 
(Recall that Theorem \ref{variance estimation} provides an estimate $\frac{1}{n}\sum_{j=1}^n \varphi_t(T_X)k_x(X_j)\varepsilon_j$, and the proof of (\ref{similar with variance estimation}) follows directly from that of Theorem \ref{variance estimation} with $T_X$ replaced by $T$). Then we have 
\begin{equation} 
    \P\left(\|\mathcal{Z}_{n,t}\|_\infty>\frac{\sqrt{3}\log n}{\sqrt{n}}t^{\frac{\alpha}{2}}\right)=o(n^{-3/2})
\end{equation} 
Therefore, 
\begin{equation}
    \P\left(\left|\widehat{Z}_{n,t}^0-Z_t^e\right|>\frac{bK_n}{\gamma^{1+\frac{1}{q}}n^{\frac{1}{2}-\frac{1}{q}}}+\frac{(b\tilde{\sigma})^{\frac{1}{2}}K_n^{\frac{3}{4}}}{\gamma^{1+\frac{1}{q}}n^{\frac{1}{4}}}+\frac{\sqrt{3}\log n}{\sqrt{n}}t^{\frac{\alpha}{2}}\right)<C(\gamma+\frac{1}{n}).
\end{equation}

By taking $\gamma=(nt^{\frac{1}{\beta}-2\alpha})^{-\frac{1}{8}}$ and $q$ great enough, we have 
\begin{equation}\frac{bK_n}{\gamma^{1+\frac{1}{q}}n^{\frac{1}{2}-\frac{1}{q}}}+\frac{(\log n)^{\frac{3}{4}}}{(n t^{\frac{1}{\beta}-2\alpha})^{\frac{1}{8}-\frac{1}{8q}}t^{\frac{1}{4}(\frac{1}{\beta}-\alpha)}}\leq C\frac{\log n}{(n t^{\frac{1}{\beta}-2\alpha})^{\frac{1}{10}}\cdot t^{\frac{1}{4}(\frac{1}{\beta}-\alpha)}},\end{equation}
hence we have 
\begin{equation}\P\left(\left|\widehat{Z}_{n,t}^0-Z_t^e\right|\geq\Delta_3\right)<r_2,\end{equation} 
where 
\begin{equation}\Delta_3=C\frac{\log n}{(n t^{\frac{1}{\beta}-2\alpha})^{\frac{1}{10}}\cdot t^{\frac{1}{4}(\frac{1}{\beta}-\alpha)}}+\frac{\sqrt{3}\log n}{\sqrt{n}}t^{\frac{\alpha}{2}},\quad r_2=C((nt^{\frac{1}{\beta}-2\alpha})^{-\frac{1}{8}}+\frac{1}{n}).\end{equation} 
By Markov's inequality, for any $\eta\in(0,1)$, with probability at least $1-\eta$ (the randomness comes from $\D_n$), we have 
\begin{equation}\P\left(\left.\left|\widehat{Z}_n^0-Z_n^e\right|\geq\Delta_3\right|\D_n\right)<\frac{r_2}{\eta},\end{equation} 
and then, using the same discussion in the proof of Lemma \ref{approximation 1}, by Lemma \ref{expectation of gaussian process} and Theorem \ref{anti-concentration inequality}, we have 
\begin{equation}\label{estimation A} 
    \sup_{a\in\R}\left|\P\left(\left.\widehat{Z}_{n,t}^0\leq a\right|\D_n\right)-\P\left(\left.Z_t^e\leq a\right|\D_n\right)\right|\leq\frac{r_2}{\eta}+C\Delta_3\sqrt{\log n},
\end{equation}
where we can choose $\eta$ properly and $\alpha=\frac{1}{\beta}+\varepsilon$ sufficiently close to $\frac{1}{\beta}$ (depending on $\theta$) such that 
\begin{equation}\eta\leq c_1n^{-c_2},\quad\frac{r_2}{\eta}+C\Delta_3\sqrt{\log n}\leq c_3n^{-c^4}\end{equation}
for some $c_1,c_2,c_3,c_4>0$. 

\textit{Estimation (A).} By Assumption \ref{lower bound of C} and Lemma \ref{C hat and C}, with probability $1-\O(n^{-10})$, we have $\widehat{C}_{n,t}(x,x)^{-\frac{1}{2}}\leq Ct^{-\frac{1}{2\beta}}$. Then, By Assumption \ref{lower bound of C}, Lemma \ref{variance estimation}, Lemma \ref{epsilon hat and epsilon} and Lemma \ref{C hat and C}, with probability $1-\O(n^{-10})$, we have 
\begin{equation}\begin{aligned} 
    \left|\widehat{Z}_{n,t}-\widehat{Z}_{n,t}^0\right|&\leq\sqrt{n}\left\|\frac{\sum_{i=1}^n\varphi_t(T_X)k_x(x_i)\varepsilon_i g_i}{n\widehat{C}_n(x,x)^{\frac{1}{2}}}-\frac{\sum_{i=1}^n\varphi_t(T)k_x(x_i)\varepsilon_i g_i}{nC_n(x,x)^{\frac{1}{2}}}\right\|_\infty\\
    &\leq\sqrt{n}\left\|\frac{\sum_{i=1}^n\varphi_t(T_X)k_x(x_i)\varepsilon_i g_i-\sum_{i=1}^n\varphi_t(T)k_x(x_i)\varepsilon_i g_i}{n\widehat{C}_n(x,x)^{\frac{1}{2}}}\right\|_\infty\\
    &\quad +\sqrt{n}\left\|\frac{1}{n}\sum_{i=1}^n\varphi_t(T)k_x(x_i)\varepsilon_i g_i\cdot\frac{C_n(x,x)-\widehat{C}_n(x,x)}{C_n(x,x)^{\frac{1}{2}}\widehat{C}_n(x,x)^{\frac{1}{2}}(C_n(x,x)^{\frac{1}{2}}+\widehat{C}_n(x,x)^{\frac{1}{2}})}\right\|_\infty\\
    &\leq\sqrt{n}\cdot(Ct^{-\frac{\min\{s-\alpha,2\}}{2}}\sqrt{\frac{t^\alpha\log n}{n}}+C\frac{t^\alpha\log n}{n})\cdot Ct^{-\frac{1}{2\beta}}\\
    &\quad+\sqrt{n}\cdot\sqrt{\frac{t^\alpha\log n}{n}}\cdot o(t^{\frac{1}{\beta}})\cdot Ct^{-\frac{3}{2\beta}}\\
    &=:\Delta_4. 
\end{aligned}\end{equation} 
In other words, 
\begin{equation}\P\left(\left|\widehat{Z}_n-\widehat{Z}_n^0\right|>\Delta_4\right)<\O(n^{-10}).\end{equation} 
By Markov's inequality, for any $\eta\in(0,1)$, with probability at least $1-\O(n^{-5})$, we have 
\begin{equation}\label{estimation B}
    \P\left(\left.\left|\widehat{Z}_n-\widehat{Z}_n^0\right|>\Delta_4\right|\D_n\right)<\O(n^{-5}).
\end{equation}

~ 

Now we have two estimations (\ref{estimation A}) and (\ref{estimation B}). Then, using the same argument as in (\ref{estimation combination}), by Lemma \ref{expectation of gaussian process} and Theorem \ref{anti-concentration inequality}, with probability $1-p_n$, we have 
\begin{equation}\sup_{a\in\R}\left|\P\left(\left.\widehat{Z}_{n,t}\leq a\right|\D_n\right)-\P\left(Z_t\leq a\right)\right|=\sup_{a\in\R}\left|\P\left(\left.\widehat{Z}_{n,t}\leq a\right|\D_n\right)-\P\left(\left.Z_t^e\leq a\right|\D_n\right)\right|\leq q_n\end{equation} 
for some $p_n=c_5n^{-c_6}$, $q_n=c_7n^{-c^8}$. 
\qed 

~

\textbf{Final proof of Theorem \ref{main result}}. By Lemma \ref{2nd term of bias}, we have 
\begin{equation}\|f_t-f^*\|_\infty\leq Ct^{-\frac{s-\alpha}{2}},\end{equation} 
and recall that by Assumption \ref{lower bound of C} and Lemma \ref{C hat and C}, we have $\widehat{C}_{n,t}(x,x)^{-\frac{1}{2}}\leq Ct^{-\frac{1}{2\beta}}$. Then we obtain that 
\begin{equation}\begin{aligned} 
    \P\left(\sqrt{n}\cdot\left\|\frac{\hat{f}_t(x)-f^*(x)}{\widehat{C}_{n,t}(x,x)^{\frac{1}{2}}}\right\|_\infty\leq a\right)&=\P\left(\sqrt{n}\cdot\left\|\frac{\hat{f}_t(x)-f_t(x)+f_t(x)-f^*(x)}{\widehat{C}_{n,t}(x,x)^{\frac{1}{2}}}\right\|_\infty\leq a\right)\\
    &\leq\P\left(\widetilde{Z}_{n,t}\leq a+C\sqrt{n t^{-\frac{1}{\beta}}}t^{\frac{s-\alpha}{2}}\right). 
\end{aligned}\end{equation}
Thus, 
\begin{equation}\label{decomposition into P1 and P2}
    \P\left(\sqrt{n}\cdot\left\|\frac{\hat{f}_t(x)-f^*(x)}{\widehat{C}_{n,t}(x,x)^{\frac{1}{2}}}\right\|_\infty\leq a\right)-\P\left(\left.\widehat{Z}_{n,t}\leq a\right|\D_n\right)\leq P_1(a)+P_2(a),
\end{equation} 
where 
\begin{equation}P_1(a)=\P\left(\widetilde{Z}_{n,t}\leq a+C\sqrt{n t^{-\frac{1}{\beta}}}t^{\frac{s-\alpha}{2}}\right)-\P(Z_t\leq a),\end{equation} 
\begin{equation}P_2(a)=\P(Z_t\leq a)-\P\left(\left.\widehat{Z}_{n,t}\leq a\right|\D_n\right).\end{equation} 
By Lemma \ref{approximation 1} and Lemma \ref{approximation 2}, with probability at least $1-p_n$, we have 
\begin{equation}\label{P1}
    \sup_{a\in\R}P_2(a)\leq q_n,
\end{equation}
where $p_n=c_1n^{-c_2}$, $q_n=c_3n^{-c_4}$ for some $c_i>0$ ,$i=1,2,3,4$. 

In order to estimate $P_1(t)$, we first note that 
\begin{equation}\begin{aligned} 
    P_1(t)&=\P\left(\widetilde{Z}_{n,t}\leq a+C\sqrt{n t^{-\frac{1}{\beta}}}t^{\frac{s-\alpha}{2}}\right)-\P\left(Z_t\leq a+C\sqrt{n t^{-\frac{1}{\beta}}}t^{\frac{s-\alpha}{2}}\right)\\
    &\quad+\P\left(Z_t\leq a+C\sqrt{n t^{-\frac{1}{\beta}}}t^{\frac{s-\alpha}{2}}\right)-\P(Z_t\leq a)\\
    &\leq r_n+\P\left(a\leq Z_t\leq a+C\sqrt{n t^{-\frac{1}{\beta}}}t^{\frac{s-\alpha}{2}}\right)\\
    &\leq r_n+\P\left(|Z_t-a|\leq C\sqrt{n t^{-\frac{1}{\beta}}}t^{\frac{s-\alpha}{2}}\right)\\
    &\leq r_n+C\sqrt{n t^{-\frac{1}{\beta}}}t^{\frac{s-\alpha}{2}}\sqrt{\log n},
\end{aligned}\end{equation}
where $r_n=c_5n^{-c_6}$ for some $c_5,c_6>0$. Here, we use Lemma \ref{approximation 1} in the first inequality, and use Lemma \ref{expectation of gaussian process} and Theorem \ref{anti-concentration inequality} in the second inequality. Thus, there exists some $c_7,c_8>0$ such that 
\begin{equation}\label{P2}
    \sup_{a\in\R}P_1(a)\leq c_7n^{-c_8}.
\end{equation} 
Combining (\ref{decomposition into P1 and P2}) with (\ref{P1}) and (\ref{P2}), we obtain that with probability at least $1-c'_1n^{-c'_2}$, 
\begin{equation}\sup_{a\in\R}\left(\P\left(\sqrt{n}\cdot\left\|\frac{\hat{f}_t(x)-f^*(x)}{\widehat{C}_{n,t}(x,x)^{\frac{1}{2}}}\right\|_\infty\leq a\right)-\P\left(\left.\widehat{Z}_{n,t}\leq a\right|\D_n\right)\right)\leq c'_1n^{-c'_2}\end{equation} 
for some $c'_i>0$, $i=1,2,3,4$. 

The proof for the lower bound is similar. In conclusion, with probability at least $1-c'_1n^{-c'_2}$, we have 
\begin{equation}\sup_{a\in\R}\left(\P\left|\sqrt{n}\cdot\left\|\frac{\hat{f}_t(x)-f^*(x)}{\widehat{C}_n(x,x)^{\frac{1}{2}}}\right\|_\infty\leq a\right|-\P\left(\left.\widehat{Z}_n\leq a\right|\D_n\right)\right)\leq c'_1n^{-c'_2},\end{equation} 
which completes the proof. 
\qed

\subsection{Auxiliary Lemmata} 

\begin{lemma}\label{Holder coefficient of C} 
    For any $\alpha=\alpha_0+\varepsilon\in(\alpha_0,1]$ and $x_1,x_2\in\X$, we have 
    \begin{equation}
        |C_t(x_1,x_1)-C_t(x_2,x_2)|\leq C t^{\frac{1+3\alpha}{2}}\cdot|x_1-x_2|^h
    \end{equation} 
    and 
    \begin{equation} 
        |C_t(x_1,x_1)-C_t(x_1,x_2)|\leq C t^{\frac{1+3\alpha}{2}}\cdot|x_1-x_2|^h
    \end{equation} 
    for some constant $C>0$ depending only on $\sigma$, $\varepsilon$ and $L_k$. 
\end{lemma} 

\proof 
For the first inequality, by Lemma \ref{operated basis function norm bound} and Lemma \ref{difference estimation}, we have 
\begin{equation}\begin{aligned} 
    &|C_t(x_1,x_1)-C_t(x_2,x_2)|\\
    \leq&\sigma^2\int_\X\left|(\varphi_t(T)k_{x_1}(x))^2-(\varphi_t(T)k_{x_2}(x))^2\right|d\mu(x)\\
    \leq&\sigma^2\int_\X\left|\varphi_t(T)k_{x_1}(x)+\varphi_t(T)k_{x_2}(x)\right|\cdot\left|\varphi_t(T)k_{x_1}(x)-\varphi_t(T)k_{x_2}(x)\right|d\mu(x)\\
    \leq&\sigma^2\cdot 2\sup_{x\in\X}\|\varphi_t(T)k_x(\cdot)\|_\infty\cdot\sup_{x\in\X}\|\varphi_t(T)k_x(\cdot)\|_\H\cdot\sqrt{L_k}|x_1-x_2|^h\\
    \leq&\sigma^2\cdot 2M_\alpha^2 t^\alpha\cdot M_\alpha t^{\frac{1+\alpha}{2}}\cdot\sqrt{L_k}|x_1-x_2|^h\\
    =&2\sigma^2M_\alpha^3\sqrt{L_k}\cdot t^{\frac{1+3\alpha}{2}}\cdot|x_1-x_2|^h.
\end{aligned}\end{equation} 
The proof for the second inequality is similar. 
\qed 

\begin{lemma}\label{vc lemma}
    The function class 
    \begin{equation}\mathcal{F}=\left\{f_z(x,\varepsilon)=C_n(z,z)^{-\frac{1}{2}}\varphi_t(T)k_z(x)\cdot\varepsilon:\,z\in\X\right\}\end{equation} 
    is a VC-type class with envelope $F(x,\varepsilon)=C_1 t^{-\frac{1}{2\beta}+\alpha}|\varepsilon|$, $A=C_2 t^{\frac{3\alpha+1}{2}-\frac{1}{\beta}}$ and $v=\frac{d}{h}$, where the constants $C_1,C_2>0$ depends only on $\varepsilon$, $d$, $\sigma$, $\beta$, $h$ and $L_k$. 
\end{lemma} 

\proof 
By Lemma \ref{operated basis function norm bound} and Assumption \ref{lower bound of C}, the function $F(x,\varepsilon)$ is clearly an envelope of $\mathcal{F}$. 

Let $S=\X\times\R$ (equipped with Borel $\sigma$-field), and suppose that $Q$ is a finite probability distribution on $S$: 
\begin{equation}Q=\sum_{m=1}^M\pi_m\delta_{(x_m,\varepsilon_m)},\end{equation} 
where $\pi_m>0$, $\pi_1+\dots+\pi_M=1$ and $\delta_{(x_m,\varepsilon_m)}$ is the Dirac measure at the point $(x_m,\varepsilon_m)$. Then by direct computation, we obtain that 
\begin{equation}\|F\|_{L^2(Q)}^2=\sum_{m=1}^M\pi_m C_1^2 t^{-\frac{1}{\beta}+2\alpha}\varepsilon_m^2\end{equation} 
and 
\begin{equation}e_Q(f_z,f_w)^2=\sum_{m=1}^M\pi_m\left(\frac{\varphi_t(T)k_z(x_m)}{C_n(z,z)^{\frac{1}{2}}}-\frac{\varphi_t(T)k_w(x_m)}{C_n(w,w)^{\frac{1}{2}}}\right)^2\varepsilon_m^2.\end{equation}
Note that 
\begin{equation}\begin{aligned} 
    &\left|\frac{\varphi_t(T)k_z(x_m)}{C_n(z,z)^{\frac{1}{2}}}-\frac{\varphi_t(T)k_w(x_m)}{C_n(w,w)^{\frac{1}{2}}}\right|\\
    \leq&\left|\frac{\varphi_t(T)k_z(x_m)}{C_n(z,z)^{\frac{1}{2}}}-\frac{\varphi_t(T)k_w(x_m)}{C_n(z,z)^{\frac{1}{2}}}\right|+\left|\frac{\varphi_t(T)k_w(x_m)}{C_n(z,z)^{\frac{1}{2}}}-\frac{\varphi_t(T)k_w(x_m)}{C_n(w,w)^{\frac{1}{2}}}\right|.
\end{aligned}\end{equation} 
By Assumption \ref{lower bound of C}, Lemma \ref{operated basis function norm bound} and Lemma \ref{difference estimation}, we have 
\begin{equation}\begin{aligned} 
    \left|\frac{\varphi_t(T)k_z(x_m)}{C_n(z,z)^{\frac{1}{2}}}-\frac{\varphi_t(T)k_w(x_m)}{C_n(z,z)^{\frac{1}{2}}}\right|&\leq c\sigma^{-1}t^{-\frac{1}{2\beta}}\cdot\sup_{z\in\X}\|\varphi_t(T)k_z(\cdot)\|_\H\cdot\sqrt{L_k}|z-w|^h\\
    &\leq c\sigma^{-1}t^{-\frac{1}{2\beta}}\cdot M_\alpha t^{\frac{1+\alpha}{2}}\cdot\sqrt{L_k}|z-w|^h
\end{aligned}\end{equation}
and by Assumption \ref{lower bound of C}, Lemma \ref{operated basis function norm bound} and Lemma \ref{Holder coefficient of C}, we have 
\begin{equation}\begin{aligned} 
    &\left|\frac{\varphi_t(T)k_w(x_m)}{C_n(z,z)^{\frac{1}{2}}}-\frac{\varphi_t(T)k_w(x_m)}{C_n(w,w)^{\frac{1}{2}}}\right|\\
    =&\left|\varphi_t(T)k_w(x_m)\frac{C_n(z,z)-C_n(w,w)}{C_n(z,z)^{\frac{1}{2}}C_n(w,w)^{\frac{1}{2}}(C_n(z,z)^{\frac{1}{2}}+C_n(w,w)^{\frac{1}{2}})}\right|\\
    \leq&M_\alpha^2 t^\alpha\cdot c\sigma^{-3}t^{-\frac{3}{2\beta}}\cdot 2\sigma^2M_\alpha^3\sqrt{L_k}\cdot t^{\frac{1+3\alpha}{2}}\cdot|z-w|^h.
\end{aligned}\end{equation} 
Then
\begin{equation}\left|\frac{\varphi_t(T)k_z(x_m)}{C_n(z,z)^{\frac{1}{2}}}-\frac{\varphi_t(T)k_w(x_m)}{C_n(w,w)^{\frac{1}{2}}}\right|\leq C\sigma^{-1}t^{-\frac{3}{2\beta}+\frac{1+5\alpha}{2}}|z-w|^h\end{equation}
for some constant $C>0$ depending only on $\varepsilon$ and $L_k$. Therefore, 
\begin{equation}e_Q(f_z,f_w)^2\leq\sum_{m=1}^M\pi_m\cdot C\sigma^{-2}t^{-\frac{3}{\beta}+1+5\alpha}|z-w|^{2h}\cdot\varepsilon_m^2.\end{equation} 
For any $\varepsilon_0>0$, we have $e_Q(f_z,f_w)\leq\varepsilon_0\|F\|_{L^2(Q)}$ if 
\begin{equation}|z-w|^{2h}\leq Ct^{-\frac{2}{\beta}+1+3\alpha}\varepsilon_0^2.\end{equation} 
Denote $\tilde{\varepsilon}_0^{2h}=Ct^{-\frac{2}{\beta}+1+3\alpha}\varepsilon_0^2$. Then there exists a $\tilde{\varepsilon}_0$ net $\X_0$ of $\X$ such that $|\X_0|\leq C(\tilde{\varepsilon}_0)^{-d}$. In other words, there exists an $\varepsilon_0$ net $\mathcal{F}_0$ of $\mathcal{F}$ with respect to $\|\cdot\|_{L^2(Q)}$ such that 
\begin{equation}|\mathcal{F}_0|\leq C(\tilde{\varepsilon}_0)^{-d}=\left(\frac{C t^{-\frac{1}{\beta}+\frac{1+3\alpha}{2}}}{\varepsilon_0}\right)^{\frac{d}{h}}.\end{equation}
This completes the proof of the lemma. 
\qed 

\begin{lemma}\label{epsilon hat and epsilon} 
    Assume $t=n^{\theta}$ for $\theta\in(\frac{1}{s},\beta)$. Then for sufficiently small $\varepsilon>0$, when $n$ is sufficiently great, with great probability $1-\O(n^{-10})$, we have 
    \begin{equation}\left\|\frac{1}{n}\sum_{i=1}^n\varphi_t(T_X)k_x(x_i)\hat\varepsilon_i g_i-\frac{1}{n}\sum_{i=1}^n\varphi_t(T)k_x(x_i)\varepsilon_i g_i\right\|_\infty\leq C\sqrt{\frac{t^\alpha\log n}{n}}t^{-\frac{\min\{s-\alpha,2\}}{2}}+C\frac{t^\alpha\log n}{n}\end{equation} 
    for some constants $C>0$ depending only on $\varepsilon$, $d$, $\kappa$, $\beta$, $s$, $R$, $\sigma$, $L$, $h$ and $L_k$ (and $\eta$ additionally for discrete kernel gradient flow), where $\alpha=\alpha_0+\varepsilon$. 
\end{lemma} 

\proof 
Note that 
\begin{equation}\label{bootstrap variable decomposition} 
    \begin{aligned} 
        &\left\|\frac{1}{n}\sum_{i=1}^n\varphi_t(T_X)k_x(x_i)\hat\varepsilon_i g_i-\frac{1}{n}\sum_{i=1}^n\varphi_t(T)k_x(x_i)\varepsilon_i g_i\right\|_\infty\\
        =&\left\|\frac{1}{n}\sum_{i=1}^n\varphi_t(T_X)k_x(x_i)\hat\varepsilon_i g_i-\frac{1}{n}\sum_{i=1}^n\varphi_t(T_X)k_x(x_i)\varepsilon_i g_i\right\|_\infty\\
        &+\left\|\frac{1}{n}\sum_{i=1}^n\varphi_t(T_X)k_x(x_i)\varepsilon_i g_i-\frac{1}{n}\sum_{i=1}^n\varphi_t(T)k_x(x_i)\varepsilon_i g_i\right\|_\infty.
    \end{aligned}
\end{equation} 
By Lemma \ref{second-order estimation theorem 2}, the second term in (\ref{bootstrap variable decomposition}) is controlled by 
\begin{equation}\label{bootstrap variable estimation 1} 
    \begin{aligned} 
        &\left\|\frac{1}{n}\sum_{i=1}^n\varphi_t(T_X)k_x(x_i)\varepsilon_i g_i-\frac{1}{n}\sum_{i=1}^n\varphi_t(T)k_x(x_i)\varepsilon_i g_i\right\|_\infty\\
        \leq& C\sqrt{\frac{t^\alpha\log n}{n}}t^{-\frac{\min\{s-\alpha,2\}}{2}}+C\frac{t^\alpha\log n}{n},
    \end{aligned}
\end{equation} 
and it remains to estimate the first term in (\ref{bootstrap variable decomposition}). 

First, we note that 
\begin{equation}|\hat{\varepsilon}_i-\varepsilon_i|=|(y_i-\hat{f}_t(x_i))-(y_i-f^*(x_i))|=|\hat{f}_t(x_i)-f^*(x_i)|\leq\|\hat{f}_t-f^*\|_\infty.\end{equation}
Recall that we choose $t=n^{-\theta}$ for $\theta\in(\frac{1}{s},\beta)$, whence the variance term dominates the upper bound of $\|\hat{f}_t-f^*\|_\infty$ by Theorem \ref{upper bound}: 
\begin{equation}\label{error estimation bound} 
    |\hat\varepsilon_i-\varepsilon_i|\leq\|\hat{f}_t-f^*\|_\infty\leq\sqrt{\frac{t^\alpha\log n}{n}}.
\end{equation} 
Define $\eta_i(x)=\frac{1}{n}\varphi_t(T_X)k_x(x_i)(\hat\varepsilon_i-\varepsilon_i)$. By (\ref{error estimation bound}) and Lemma \ref{operated basis function norm bound}, we have 
\begin{equation}|\eta_i(x)|\leq\frac{1}{n}\cdot C\sqrt{\frac{t^\alpha\log n}{n}}t^\alpha,\end{equation} 
and by (\ref{error estimation bound}) and Lemmma \ref{operated basis function empirical norm bound}, we have 
\begin{equation}\sum_{i=1}^n(\eta_i(x))^2\leq\frac{t^\alpha}{\log n}{n}.\end{equation}

Thus, by Bernstein inequality (Theorem \ref{Bernstein ineq}) and the fact that $g_i\sim N(0,1)$, we obtain that the first term of the right-hand side of (\ref{bootstrap variable decomposition}) is bounded by 
\begin{equation}\label{bootstrap variable estimation 2}
    \begin{aligned} 
        &\left\|\frac{1}{n}\sum_{i=1}^n\varphi_t(T_X)k_x(x_i)\hat\varepsilon_i g_i-\frac{1}{n}\sum_{i=1}^n\varphi_t(T_X)k_x(x_i)\varepsilon_i g_i\right\|_\infty\\
        =&\left\|\sum_{i=1}^n\eta_i(x)\cdot g_i\right\|_\infty\leq C\sqrt{\frac{t^\alpha\log n}{n}}.
    \end{aligned}
\end{equation} 

Finally, combining (\ref{bootstrap variable estimation 1}) and (\ref{bootstrap variable estimation 2}), we have 
\begin{equation}
    \begin{aligned} 
        &\left\|\frac{1}{n}\sum_{i=1}^n\varphi_t(T_X)k_x(x_i)\hat\varepsilon_i g_i-\frac{1}{n}\sum_{i=1}^n\varphi_t(T)k_x(x_i)\varepsilon_i g_i\right\|_\infty\\
        \leq& C\sqrt{\frac{t^\alpha\log n}{n}}t^{-\frac{\min\{s-\alpha,2\}}{2}}+C\frac{t^\alpha\log n}{n}.
    \end{aligned} 
\end{equation} 
\qed

\begin{lemma}\label{C hat and C} 
    When $n$ is sufficiently great, with probability $1-\O(n^{-10})$, we have 
    \begin{equation}\sup_{x\in\X}\left|\widehat{C}_{n,t}(x,x)-C_t(x,x)\right|=o(t^{\frac{1}{\beta}}),\end{equation}
    where the invention $o$ hides all the terms involving $\varepsilon$, $d$, $\kappa$, $\beta$, $s$, $R$, $\sigma$, $L$, $h$ and $L_k$ (and $\eta$ additionally for discrete kernel gradient flow). 
\end{lemma} 

\proof 
Our proof is based on the following four approximations: 
\begin{equation}\begin{aligned}
    \widehat{C}_n(x,x)&=\frac{1}{n}\sum_{i=1}^n|\varphi_t(T_X)k_x(x_i)\hat\varepsilon_i|^2\\
    &\mathop{\approx}\limits^{(A)}\,\frac{1}{n}\sum_{i=1}^n|\varphi_t(T_X)k_x(x_i)\varepsilon_i|^2\\
    &\mathop{\approx}\limits^{(B)}\,\frac{\sigma^2}{n}\sum_{i=1}^n|\varphi_t(T_X)k_x(x_i)|^2\\
    &\mathop{\approx}\limits^{(C)}\,\frac{\sigma^2}{n}\sum_{i=1}^n|\varphi_t(T)k_x(x_i)|^2\\
    &\mathop{\approx}\limits^{(D)}\,\sigma^2\|\varphi_t(T)k_x(\cdot)\|_{L^2}^2=C_n(x,x).
\end{aligned}\end{equation}

\textit{Approximation (A).} By Lemma \ref{operated basis function empirical norm bound} and estimation (\ref{error estimation bound}), for $n$ sufficiently great, with probability $1-\O(n^{-10})$, we have 
\begin{equation}\begin{aligned} 
    &\left\|\frac{1}{n}\sum_{i=1}^n|\varphi_t(T_X)k_x(x_i)\hat\varepsilon_i|^2-\frac{1}{n}\sum_{i=1}^n|\varphi_t(T_X)k_x(x_i)\varepsilon_i|^2\right\|_\infty\\
    \leq&\sup_{x\in\X}\frac{1}{n}\sum_{i=1}^n\left|\varphi_t(T_X)k_x(x_i)|^2\right|\cdot\sup_{1\leq i\leq n}\left|\hat\varepsilon_i^2-\varepsilon_i^2\right|\\
    &\leq Ct^\alpha\cdot\sqrt{\frac{t^\alpha\log n}{n}}. 
\end{aligned}\end{equation} 
Recall that we choose $\varepsilon\in(0,\min\{s-\frac{1}{\beta},\frac{1}{\theta}-\frac{1}{\beta}\})$ such that $\alpha=\frac{1}{\beta}+\varepsilon$ is sufficiently close to $\frac{1}{\beta}$, whence we have 
\begin{equation}\label{approximation A of C(x,x)}
    \sup_{x\in\X}\left|\frac{1}{n}\sum_{i=1}^n|\varphi_t(T_X)k_x(x_i)\hat\varepsilon_i|^2-\frac{1}{n}\sum_{i=1}^n|\varphi_t(T_X)k_x(x_i)\varepsilon_i|^2\right|=o(t^\frac{1}{\beta}).
\end{equation} 

\textit{Approximation (B).} The proof of approximation (B) follows the same discussion as in the truncation part of the proof of Theorem \ref{variance estimation}. We first have the following discussion: By Assumption \ref{noise moment}, Lemma \ref{operated basis function norm bound} and Lemma \ref{operated basis function empirical norm bound}, for any fixed $x\in\X$, when $n$ is sufficiently great, with probability $1-\O(n^{-10})$, we have 
\begin{equation}\sum_{i=1}^n\frac{1}{n^2}|\varphi_t(T_X)k_x(x_i)|^4\leq\frac{1}{n}\sup_{x\in\X}\|\varphi_t(T_X)k_x(\cdot)\|_\infty^2\cdot\frac{1}{n}\sum_{i=1}^n|\varphi_t(T_X)k_x(x_i)|^2\leq\frac{C}{n}t^{3\alpha}\end{equation}
and 
\begin{equation}\left|\frac{1}{n}\sum_{i=1}^n|\varphi_t(T_X)k_x(x_i)\varepsilon_i|^2-\frac{\sigma^2}{n}\sum_{i=1}^n|\varphi_t(T_X)k_x(x_i)|^2\right|_\infty\leq C\sqrt{\frac{t^\alpha\log n}{n}}\cdot t^{-\alpha}=o(t^{\frac{1}{\beta}}).\end{equation} 
Thus, using the truncation argument similar with (\ref{variance decomposition}), by Bernstein inequality (Theorem \ref{Bernstein ineq}), we obtain that when $n$ is sufficiently great, with probability $1-\O(n^{-10})$, 
\begin{equation}\label{approximation B of C(x,x)}
    \sup_{x\in\X}\left|\frac{1}{n}\sum_{i=1}^n|\varphi_t(T_X)k_x(x_i)\varepsilon_i|^2-\frac{\sigma^2}{n}\sum_{i=1}^n|\varphi_t(T_X)k_x(x_i)|^2\right|=o(t^{\frac{1}{\beta}}).
\end{equation} 

Before proving approximation (C), we first prove approximation (D). 

\textit{Approximation (D).} Let $\mathcal{F}_n$ be a $\frac{1}{n}$-net of the function family $\{\varphi_t(T)k_x(\cdot):\,x\in\X\}$. For any $\varphi_t(T)k_x(\cdot)$, denote $\xi_i=|\varphi_t(T)k_x(x_i)|^2$. By Lemma \ref{operated basis function norm bound}, we have 
\begin{equation}|\xi_i|\leq\|\varphi_t(T)k_x(\cdot)\|_\infty^2\leq M_\alpha^2\|\varphi_t(T)k_x(\cdot)\|_{[\H]^\alpha}^2\leq Ct^{2\alpha},\end{equation} 
\begin{equation}\E\xi_i=\|\varphi_t(T)k_x(\cdot)\|_{L^2}^2\leq Ct^{\alpha},\end{equation}
\begin{equation}\E|\xi_i|^2\leq\|\varphi_t(T)k_x(\cdot)\|_\infty^2\cdot\|\varphi_t(T)k_x(\cdot)\|_{L^2}^2\leq Ct^{3\alpha}.\end{equation}  
By Lemma \ref{concentration inequality bounded cases}, for any $\varepsilon>0$ and $\delta\in(0,1)$, with probability at least $1-\delta$, we have 
\begin{equation}\begin{aligned} 
    &\left|\|\varphi_t(T)k_x(\cdot)\|_{L^2,n}^2-\|\varphi_t(T)k_x(\cdot)\|_{L^2}^2\right|\\
    =&\left|\frac{1}{n}\sum_{i=1}^n\xi_i-\E\xi_i\right|\leq\varepsilon_0\cdot Ct^{3\alpha}+\frac{3+4\varepsilon_0 t^{2\alpha}}{6\varepsilon_0 n}\log\frac{2|\mathcal{F}_n|}{\delta},
\end{aligned}\end{equation} 
By choosing $\delta=n^{-10}$ and $\varepsilon_0=n^{-c_0}t^{-2\alpha}$ for $c_0\in(\theta(\frac{1}{\beta}-\alpha),1-\theta(2\alpha-\frac{1}{\beta}))$, we obtain that with probability $1-\O(n^{-10})$, 
\begin{equation}\left|\|\varphi_t(T)k_x(\cdot)\|_{L^2,n}^2-\|\varphi_t(T)k_x(\cdot)\|_{L^2}^2\right|\leq Cn^{-c_0}t^\alpha+Cn^{-(1-c_0)}t^{2\alpha}\log(n|\mathcal{F}_n|)\end{equation}
for any $\varphi_t(T)k_x(\cdot)\in\mathcal{F}_n$. 

We need to estimate the covering number $|\mathcal{F}_n|$. By Lemma \ref{difference estimation} and Lemma \ref{operated basis function empirical norm bound}, we have 
\begin{equation}\label{operated basis function holder estimation} 
    \|\varphi_t(T)k_{x_1}(\cdot)-\varphi_t(T)k_{x_2}(\cdot)\|_\infty\leq C\cdot\sup_{x\in\X}\|\varphi_t(T)k_x(\cdot)\|_\H\cdot|x_1-x_2|^h\leq Ct^{\frac{1+\alpha}{2}}|x_1-x_2|^h 
\end{equation} 
for any $x\in\X$. Moreover, it is well-known that the $\varepsilon$-covering number of a bounded domain in $\R^d$ is bounded by $C\varepsilon^{-d}$ (see section 4.8 of \citep{vershynin} for example). Thus, by (\ref{operated basis function holder estimation}), we can choose $\mathcal{F}_n$ such that $|\mathcal{F}_n|\leq C(nt^{\frac{1+\alpha}{2}})^{d/h}$. Recall that we select $t=n^{\theta}$ for some $\theta\in(\frac{1}{s},\beta)$, and $c_0\in(\theta(\frac{1}{\beta}-\alpha),1-\theta(2\alpha-\frac{1}{\beta}))$. Thus, 
\begin{equation}\label{estimation in F_n} 
    \begin{aligned}
        &\left|\frac{\sigma^2}{n}\sum_{i=1}^n|\varphi_t(T)k_x(x_i)|^2-\sigma^2\|\varphi_t(T)k_x(\cdot)\|_{L^2}^2\right|\\
        =&\sigma^2\left|\|\varphi_t(T)k_x(\cdot)\|_{L^2,n}^2-\|\varphi_t(T)k_x(\cdot)\|_{L^2}^2\right|\\
        \leq&Cn^{-c_0}t^\alpha+Cn^{-(1-c_0)}t^{2\alpha}\log(n)\\
        =&o(t^{\frac{1}{\beta}}) 
    \end{aligned}
\end{equation} 
for any $\varphi_t(T)k_x(\cdot)\in\mathcal{F}_n$. 

By the definition of $\mathcal{F}_n$, for any $x\in\X$, there exists some $z\in\X$ such that $\varphi_t(T)k_z(\cdot)\in\mathcal{F}_n$ and 
\begin{equation}\label{covering argument} 
    \|\varphi_t(T)k_x(\cdot)-\varphi_t(T)k_z(\cdot)\|_\infty\leq\frac{1}{n}=o(t^{\frac{1}{\beta}}).
\end{equation} 
Combining (\ref{estimation in F_n}) and (\ref{covering argument}), we obtain that 
\begin{equation}\label{approximation D of C(x,x)} 
    \sup_{x\in\X}\left|\frac{\sigma^2}{n}\sum_{i=1}^n|\varphi_t(T)k_x(x_i)|^2-\sigma^2\|\varphi_t(T)k_x(\cdot)\|_{L^2}^2\right|=o(t^{\frac{1}{\beta}}).
\end{equation}

\textit{Approximation (C).} Note that 
\begin{equation}\begin{aligned} 
    &\left|\frac{\sigma^2}{n}\sum_{i=1}^n|\varphi_t(T_X)k_x(x_i)|^2-\frac{\sigma^2}{n}\sum_{i=1}^n|\varphi_t(T)k_x(x_i)|^2\right|\\
    =&\sigma^2\left|\|\varphi_t(T_X)k_x(\cdot)\|_{L^2,n}^2-\|\varphi_t(T)k_x(\cdot)\|_{L^2,n}^2\right|\\
    \leq&\sigma^2\|(\varphi_t(T_X)-\varphi_t(T))k_x(\cdot)\|_{L^2,n}\cdot\left(\|(\varphi_t(T_X)-\varphi_t(T))k_x(\cdot)\|_{L^2,n}+2\|\varphi_t(T)k_x(\cdot)\|_{L^2,n}\right),
\end{aligned}\end{equation} 
then by Lemma \ref{empirical norm bound of zeta}, we have 
\begin{equation}
    \begin{aligned} 
        &\sup_{x\in\X}\left|\frac{\sigma^2}{n}\sum_{i=1}^n|\varphi_t(T_X)k_x(x_i)|^2-\frac{\sigma^2}{n}\sum_{i=1}^n|\varphi_t(T)k_x(x_i)|^2\right|\\
        \leq&\sigma^2 S\left(S+2\sup_{x\in\X}\|\varphi_t(T)k_x(\cdot)\|_{L^2,n}\right),
    \end{aligned} 
\end{equation} 
where 
\begin{equation}S=\sqrt{\frac{t^\alpha\log n}{n}}t^{\frac{\alpha}{2}}\log t.\end{equation} 
By (\ref{approximation D of C(x,x)}) and Lemma \ref{operated basis function norm bound}, with probability $1-\O(n^{-10})$, we have 
\begin{equation}\|\varphi_t(T)k_x(\cdot)\|_{L^2,n}\leq Ct^{\frac{\alpha}{2}},\end{equation} 
hence 
\begin{equation}\label{approximation C of C(x,x)}
    \sup_{x\in\X}\left|\frac{\sigma^2}{n}\sum_{i=1}^n|\varphi_t(T_X)k_x(x_i)|^2-\frac{\sigma^2}{n}\sum_{i=1}^n|\varphi_t(T)k_x(x_i)|^2\right|\leq S(S+Ct^{\frac{\alpha}{2}})=o(t^{\frac{1}{\beta}}).
\end{equation}

Finally, combining (\ref{approximation A of C(x,x)}), (\ref{approximation B of C(x,x)}), (\ref{approximation D of C(x,x)}) and (\ref{approximation C of C(x,x)}) together, we obtain that when $n$ is sufficiently great, with probability $1-\O(n^{-10})$, 
\begin{equation}\sup_{x\in\X}\left|\widehat{C}_{n,t}(x,x)-C_t(x,x)\right|=o(t^{\frac{1}{\beta}}).\end{equation} 
\qed

\begin{lemma}\label{expectation of gaussian process}
    For $t=n^\theta$, $\theta\in(0,\beta)$, we have 
    \begin{equation}\E(Z_t)\leq C\sqrt{\log n},\end{equation} 
    where the constant $C>0$ depends only on $d$, $\beta$, $\sigma$, $h$ and $L_k$. 
\end{lemma} 

\proof 
We first define a semi-metric $d(\cdot,\cdot)$ on $\X$ by 
\begin{equation}d(z_1,z_2)=\sqrt{\E\left(W_t(z_1)-W_t(z_2)\right)^2},\quad\forall z_1,z_2\in\X.\end{equation} 
Let $\mathcal{N}(\X,d,\varepsilon)$ be the $\varepsilon$-covering number of $\X$ with respect to this semi-metric. By Dudley's inequality (Corollary 2.2.8 of \citep{dudley_ineq}), we have 
\begin{equation}\label{dudley's inequality}
    \E(Z_t)\leq C\int_0^\infty\sqrt{\log\mathcal{N}(\X,d,\varepsilon)}d\varepsilon
\end{equation} 
for some universal constant $C>0$. Thus, in order to prove the lemma, it suffices to estimate the covering number $\mathcal{N}(\X,d,\varepsilon)$. 

By the definition of the Gaussian process $W_n$, the random variable $W_n(z_1)-W_n(z_2)$ has mean zero and variance $2-2C_n(z_1,z_2)C_n(z_1,z_1)^{-\frac{1}{2}}C_n(z_2,z_2)^{-\frac{1}{2}}$. Then by Assumption \ref{lower bound of C} and Lemma \ref{Holder coefficient of C}, for any $\alpha=\alpha_0+\varepsilon\in(\alpha_0,1]$, we have 
\begin{equation}\begin{aligned} 
    \frac{1}{2}\E(W_t(z_1)-W_t(z_2))^2&=1-\frac{C_t(z_1,z_2)}{C_t(z_1,z_1)^{\frac{1}{2}}C_t(z_2,z_2)^{\frac{1}{2}}}\\
    &\leq Ct^{-\frac{1}{\beta}}(C_n(z_1,z_1)^{\frac{1}{2}}C_n(z_2,z_2)^{\frac{1}{2}}-C_n(z_1,z_2))\\
    &\leq Ct^{-\frac{1}{\beta}}|C_n(z_1,z_1)^{\frac{1}{2}}C_n(z_2,z_2)^{\frac{1}{2}}-C_n(z_1,z_1)|\\
    &\quad+Ct^{-\frac{1}{\beta}}|C_n(z_1,z_1)-C_n(z_1,z_2)|\\
    &\leq Ct^{-1/\beta}\cdot t^{\frac{1+3\alpha}{2}}|z_1-z_2|^h, 
\end{aligned}\end{equation} 
where $|\cdot|$ is the Euclidean norm in $\R^d$. Again, using the fact that the $\varepsilon_0$-covering number of $(\X,|\cdot|)$ is bounded by $C\varepsilon_0^{-d}$, we obtain that 
\begin{equation}\mathcal{N}(\X,d,\varepsilon)\leq C(t^{\frac{1}{\beta}-\frac{1+3\alpha}{2}}\varepsilon^2)^{-\frac{d}{h}}.\end{equation} 
Since $\E(W_n(z))^2=1$, then the diameter of the semi-metric space $(\X,d)$ is bounded by some universal constant $C_0>0$. Thus, by Dudley's inequality (\ref{dudley's inequality}), we obtain 
\begin{equation}\begin{aligned} 
    \E(Z_t)&\leq C\int_0^{C_0}\sqrt{\log(t^{\frac{1}{\beta}-\frac{1+3\alpha}{2}}\varepsilon^2)^{-\frac{d}{h}}}d\varepsilon=C\int_0^{C_0}\sqrt{\log t+\log\varepsilon}d\varepsilon\leq C\sqrt{\log t}\leq C\sqrt{\log n}.
\end{aligned}\end{equation} 
\qed

\section{Auxiliary Results}

\subsection{Some useful bounds}

\begin{lemma}\label{fractal estimation}
    Let $\lambda>0$. Then for any $\gamma\in[0,1]$, we have 
    \begin{equation}\sup_{r\geq 0}\frac{r^\gamma}{r+\lambda}\leq\lambda^{\gamma-1}.\end{equation} 
\end{lemma}

\proof 
The result follows directly from the following estimation: 
\begin{equation}\left(\frac{r}{\lambda}\right)^\gamma\leq 1+\frac{r}{\lambda}=\frac{r+\lambda}{\lambda}.\end{equation} 
\qed 

\begin{lemma}\label{estimation of filter function 2} 
    Recall that the filter functions $\varphi_t^{con}$ and $\varphi_t^{dis}$ are defined in Definition \ref{filter function def}. There exists a universal constant $E>0$ such that for all $t>1$ and $r>0$; 
    \begin{equation}\varphi_t^{con}(r)\leq\frac{E}{r+\lambda},\quad\lambda=\frac{1}{t};\end{equation}
    There exists a constant $E>0$ depending only on $\eta$ and $\kappa$ such that for any $t>1$ and $r\in(0,\kappa^2]$, we have 
    \begin{equation} \varphi_t^{dis}(r)\leq\frac{E}{r+\lambda},\quad\lambda=\frac{1}{t}.\end{equation}
\end{lemma} 

\proof 
First, we note that for any $r>0$, 
\begin{equation}\frac{1-e^{-tr}}{r}(r+\frac{1}{t})=1-e^{-tr}+\frac{1-e^{-tr}}{tr}\leq 1+\sup_{x\in(0,\infty)}\frac{1-e^{-x}}{x}<\infty\end{equation}
and for any $r\in(0,\kappa^2]$, we have 
\begin{equation}\begin{aligned} 
    \frac{1-(1-\eta r)^{t/\eta}}{r}(r+\frac{1}{t})&=1-(1-\eta r)^{t/\eta}+\frac{1-(1-\eta r)^{t/\eta}}{rt}\\
    &\leq 1+\frac{1-(1-\eta r)^{t/\eta}}{rt}\\
    &\leq 1+\frac{1-e^{-Crt}}{rt}\\
    &\leq 1+\sup_{x\in(0,\infty)}\frac{1-e^{-Cx}}{x}<\infty,
\end{aligned}\end{equation}  
where we use the inequality $(1-\eta r)^{t/\eta}\geq e^{-Crt}$, $\forall r\in(0,\kappa^2]$, $t>0$ for some $C>0$ depending only on $\eta$ and $\kappa$. 
\qed 

\begin{lemma}\label{estimation of filter function} 
    For any $\theta\in(0,1)$ and $t\in(0,\infty)$, we have 
    \begin{equation}\sup_{r>0}\varphi_t^{con}(r)r^{\theta}\leq Ct^{1-\theta}\end{equation} 
    for some universal constant $C>0$, and 
    \begin{equation}\sup_{r\in(0,\kappa^2]}\varphi_t^{dis}(r)r^{\theta}\leq Ct^{1-\theta}\end{equation} 
    for some constant $C>0$ depending only on $\eta$ and $\kappa$. 
\end{lemma} 

\proof 
By Lemma \ref{estimation of filter function 2}, we have 
\begin{equation}(1-e^{-tr})r^{\theta-1}\leq C\frac{r^\theta}{r+\lambda},\quad (1-(1-\eta r)^{t/\eta})r^{\theta-1}\leq C\frac{r^\theta}{r+\lambda}\end{equation} 
where $\lambda=\frac{1}{t}$. By Lemma \ref{fractal estimation}, we have 
\begin{equation}\frac{r^\theta}{r+\lambda}\leq\lambda^{\theta-1}=t^{1-\theta}.\end{equation} 
In conclusion, 
\begin{equation}(1-e^{-tr})r^{\theta-1}\leq Ct^{1-\theta}\end{equation} 
for any $\theta\in(0,1)$ and $r,t>0$, and 
\begin{equation}(1-(1-\eta r)^{t/\eta})r^{\theta-1}\leq Ct^{1-\theta}\end{equation} 
for any $\theta\in(0,1)$ and $r\in(0,\kappa^2]$, $t>0$ 
\qed 

\begin{lemma}\label{exponential estimation}
    Recall that the remainder functions $\psi_t^{con}$ and $\psi_t^{dis}$ are defined in Definition \ref{filter function def}. 

    (1) For any $s>0$, let $r^*=\frac{s}{t}$ and $F_s=(\frac{s}{e})^s$. Then the function  $r\mapsto r^s\psi_t^{con}(r)$ is increasing on $r\in[0,r^*]$, is decreasing on $r\in[r^*,\infty)$, and 
    \begin{equation} 
        \sup_{r\geq 0}r^s\psi_t^{con}(r)=(r^*)^s\psi_t^{con}(r^*)=F_s t^{-s}; 
    \end{equation} 
    
    (2) For any $s>0$, let $r^*=\frac{s}{\eta(t+s)}$. Then the function $r\mapsto r^s\psi_t^{dis}(r)$, $r\in(0,\kappa^2]$ is increasing when $r<r^*$, is decreasing when $r>r^*$, and there exists a constant $F_s>0$ depending only on $s$, $\eta$ and $\kappa^2$ such that 
    \[\sup_{r\in(0,\kappa^2]}r^s\psi_t^{dis}(r)\leq F_s t^s.\] 
\end{lemma}

\proof 
The result follows directly from 
\begin{equation} 
    \frac{d}{dr}(r^se^{-tr})=(s-tr)r^{s-1}e^{-tr}
\end{equation}
and 
\begin{equation} 
    \frac{d}{dr}(r^s(1-\eta r)^{t/\eta})=(s-(t+\eta s)r)(1-\eta r)^{t-1}r^{s-1}.
\end{equation}
\qed

\begin{lemma}\label{norm bound of T} 
    For any $\gamma\in[0,\infty)$, we have 
    \begin{equation}\|T\|_{[\H]^\gamma}\leq \kappa^2.\end{equation} 
\end{lemma} 

\proof 
For any $\gamma\in[0,\infty)$, the greatest eigenvalue of $T$ is $\lambda_1$, which is computed by 
\begin{equation}\lambda_1=\int_\X\lambda_1e_1(x)^2dx\leq\int_\X\sum_{j=1}^\infty\lambda_je_j(x)^2dx=\int_\X k(x,x)dx\leq\kappa^2.\end{equation} 
\qed 

\begin{lemma}\label{effect dimension bound} 
    The effect dimension of $\H$ \citep{optimal_rates_krr}, which is defined by 
    \begin{equation}\mathcal{N}_1(\lambda)=\sum_{i=1}^\infty\frac{\lambda_i}{\lambda_i+\lambda},\end{equation} 
    satisfies 
    \begin{equation}C_1\lambda^{-\frac{1}{\beta}}\leq\mathcal{N}_1(\lambda)\leq C_2\lambda^{-\frac{1}{\beta}}\end{equation} 
    for some universal constants $C_1,C_2>0$. 
\end{lemma} 

\proof 
By Assumption \ref{edr}, there exist constants $c,C>0$ such that 
\begin{equation}ci^{-\beta}\leq\lambda_i\leq Ci^{-\beta},\end{equation} 
hence 
\begin{equation}
    \begin{aligned} 
        \mathcal{N}_1(\lambda)&=\sum_{i=1}^\infty\frac{\lambda_i}{\lambda_i+\lambda}\leq\sum_{i=1}^\infty\frac{Ci^{-\beta}}{Ci^{-\beta}+\lambda}=\sum_{i=1}^\infty\frac{C}{C+\lambda i^{\beta}}\\
        &\leq \int_0^\infty\frac{C}{C+\lambda t^\beta}dt=\lambda^{-\frac{1}{\beta}}\int_0^\infty\frac{C}{C+s^\beta}ds=C_2\lambda^{-\frac{1}{\beta}}.
    \end{aligned} 
\end{equation} 
The lower bound for $\mathcal{N}_1(\lambda)$ is similar. 
\qed

\begin{lemma}\label{difference estimation}
    For any $f\in\H$ and $x,x'\in\X$, we have 
    \begin{equation}|f(x)-f(x')|\leq \sqrt{L_k}\|f\|_\H|x-x'|^h.\end{equation}
\end{lemma} 

\proof 
By direct computations, we obtain 
\begin{equation}|f(x)-f(x')|=|\langle f,k(x,\cdot)-k(x',\cdot)\rangle_\H|\leq\|f\|_\H\|k(x,\cdot)-k(x',\cdot)\|_\H\end{equation} 
and 
\begin{equation}\begin{aligned} 
    \|k(x,\cdot)-k(x',\cdot)\|_\H^2&=k(x,x)k(x',x')-k(x,x')^2\\
    &\leq k(x,x)|k(x',x')-k(x,x')|+k(x,x')|k(x,x)-k(x,x')|\\
    &\leq 2L_k|x-x'|^{2h}.
\end{aligned}\end{equation}
\qed 

\begin{lemma}\label{exchange of x and x_i} 
    Recall that $X=(x_1,\dots,x_n)$ is the matrix of samples. For any $x\in\X$ and $x_i$ being one of the samples, we have 
    \begin{equation}\varphi_t(T_X)k_{x_i}(x)=\varphi_t(T_X)k_x(x_i).\end{equation} 
\end{lemma} 

\proof 
Denote $K=\frac{1}{n}\K(X,X)$. By the definition of the operator $\varphi_t(T_X)$, we have the following equality: 
\begin{equation}\label{exchange of x and x_i step 1}
    \varphi_t(T_X)(k_{x_1}(x),\dots,k_{x_n}(x))^T=\varphi_t(K)(k_{x_1}(x),\dots,k_{x_n}(x))^T.
\end{equation} 
Also note that the rank of the operator $\varphi_t(T_X):\,\H\to\H$ is finite, and its co-kernel in $\H$ is contained in $\mathrm{span}\{k_{x_1}(\cdot),\dots,k_{x_n}(\cdot)\}$. Thus, for any $f\in\H$, by the reproducing property of $\H$, we have 
\begin{equation}\label{matrix representation of T_X func version} 
    \begin{aligned} 
        (\varphi_t (T_X)f)(X)&=\langle\varphi_t (T_X)\K(X,\cdot),f(\cdot)\rangle_\H\\
        &=\langle\varphi_t (\frac{1}{n}\K(X,X))\K(X,\cdot),f(\cdot)\rangle_\H=\varphi_t (\frac{1}{n}\K(X,X))f(X).
    \end{aligned}
\end{equation} 
By taking $f(\cdot)=k_x(\cdot)$, we obtain that 
\begin{equation}\label{exchange of x and x_i step 2}
    \begin{aligned} 
        (\varphi_t(T_X)k_x(x_1),\dots,\varphi_t(T_X)k_x(x_n))^T&=\varphi_t(K)(k_x(x_1),\dots,k_x(x_n))^T\\
        &=\varphi_t(K)(k_{x_1}(x),\dots,k_{x_n}(x))^T.
    \end{aligned}
\end{equation} 
Combining (\ref{exchange of x and x_i step 1}) and (\ref{exchange of x and x_i step 2}) together, we complete the proof. 
\qed 

\begin{lemma}\label{lower bound verification} 
    Suppose that $k(x,x')$ is an inner product kernel on the $d$-dimensional sphere $\mathbb{S}^d$. Then its Mercer decomposition is in the following form: 
    \begin{equation} 
        k(x,x')=\sum_{k=0}^\infty\mu_k\sum_{l=1}^{N(d,k)}Y_{k,l}(x)Y_{k,l}(x'),
    \end{equation} 
    where $Y_{k,l}$ are the spherical harmonics, and $N(d,k)=\binom{k+d}{k}-\binom{k+d-2}{k-2}$. Furthermore, if $\mu_k\asymp k^{-d\beta}$ for some $\beta>1$, then $k(x,x')$ satisfies Assumption \ref{edr} and \ref{lower bound of C}.
\end{lemma} 

\proof 
We refer to \citet{aha} for the proof of the Mercer decomposition of $k(x,x')$. 

Suppose $\mu_k\asymp k^{-d\beta}$. Denoting $\lambda_j$ as the corresponding eigenvalues of $k(x,x')$ without multiplicity. Then $\lambda_j\asymp j^{-\beta}$ follows directly from the fact that $\mu_k\asymp k^{-d\beta}$, $N(d,k)\asymp k^{d-1}$ and $\sum_{i=0}^k N(d,i)\asymp k^d$. Furthermore, combining with Lemma \ref{effect dimension bound}, we obtain that 
\begin{equation} 
    \begin{aligned} 
        C_t(x,x)&=\sigma^2\sum_{m=0}^\infty\sum_{l=1}^{N(d,m)}(1-e^{-t\mu_m})^2 Y_{m,l}(x)^2\\
        &=\sigma^2\sum_{m=0}^\infty(1-e^{-t\mu_m})^2 N(d,m)\\
        &\gtrsim\sigma^2\sum_{m=0}^\infty\left(\frac{\mu_m}{\lambda+\mu_m}\right)^2 N(d,m)\\
        &\gtrsim\sigma^2\lambda^{-\frac{1}{\beta}}=\sigma^2 t^{\frac{1}{\beta}},
    \end{aligned} 
\end{equation} 
where $\lambda=1/t$. 
\qed

\subsection{Concentration inequalities} 

\begin{lemma}\label{Bernstein ineq} 
    \textnormal{(Theorem 2.10 of \citealt{concentration})} Let $\xi_{in}$, $i=1,\dots,n$ be independent real-valued random variables. Assume that there exist constants $v>0$, $c>0$ such that $\sum_{i=1}^n\E\xi_{in}^2\leq v$ and 
    \begin{equation}\sum_{i=1}^n\E|\xi_{in}|^m\leq\frac{m!}{2}vc^{m-2},\quad\forall m\geq 3,\end{equation} 
    then 
    \begin{equation}\P\left(\sum_{i=1}^n(\xi_i-\E\xi_i)>\sqrt{2v\tau}+c\tau\right)<e^{-\tau}.\end{equation} 
\end{lemma}

\begin{lemma}\label{Bernstein ineq vec ver}
    \textnormal{(Theorem 26 of \citealt{concentration_2})} Let $H$ be a Hilbert space, and $\xi_1,\xi_2,\dots$ are i.i.d. random variables on $H$. Suppose that 
    \begin{equation}\E\|\xi_i\|_H^m\leq\frac{1}{2}m!\sigma^2 L^{m-2},\quad\forall m>2,\end{equation} 
    then for any $\delta\in(0,1)$, with probability at least $1-\delta$, we have 
    \begin{equation}\left\|\frac{1}{n}\sum_{i=1}^n\xi_i-\E\xi_1\right\|_H\leq 4\sqrt{2}\left(\frac{L}{n}+\frac{\sigma}{\sqrt{n}}\right)\log\frac{2}{\delta}.\end{equation} 
\end{lemma}

\begin{lemma}\label{Bernstein ineq operater ver}
    \textnormal{(Theorem 27 of \citealt{concentration_2} and Lemma 26 of \citealt{concentration_3})} Let $H$ be a Hilbert space. Suppose that $A_1,A_2,\dots$ are i.i.d. random variables with values in the self-adjoint Hilbert-Schmit operator space of $H$. If $\E A_i=0$, $\|A_i\|\leq L$ a.e. and there exists a self-adjoint positive semi-definite trace-class operator $V$ such that $\E A_i^2\preceq V$, then for any $\delta\in(0,1)$, with probability at least $1-\delta$, we have 
    \begin{equation}\left|\frac{1}{n}\sum_{i=1}^nA_i\right|_H\leq\frac{2L\beta}{3n}+\sqrt{\frac{2\|V\|\beta}{n}},\quad\beta=\log\frac{4\tr V}{\delta\|V\|}.\end{equation} 
\end{lemma}

\begin{lemma}\label{concentration inequality bounded cases} 
    \textnormal{\citep{concentration_4}} Let $\xi_i$, $i=1,\dots,n$ be i.i.d. random variables. Suppose that $|\xi_i|\leq M$ almost surely, $\E\xi_i=\mu$, $\mathrm{Var}(\xi_i)\leq\sigma^2$. Then, for any $\varepsilon_0>0$ and $\delta\in(0,1)$, with probability at least $1-\delta$, we have 
    \begin{equation}\left|\frac{1}{n}\sum_{i=1}^n\xi_i-\mu\right|\leq\varepsilon_0\sigma^2+\frac{3+4\varepsilon_0 M}{6\varepsilon_0 n}\log\frac{2}{\delta}.\end{equation} 
\end{lemma}

\subsection{Analytic functional calculus}\label{analytic section} 

In this section, we recall some fundamental facts of analytic functional calculus. Let $A$ be a linear operator on a Hilbert space $H$. The resolvent set $\rho(A)$ of $A$ is given by 
\begin{equation}\rho(A)=\{\lambda\in\C:\,A-\lambda\mbox{ is invertible}\}.\end{equation} 
For any $\lambda\in\rho(A)$, we denote 
\begin{equation}R_A(\lambda)=(A-\lambda)^{-1}.\end{equation} 
The spectrum of $A$ is $\sigma(A)=\C-\rho(A)$. 

As an analogue of the famous Cauchy integral formula for analytic functions, we have the following version of integral formula for analytic operators: 

\begin{theorem}\label{analytic functional estimation}
    Let $A$ be a linear operator on a Hilbert space $H$. Let $f$ be an analytic function defined on $D_f\subset\C$. Let $\Gamma\subset D_f$ be a contour surrounding $\sigma(A)$. Then we have 
    \begin{equation}\label{analytic integral forumla} 
        f(A)=-\frac{1}{2\pi i}\oint f(z)R_A(z)dz.
    \end{equation}
\end{theorem} 

\proof 
See Proposition 2.3.1 of \citep{operator_theory}. 
\qed 

In order to apply Theorem \ref{analytic functional estimation} to the analysis of the gradient flow algorithm, we need to carefully select a contour $\Gamma_t$ with index $t\in[0,\infty)$ that surrounds the spectra of the operators $T$ and $T_X$. 

We follow \citet{analytic} and select the following contour which surrounds the spectra of $T:[\H]^\gamma\to[\H]^\gamma$ and $T_X:[\H]^\gamma\to[\H]^\gamma$ for $\gamma\in[\alpha,1]$, with probability $1-\O(n^{-10})$: 

\begin{definition}\label{contour definition} 
    Define the contour $\Gamma_t$ with index $t\in(0,\infty)$ by 
    \begin{equation}\Gamma_t=\Gamma_{t,1}\cup\Gamma_{t,2}\subset\C,\end{equation} 
    where 
    \begin{equation}\Gamma_{t,1}=\{x\pm(x+\eta)i\in\C:\,x\in[-\eta,\kappa^2]\},\end{equation} 
    \begin{equation}\Gamma_{t,2}=\{z\in\C:\,|z-\kappa^2|=\kappa^2+\eta,\,\mathrm{Re}(z)\geq \kappa^2\},\end{equation} 
    and $\eta=\frac{\lambda}{2}=\frac{1}{2t}$. 
\end{definition} 

\begin{remark} 
    Note that $[0,\kappa^2]$ is surrounded by $\Gamma_t$. Since $T$ and $T_X$ are both bounded, positive semi-definite, and self-adjoint operators on $\H$, and $\|T\|_\H$, $\|T_X\|_\H\leq\kappa^2$, then $\Gamma_t$ is indeed a contour that surrounds the spectra of $T$ and $T_X$ for any $t\in(0,\infty)$. We also note that both $\varphi_t^{con}(z)$ and $\varphi_t^{dis}(z)$ (defined in Definition \ref{filter function def}) can be extended to analytic functions in the domain surrounded by $\Gamma_\lambda$. 
\end{remark}

\begin{lemma}\label{norm estimation along contour} 
    For any $\gamma\in[\alpha,1]$ and $z\in\Gamma_t$, we have 
    \begin{equation}\|T_\lambda(T-z)^{-1}\|_{[\H]^\gamma}=\|T_\lambda^{\frac{1}{2}}(T-z)^{-1}T_\lambda^{\frac{1}{2}}\|_{[\H]^\gamma}\leq D\end{equation} 

    Moreover, with probability $1-\O(n^{-10})$, if $n$ is sufficiently great, we have 
    \begin{equation}\|T_\lambda(T_X-z)^{-1}\|_{[\H]^\gamma}\leq 3D,\quad\|T_\lambda^{\frac{1}{2}}(T_X-z)^{-1}T_\lambda^{\frac{1}{2}}\|_{[\H]^\gamma}\leq 3D.\end{equation} 
\end{lemma} 

\proof 
Firstly, we note that 
\begin{equation}\|T_\lambda^{\frac{1}{2}}(T-z)^{-1}T_\lambda^{\frac{1}{2}}\|_{[\H]^\gamma}=\sup_{r\in\sigma(T)}\left|\frac{t+\lambda}{t-z}\right|\leq\sup_{t\in[0,\kappa^2]}\left|\frac{t+\lambda}{t-z}\right|.\end{equation} 
When $z=x\pm(x+\frac{\lambda}{2})i\in\Gamma_{t,1}$ where $x\in[-\frac{\lambda}{2},\kappa^2]$, we have 
\begin{equation}\sup_{t\in[0,\kappa^2]}\left|\frac{t+\lambda}{t-z}\right|^2\leq\sup_{t\in[0,\infty)}\left|\frac{t+\lambda}{t-z}\right|^2=\left\{ 
    \begin{aligned} 
        &\frac{4\lambda^2}{8x^2+4\lambda x+\lambda^2},\quad&&x\in[-\frac{1}{2}\lambda,-\frac{1}{2(2+\sqrt{2})}\lambda];\\
        &\frac{8x^2+12\lambda x+5\lambda^2}{(2x+\lambda)^2},\quad&&x>-\frac{1}{2(2+\sqrt{2})}\lambda, 
    \end{aligned} 
\right.\end{equation} 
The right-hand side reaches its maximum $8$ at $x=-\frac{\lambda}{4}$. Thus, 
\begin{equation}\sup_{t\in[0,\kappa^2]}\left|\frac{t+\lambda}{t-z}\right|\leq 2\sqrt{2},\quad\forall z\in\Gamma_{t,1}.\end{equation} 
When $z\in\Gamma_{t,2}$, we have 
\begin{equation}\sup_{t\in[0,\kappa^2]}\left|\frac{t+\lambda}{t-z}\right|\leq\sup_{t\in[0,\kappa^2]}\frac{|t+\lambda|}{\kappa^2+\frac{\lambda}{2}}\leq\sup_{t\in[0,\kappa^2]}\frac{\kappa^2+\lambda}{\kappa^2+\frac{\lambda}{2}}\leq 2.\end{equation}
In summary, we have 
\begin{equation}\|T_\lambda^{\frac{1}{2}}(T-z)^{-1}T_\lambda^{\frac{1}{2}}\|_{[\H]^\gamma}\leq D=2\sqrt{2}.\end{equation} 
Likewise, we also have 
\begin{equation}\|T_{X\lambda}^{\frac{1}{2}}(T_X-z)^{-1}T_{X\lambda}^{\frac{1}{2}}\|_{[\H]^\gamma}\leq D.\end{equation} 
Then, by Lemma \ref{operator quotient norm bound}, if $n$ is sufficiently large, then with probability $1-\O(n^{-10})$, we have 
\begin{equation}\|T_{\lambda}^{\frac{1}{2}}(T_X-z)^{-1}T_{\lambda}^{\frac{1}{2}}\|_{[\H]^\gamma}\leq\|T_\lambda^{\frac{1}{2}}T_{X\lambda}^{-\frac{1}{2}}\|_{[\H]^\gamma}\cdot\|T_{X\lambda}^{\frac{1}{2}}(T_X-z)^{-1}T_{X\lambda}^{\frac{1}{2}}\|_{[\H]^\gamma}\cdot\|T_{X\lambda}^{-\frac{1}{2}}T_\lambda^{\frac{1}{2}}\|_{[\H]^\gamma}\leq 3D.\end{equation}
\qed

\begin{lemma}\label{integral estimation along contour} 
    We have the following estimation for the filter function $\varphi_t(z)$ defined in Definition \ref{filter function def}: 
    \begin{equation}\oint_{\Gamma_t}|\varphi_t(z)dz|\leq C\log t\end{equation} 
    for some constant $C>0$ (depending only on $\kappa$ for continuous kernel gradient flow, and on $\eta$ additionally for discrete kernel gradient flow). 
\end{lemma}

\proof 
Note that by Lemma \ref{estimation of filter function 2}, we have 
\begin{equation}|\varphi_t(z)|\leq\frac{C}{|z+\lambda|},\quad \lambda=\frac{1}{t}\end{equation} 
for some constant $C>0$ (universal for continuous kernel gradient flow, and depending on $\kappa$ and $\eta$ additionally for discrete kernel gradient flow), then 
\begin{equation}\label{analytic phi integral estimation}
    \oint_{\Gamma_t}|\varphi_t(z)dz|\leq C\oint_{\Gamma_t}\frac{1}{|z+\lambda|}|dz|=C\oint_{\Gamma_{t,1}}\frac{1}{|z+\lambda|}|dz|+C\oint_{\Gamma_{t,2}}\frac{1}{|z+\lambda|}|dz|.
\end{equation} 
When $z\in\Gamma_{t,1}$, we have 
\begin{equation}\label{contour integration 1} 
    \begin{aligned} 
        \oint_{\Gamma_{t,1}}\frac{1}{|z+\lambda|}|dz|&=2\int_{-\frac{\lambda}{2}}^{\kappa^2}\frac{1}{|x+(x+\frac{\lambda}{2})i+\lambda|}\sqrt{2}dx\\
        &\leq C\int_{-\frac{\lambda}{2}}^{\kappa^2}\frac{1}{(x+\frac{\lambda}{2})+\lambda}dx\leq C\log\frac{1}{\lambda}=C\log t;
    \end{aligned}
\end{equation} 
When $z\in\Gamma_{t,2}$, we have $|z+\lambda|\geq\kappa^2$, hence 
\begin{equation}\label{contour integration 2} 
    \oint_{\Gamma_{t,2}}\frac{1}{|z+\lambda|}|dz|\leq\frac{1}{\kappa^2}|\Gamma_{t,2}|\leq C.
\end{equation} 
Combining (\ref{analytic phi integral estimation}), (\ref{contour integration 1}) and (\ref{contour integration 2}), we complete the proof of this lemma. 
\qed 

\begin{lemma}\label{integral estimation along contour 2} 
    We have the following estimation: 
    \begin{equation}\oint_{\Gamma_{t}}\left|e^{-tz}dz\right|\leq\frac{C}{t}\end{equation}
    for some constant $C>0$ depending only on $\kappa$, and 
    \begin{equation}\oint_{\Gamma_{t}}\left|(1-\eta z)^t dz\right|\leq\frac{C}{t}\end{equation}
    for some constant $C>0$ depending only on $\eta$ and $\kappa$. 
\end{lemma} 

\proof 
For the first estimation, we first note that 
\begin{equation}\oint_{\Gamma_\lambda}|e^{-tz}dz|\leq\oint_{\Gamma_{t,1}}|e^{-tz}|\cdot|dz|+\oint_{\Gamma_{t,2}}|e^{-tz}|\cdot|dz|.\end{equation} 
When $z\in\Gamma_{t,1}$, we have 
\begin{equation}\label{oint 2} 
    \oint_{\Gamma_{t,1}}|e^{-tz}|\cdot|dz|=C\int_{-\frac{1}{2t}}^{\kappa^2}e^{-ts}ds\leq\frac{C}{t}(e^{\frac{1}{2}}-e^{-\kappa^2 t})\leq\frac{C}{t}.
\end{equation} 
When $z\in\Gamma_{t,2}$, we have $|z|>\kappa^2$, hence 
\begin{equation}\label{oint 1}
    \oint_{\Gamma_{t,2}}|e^{-tz}|\cdot|dz|\leq e^{-\kappa^2 t}\cdot|\Gamma_{t,2}|\leq Ce^{-\kappa^2 t}.
\end{equation}
Combining (\ref{oint 1}) and (\ref{oint 2}) together, we obtain that 
\begin{equation}\oint_{\Gamma_\lambda}|e^{-tz}dz|\leq\frac{C}{t}.\end{equation} 
The proof for the second estimation is similar. 
\qed

\subsection{Gaussian approximation} 

\begin{definition}\label{vc definition}
    Let $\mathcal{F}$ be a class of measurable functions on a measurable space $(S,\mathcal{S})$. Let $F$ be an envelop of $\mathcal{F}$. In other words, $F$ is a measurable function such that $|f(x)|\leq|F(x)|$ for any $f\in\mathcal{F}$ and $x\in S$. 

    $\mathcal{F}$ is called a VC-type class with envelop $F$ if there exist constants $A,v>0$ such that 
    \begin{equation}\sup_{Q}\mathcal{N}(\mathcal{F},e_Q,\varepsilon\|F\|_{L^2(Q)})\leq\left(\frac{A}{\varepsilon}\right)^v,\quad\forall \varepsilon\in(0,1],\end{equation}
    where the supremum is taken over all finite probability measures $Q$ on $S$, $e_Q(f,g)=\|f-g\|_{L^2(Q)}$, and $mathcal{N}(\mathcal{F},e_Q,\varepsilon\|F\|_{L^2(Q)})$ denotes the $\varepsilon\|F\|_{L^2(Q)}$-covering number of $\mathcal{F}$ under the distance $e_Q$. 
\end{definition}

The following theorem and its proof can be found in \citet{anti-concentration_a, anti-concentration_b, vc-type, anti-concentration_d}; See Lemma A.1 of \citet{anti-concentration_b} for example. 

\begin{theorem}\label{anti-concentration inequality}
    \textnormal{\textbf{(Anti-concentration inequality)}} Let $(S,\mathcal{S},P)$ be a probability space, and let $\mathcal{F}\subset L^2(S, \mathcal{S}, P)$ be a $P$-pre-Gaussian class. Suppose that $G_P$ is a tight Gaussian random variable in $l^\infty(\mathcal{F})$ with zero mean and covariance function $\mathrm{Cov}_P(f,g)=\E(G_P(f)G_P(g))$, $f,g\in\mathcal{F}$. Assume that there exists some constants $\bar\sigma$, $\underline\sigma$ such that $\underline\sigma^2\leq\mathrm{Var}_P(f)\leq\bar\sigma^2$, where $\mathrm{Var}_P(f)=\mathrm{Cov}_P(f,f)$. Then, for any $\varepsilon>0$, we have 
    \begin{equation}\sup_{x\in\R}\P\left(\left|\sup_{f\in\F}G_P f-x\right|\leq\varepsilon\right)\leq C_\sigma\varepsilon\left(\E\left(\sup_{f\in\F}G_P f\right)+\sqrt{1\vee\log(\underline\sigma/\varepsilon)}\right),\end{equation}
    where $C_\sigma\leq C(\bar\sigma/\underline\sigma)^3$ and $C>0$ is a universal constant. 
\end{theorem}

We note that the VC-type class satisfies the requirements of the anti-concentration inequality above (see Lemma 4.1 of \citealt{vc-type} and Lemma 2.1 of \citealt{anti-concentration_b}). We also note that there is a slight difference in the statement of this theorem, where we explicitly determine the dependence of the constant $C_\sigma$ on $\bar\sigma$ and $\underline\sigma$; See the proof in \citet{anti-concentration_d}.

\section{Visualizations of the confidence bands}\label{visualization section}

Figure \ref{vis con} and \ref{vis dis} provide examples of the visualizations of the confidence bands for continuous and discrete kernel gradient flows, respectively, under different settings of training time $t$ and sample size $n$ (given in Section \ref{confidence band exp section}).

\begin{figure}[tbp]
    \centering
    \subfigure[$n=500$, $t=0.5t_{opt}$]{
        \includegraphics[width=4.5cm]{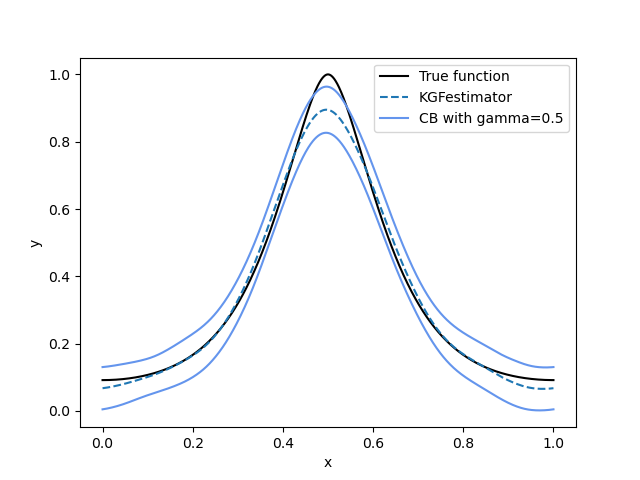} 
    } 
    \subfigure[$n=500$, $t=t_{opt}$]{
        \includegraphics[width=4.5cm]{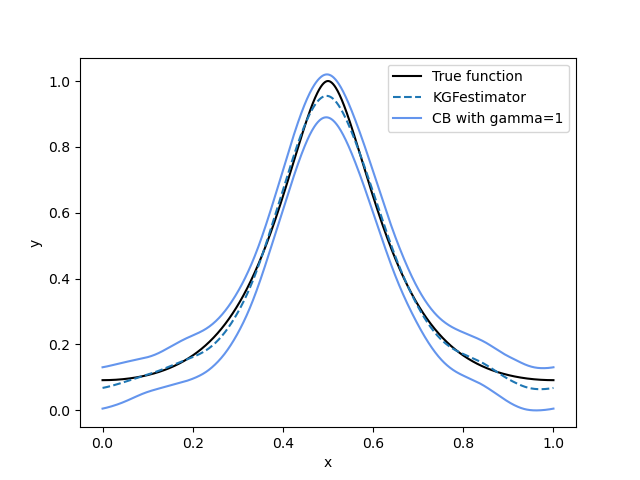} 
    } 
    \subfigure[$n=500$, $t=2t_{opt}$]{
        \includegraphics[width=4.5cm]{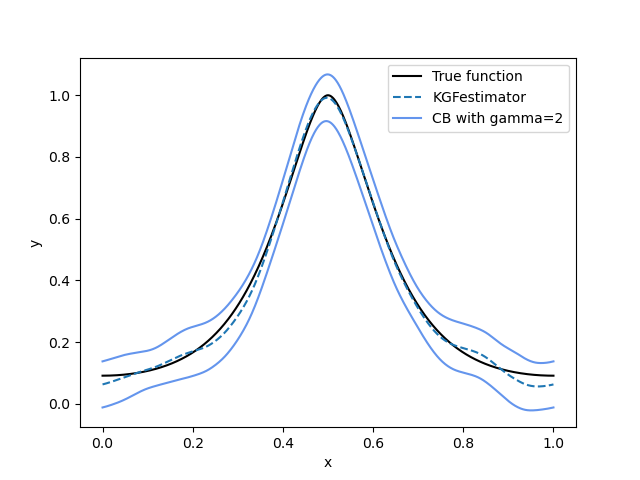} 
    } 
    \subfigure[$n=1000$, $t=0.5t_{opt}$]{
        \includegraphics[width=4.5cm]{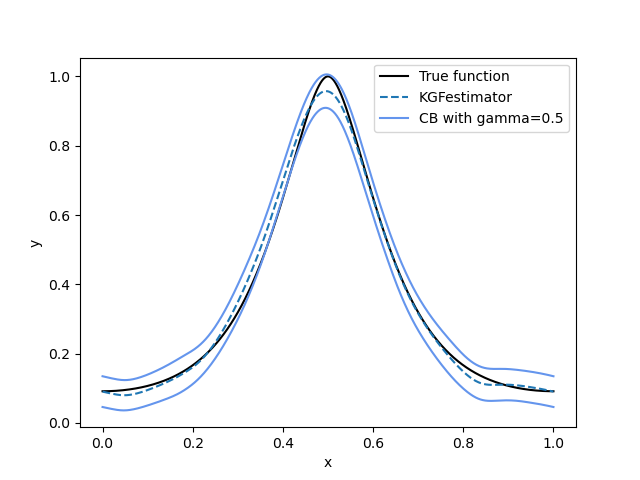} 
    } 
    \subfigure[$n=1000$, $t=t_{opt}$]{
        \includegraphics[width=4.5cm]{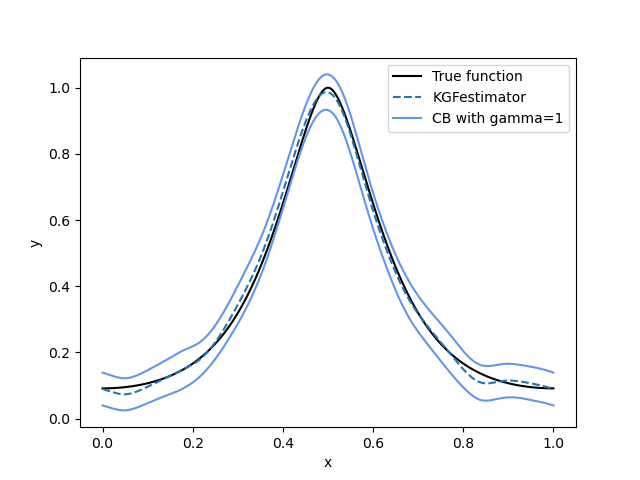} 
    } 
    \subfigure[$n=1000$, $t=2t_{opt}$]{
        \includegraphics[width=4.5cm]{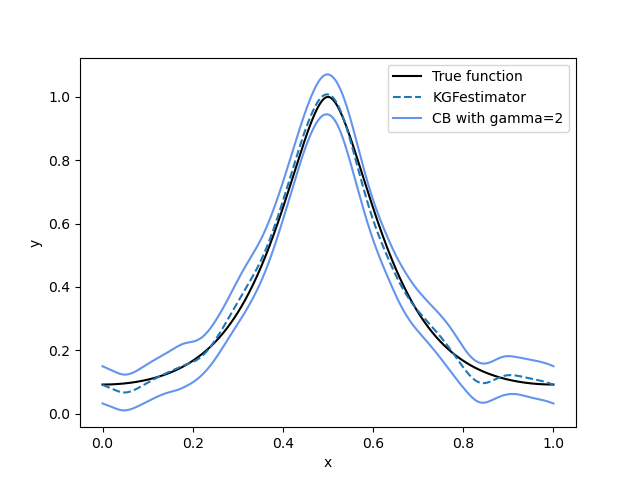} 
    } 
    \subfigure[$n=2000$, $t=0.5t_{opt}$]{
        \includegraphics[width=4.5cm]{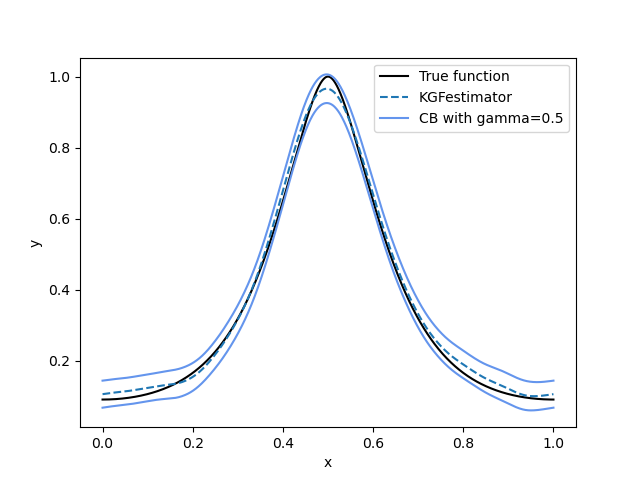} 
    } 
    \subfigure[$n=2000$, $t=t_{opt}$]{
        \includegraphics[width=4.5cm]{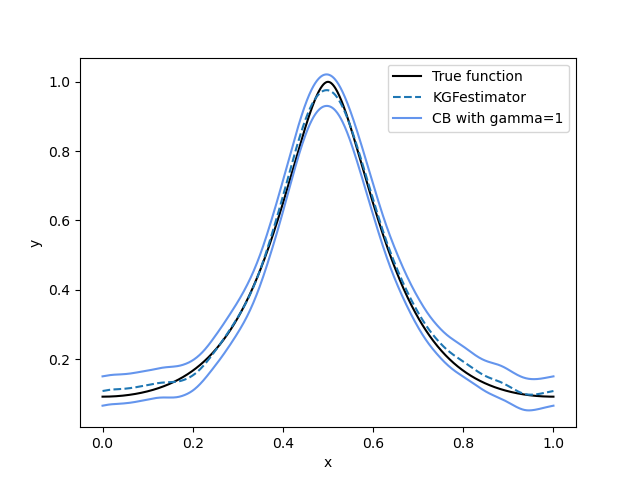} 
    } 
    \subfigure[$n=2000$, $t=2t_{opt}$]{
        \includegraphics[width=4.5cm]{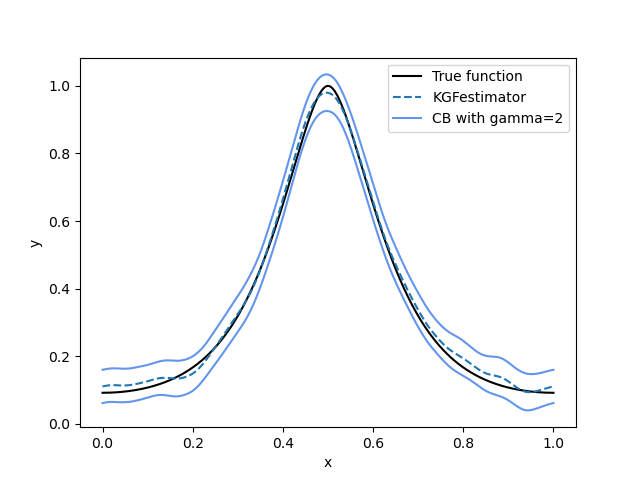} 
    } 
    \subfigure[$n=3000$, $t=0.5t_{opt}$]{
        \includegraphics[width=4.5cm]{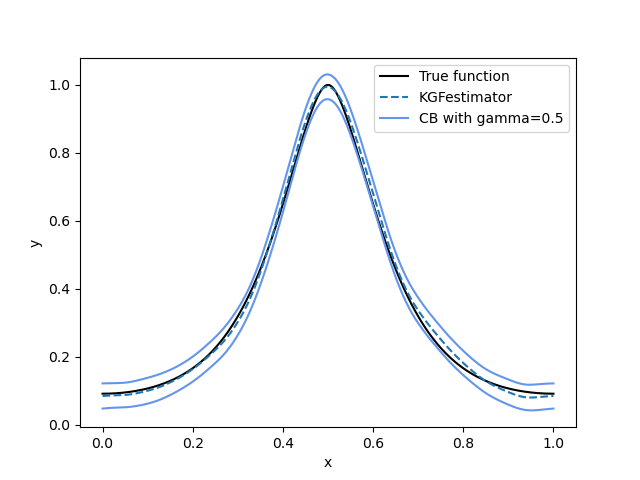} 
    } 
    \subfigure[$n=3000$, $t=t_{opt}$]{
        \includegraphics[width=4.5cm]{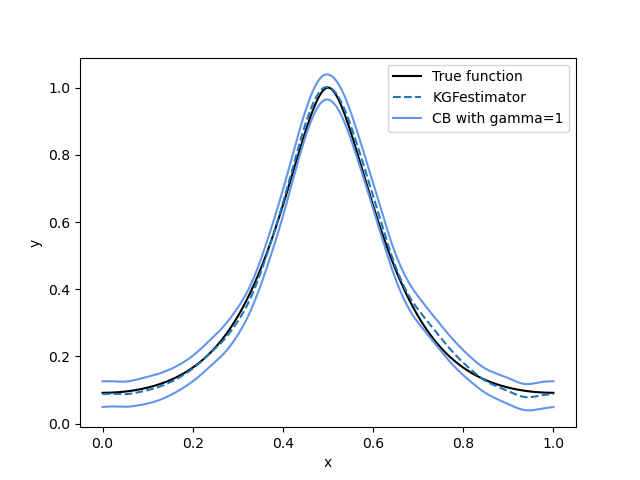} 
    } 
    \subfigure[$n=3000$, $t=2t_{opt}$]{
        \includegraphics[width=4.5cm]{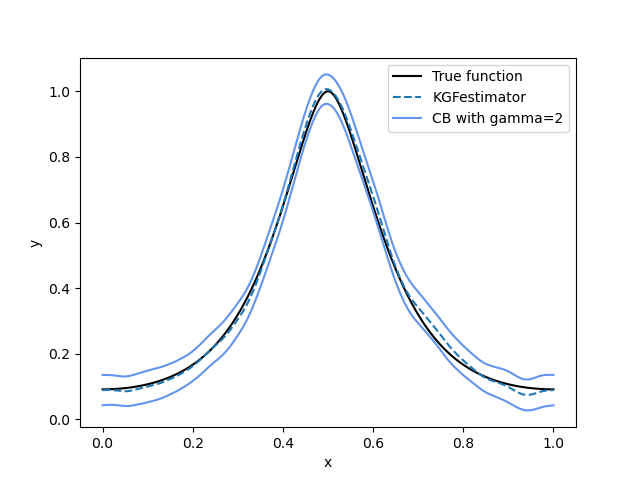} 
    } 
    \caption{Visualizations of confidence bands for continuous kernel gradient flow estimators}\label{vis con} 
\end{figure}

\begin{figure}[tbp]
    \centering
    \subfigure[$n=500$, $t=0.5t_{opt}$]{
        \includegraphics[width=4.5cm]{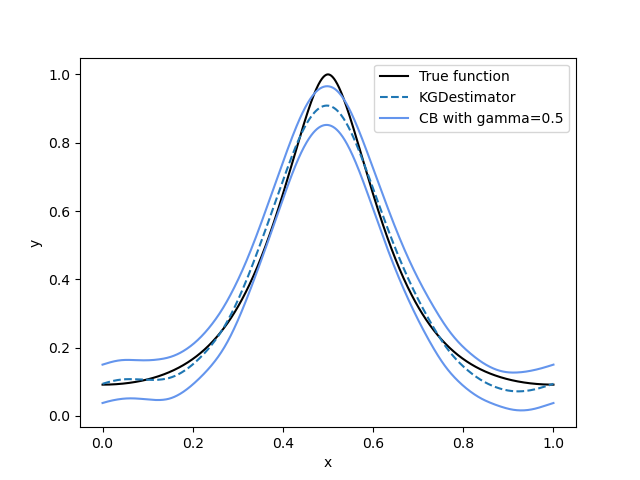} 
    } 
    \subfigure[$n=500$, $t=t_{opt}$]{
        \includegraphics[width=4.5cm]{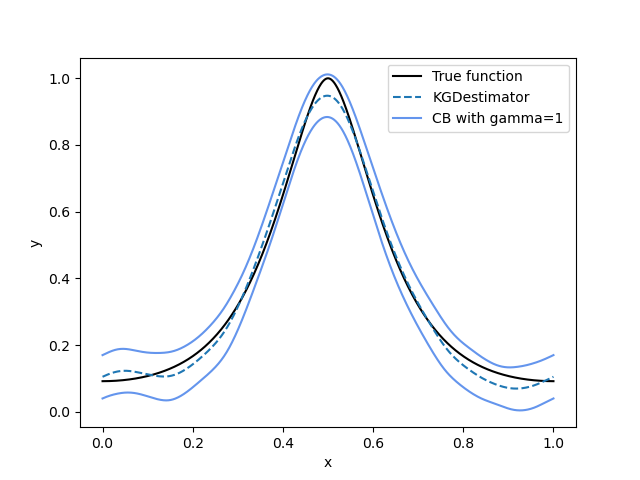} 
    } 
    \subfigure[$n=500$, $t=2t_{opt}$]{
        \includegraphics[width=4.5cm]{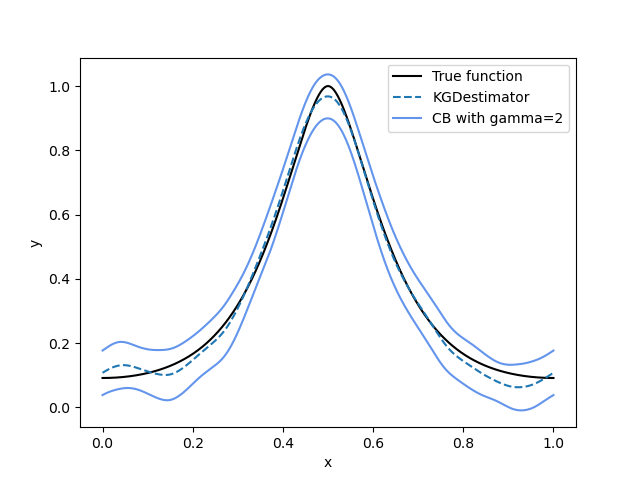} 
    } 
    \subfigure[$n=1000$, $t=0.5t_{opt}$]{
        \includegraphics[width=4.5cm]{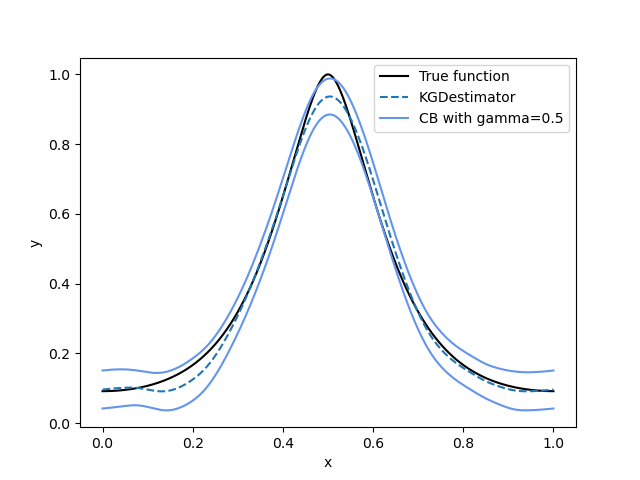} 
    } 
    \subfigure[$n=1000$, $t=t_{opt}$]{
        \includegraphics[width=4.5cm]{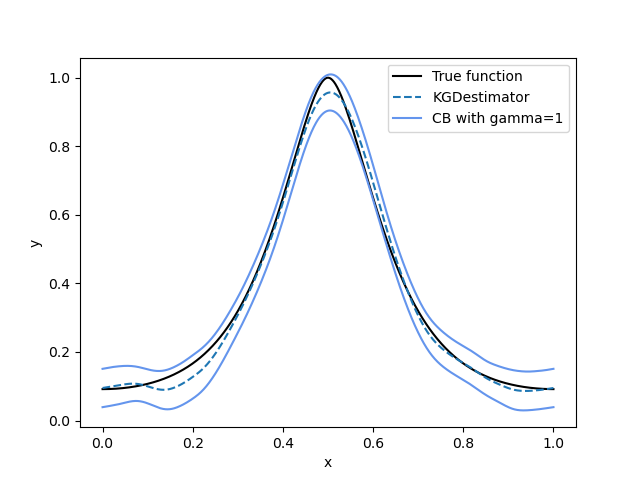} 
    } 
    \subfigure[$n=1000$, $t=2t_{opt}$]{
        \includegraphics[width=4.5cm]{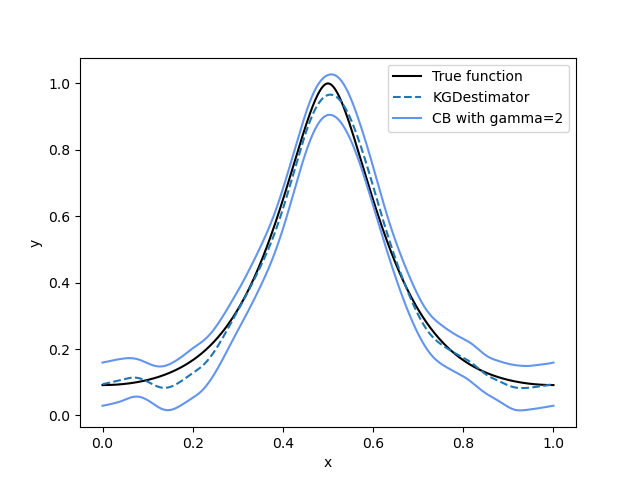} 
    } 
    \subfigure[$n=2000$, $t=0.5t_{opt}$]{
        \includegraphics[width=4.5cm]{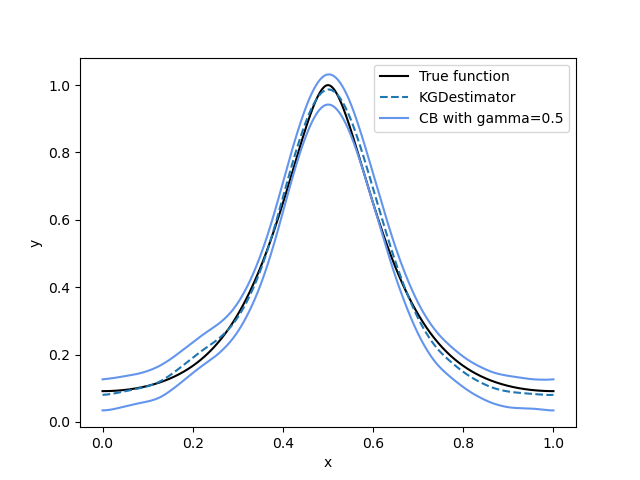} 
    } 
    \subfigure[$n=2000$, $t=t_{opt}$]{
        \includegraphics[width=4.5cm]{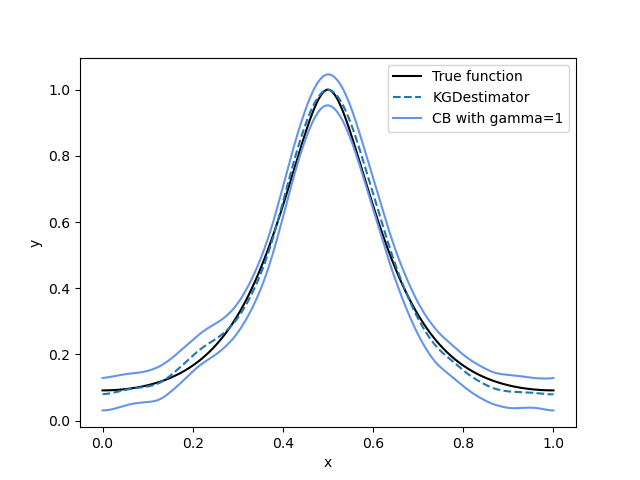} 
    } 
    \subfigure[$n=2000$, $t=2t_{opt}$]{
        \includegraphics[width=4.5cm]{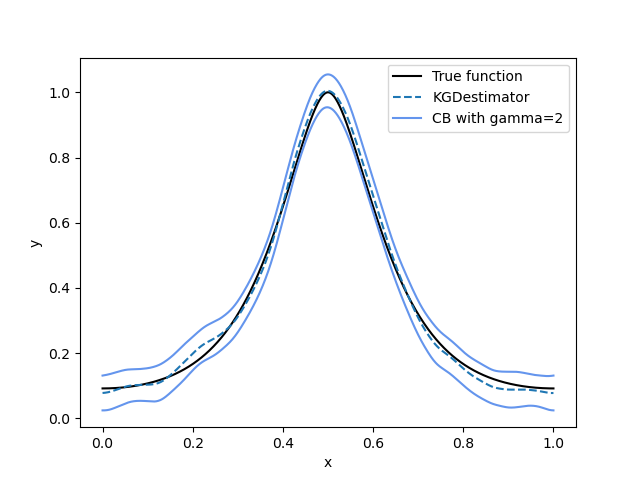} 
    } 
    \subfigure[$n=3000$, $t=0.5t_{opt}$]{
        \includegraphics[width=4.5cm]{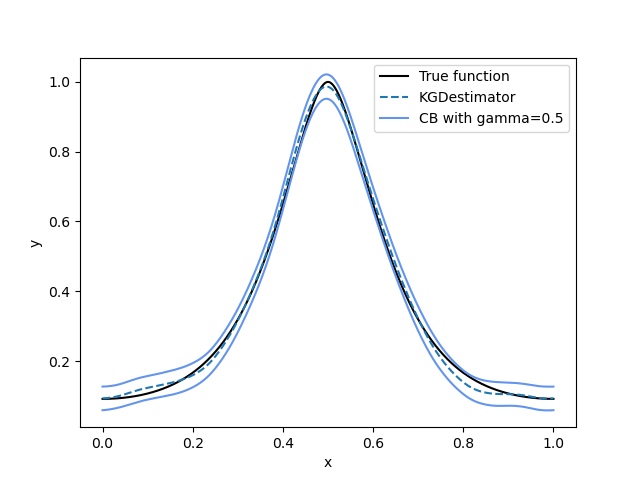} 
    } 
    \subfigure[$n=3000$, $t=t_{opt}$]{
        \includegraphics[width=4.5cm]{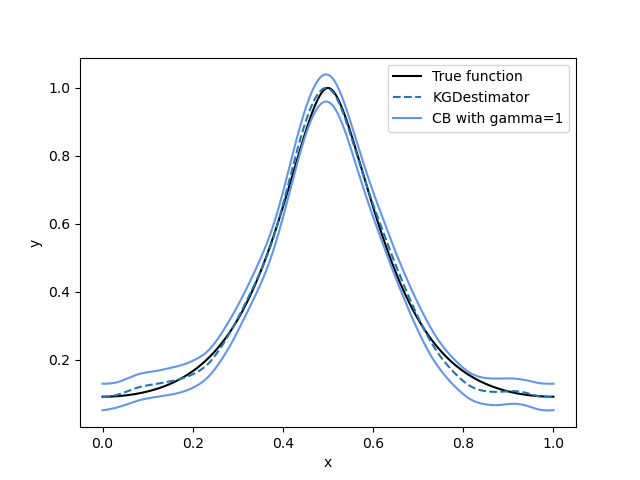} 
    } 
    \subfigure[$n=3000$, $t=2t_{opt}$]{
        \includegraphics[width=4.5cm]{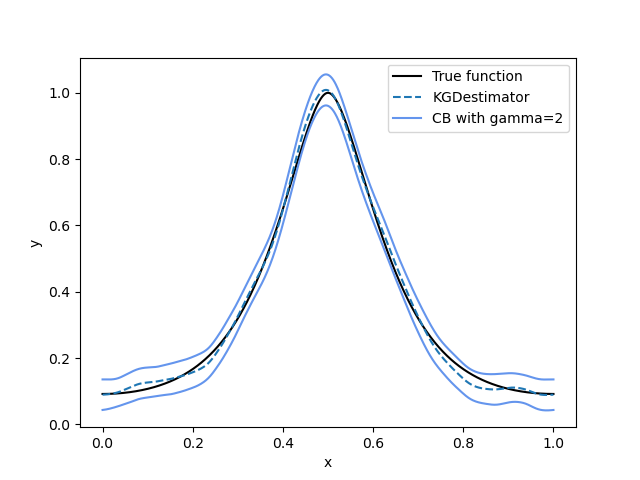} 
    } 
    \caption{Visualizations of confidence bands for discrete kernel gradient flow estimators}\label{vis dis} 
\end{figure}

\newpage

\vskip 0.2in
\bibliography{references.bib}

\end{document}